\documentclass[10pt,reqno]{amsart}

\usepackage[utf8]{inputenc}
\usepackage[english]{babel}
\usepackage{microtype}
\usepackage{mathrsfs}
\usepackage{amsthm, thmtools, mathtools}
\usepackage{amsmath, centernot}
\usepackage{amssymb}
\usepackage{indentfirst,csquotes}
\usepackage{amscd}
\usepackage{mathtools}
\usepackage{graphicx} 
\usepackage{wasysym}
\usepackage{braket}
\usepackage{hyperref}
\usepackage{appendix}
\usepackage{dsfont}
\usepackage{enumitem}
\usepackage{framed}
\usepackage{caption}
\usepackage{subfig}
\usepackage[all]{xy}
\usepackage{xfrac}
\usepackage{hyperref}
\usepackage{color}
\usepackage{bbm}
\usepackage{stmaryrd}
\usepackage{fancyhdr}
\usepackage{tikz-cd}
\usepackage{graphicx}
\usepackage{nomencl}
\usepackage{ esint }
\usepackage{tikz}
\usepackage{multicol}
\usepackage{eucal}
\usepackage{xcolor}
\usepackage[normalem]{ulem}

\numberwithin{equation}{section}

\allowdisplaybreaks

\newtheorem{teorema}{Theorem}[section]
\newtheorem{prop}[teorema]{Proposition}
\newtheorem{co}[teorema]{Corollary}
\newtheorem{lemma}[teorema]{Lemma}

\newtheorem{df}[teorema]{Definition}

\theoremstyle{definition}

\newtheorem{oss}[teorema]{Remark}

\newtheorem{es}[teorema]{Example}

\newcommand{\xx}{\mathrm{x}}

\newcommand{\R}{\mathbb{R}}

\newcommand{\N}{\mathbb{N}}

\newcommand{\PP}{\mathcal{P}}

\newcommand{\CE}{\operatorname{CE}}
\newcommand{\HJ}{\operatorname{HJ}}

\newcommand{\Diss}{{\psi}}
\newcommand{\Dpartial}{{\partial^*}}
\newcommand{\calJ}{{\mathcal J}}
\newcommand{\ucalJ}{{\mathcal J}}
\newcommand{\rS}{{S^-}}
\newcommand{\rSR}{{S^-_R}}
\newcommand{\rSk}{{S^-_k}}

\title[A direct method for doubly nonlinear equations]{A direct method for doubly nonlinear equations via convexification in spaces of measures and duality}
\author[A. Pinzi, F. Riva, G. Savar\'e]{Alessandro Pinzi, Filippo Riva and Giuseppe Savar\'e}
\date{}

\begin{document}

\begin{abstract}
    Existence of solutions to doubly nonlinear equations in reflexive Banach spaces is established by resorting to a global-in-time variational approach inspired by De Giorgi's principle, which characterizes the associated flows as null-minimizers of a suitable energy-dissipation functional defined on trajectories. In contrast to the celebrated minimizing movements scheme, the proposed strategy does not rely on any time-discretization or iterative constructions. Instead, it provides a direct method based on the relaxation of the problem in spaces of measures, constrained by the continuity equation: in this procedure, no gap is introduced due to the Ambrosio's superposition principle.
    
    Within this weak convex framework, the validity of the null-minimization property is recovered through two further steps. First, a careful application of the Von Neumann minimax theorem yields an identification of the dual problem as a supremum over the set of smooth and bounded cylinder functions, solving an Hamilton--Jacobi-type inequality. Secondly, a suitable \lq\lq backward boundedness\rq\rq  property of solutions to such Hamilton-Jacobi system gives a proper bound of the dual problem, ensuring that the minimum value of the original functional is actually zero. 
    
    The proposed strategy naturally extends to non-autonomous equations, encompassing time- and space-dependent dissipation potentials and time-dependent potential energies.
\end{abstract}

\maketitle

	{\small
		\par
	}

	\pagenumbering{arabic}

\medskip

\tableofcontents

\section*{Introduction}
We present a novel duality approach to show existence of absolutely continuous curves solving the \emph{doubly nonlinear differential inclusion}
\begin{equation}\label{eq:DNEintro}
    \partial \Diss(\dot {\xx}(t))+\partial \varphi(\xx(t))\ni 0,\quad\text{for a.e. }t\in(0,T),
\end{equation}
to be meant in the dual of a (separable and reflexive) Banach space $X$. Here, the dot stands for time differentiation, while the symbol $\partial$ denotes the Frech\'et subdifferential. As costumary in the variational community, we refer to the function $\Diss\colon X\to (-\infty,+\infty]$ as (primal) \emph{dissipation potential} and to $\varphi\colon X\to (-\infty,+\infty]$ as (potential) \emph{energy}. We require $\Diss$ to be convex and lower semicontinuous, and $\varphi$ to be lower semicontinuous and coercive, but no kind of convexity assumptions are asked for $\varphi$ (even if the case of $\lambda$-convex functionals 
covers an important class of examples). Observe that, in a Hilbertian setting, the choice $\Diss(v)=\frac 12 |v|^2$ leads to the theory of \emph{gradient flows}.

Doubly nonlinear equations arise in models of rate-dependent dissipative systems, including plasticity, damage mechanics, viscous flows, and phase transitions; see, e.g., \cite{ColliVisintin1990,MRSNonsmooth,MielkRoubbook,Visintin1996}. The nonlinearity in both the velocity (through the dissipation potential $\Diss$) and the state (through the energy $\varphi$) captures complex dissipation mechanisms beyond the quadratic setting of gradient flows. 
One of the most successful 
strategies to solve 
\eqref{eq:DNEintro} 
relies on the variational Minimizing Movement scheme \cite{DeGiorgi}:
solutions originate as limits
of a discretization method,
which involves a family of recursive minimizations 
of the energy $\varphi$ singularly perturbed by the time-step rescaled dissipation potential $\Diss$.
When the dissipation additionally depends on the current state---as occurs in models with position-dependent friction or material hardening---standard time-discretization methods face significant complications, since each implicit step becomes a nonlinear fixed-point problem coupling state and velocity.

The proposed strategy, which---as we show at the end of the paper---naturally encompasses also state-dependent dissipations, stems from a variational principle due to De Giorgi \cite{DeGiorgiMarinoTosques}, related to the theory of curves of maximal slopes in metric spaces. Under suitable regularity conditions, it states that solving \eqref{eq:DNEintro} is equivalent to be a \emph{null-minimizer}, i.e. the value of the minimum is zero, of the energy-dissipation functional
\begin{equation}\label{eq:DeGiorgi} 
    \mathcal J(T,\xx)=\varphi(\xx(T))+\int_0^T \Big(\Diss(\dot{\xx}(\tau))+S(\xx(\tau))\Big)\,d\tau-\varphi(x_0),
\end{equation}
among all absolutely continuous curves $\xx\colon [0,T]\to X$ sharing the same initial datum $\xx(0)=x_0\in X$. In \eqref{eq:DeGiorgi}, the symbol $S$ denotes the \emph{slope} functional 
\begin{equation}\label{eq:slope}
    S(x):=\inf\limits_{z\in\partial\varphi(x)}\Diss^*(-z),
\end{equation}
where $\Diss^*$ is the Fenchel (or convex) conjugate of $\Diss$, also called dual dissipation potential. The slope $S(x)$ measures the maximal instantaneous rate of energy dissipation available at state $x$, expressed in terms of the conjugate dissipation $\Diss^*$. In the gradient flow case $\Diss(v) = \frac{1}{2}|v|^2$, one has $\Diss^*(z) = \frac{1}{2}|z|^2$, and the slope reduces to
\[
S(x) = \frac{1}{2}|\mathrm D\varphi(x)|^2,
\]
where $\mathrm D\varphi(x) \in \partial\varphi(x)$ denotes the element of minimal norm in the subdifferential; cf.\ \cite[Chapter~1]{ambrosio2005gradient}.

De Giorgi's principle is based on the following observation: given an absolutely continuous curve $\xx\colon [0,T]\to X$ and a measurable selection $\mathrm{z}$ of $\partial\varphi(\xx(\cdot))$, Fenchel's inequality (see \eqref{Fenchelprop}) yields
\begin{equation}\label{eq:cr}
    \Diss(\dot {\xx}(t))+\Diss^*(-\mathrm{z}(t))\ge -\langle \mathrm{z}(t),\dot {\xx}(t)\rangle=-\frac{d}{dt}\varphi(\xx(t)),\quad\text{for a.e. }t\in(0,T),
\end{equation}
with equality if and only if $-\mathrm{z}(t)\in \partial \Diss(\dot {\xx}(t))$, namely if \eqref{eq:DNEintro} is satisfied. The derivation above relies on a chain rule for the map $t \mapsto \varphi(\xx(t))$, which is nontrivial for highly nonconvex  energies. We make this precise in assumption \ref{CR} of Proposition~\ref{prop:DeGiorgi}; sufficient conditions are discussed in \cite[Section~2.2]{MRSNonsmooth}.

Although showing that minimizers of $\mathcal J(T,\cdot)$ exist is not difficult under suitable general assumptions, proving that the minimum is exactly zero poses real challenges (of course, without knowing a priori the existence of a solution). The same issue appears in the similar Brezis-Ekeland principle \cite{BrezEk2,BrezEk}, which holds for Hilbertian gradient flows with convex energies $\varphi$. See also \cite{Stef,Visintin} for an extension to \emph{convex} doubly nonlinear equations and equations with monotone operators. According to such principle, an absolutely continuous curve $\bar {\xx}\colon [0,T]\to X$ exiting from $x_0$ is a gradient flow if and only if it is a \emph{null-minimizer} of the functional
\begin{equation*}
    \mathcal{B}\!\mathcal{E}(T,\xx)=\frac 12|\xx(T)|^2+\int_0^T\Big(\varphi(\xx(\tau))+\varphi^*(-\dot {\xx}(\tau))\Big) \,d\tau-\frac 12|x_0|^2.
\end{equation*}
Again, the observation leading to Brezis-Ekeland principle is the inequality
\begin{equation*}
    \varphi(\xx(t))+\varphi^*(-\dot {\xx}(t))\ge -(\xx(t),\dot {\xx}(t))=-\frac{d}{dt}\frac 12|\xx(t)|^2, \quad\text{for a.e. }t\in(0,T),
\end{equation*}
which saturates if and only if $-\dot {\xx}(t)\in \partial\varphi(\xx(t))$, i.e. if $\xx$ is a gradient flow.

Within this framework the task of proving that the minimum of $\mathcal{B}\!\mathcal{E}(T,\cdot)$ is in fact zero has been fully accomplished in \cite{GhoussTzou} by considering a suitable dual formulation of the problem and exploiting its \lq\lq self-dual\rq\rq structure. We stress that convexity of the energy $\varphi$ and the quadratic character of the dissipation functional play a crucial role in their argument, specifically in the application of the tool of the so-called self-dual Lagrangians (besides being necessary for the use of the Brezis-Ekeland functional $\mathcal{B}\!\mathcal{E}(T,\cdot)$ itself). The approach has been then extended to gradient flows of more general $\lambda$-convex energies in \cite{GhoussMcCann}. Still in the convex framework, we also mention the prior results of \cite{Auch,Roub} where the existence of null minimizers have been proved under additional controlled-growth assumptions. 

Our approach differs fundamentally in scope, since  the De~Giorgi's functional \eqref{eq:DeGiorgi} imposes no convexity on $\varphi$ and 
the presence of the slope $S$ prevents any direct convex or self-dual structure. The aim of the paper is to develop a new duality argument for the De Giorgi's functional \eqref{eq:DeGiorgi}, leading to existence of \emph{null-minimizers} of $\mathcal J(T,\cdot)$ (and so to solutions to the doubly nonlinear equation \eqref{eq:DNEintro}) without any kind of convexity assumption on $\varphi$. We will recover the null-minimization property by relaxing the problem to spaces of measures and invoking Hamilton-Jacobi theory.

We also stress that this argument provides a direct method to solve \eqref{eq:DNEintro}, in contrast to the existing iterative procedures based on time-discretization such as the aforementioned Minimizing Movements scheme. Our approach produces solutions directly as minimizers of a global-in-time functional and  yields a variational characterization potentially useful for qualitative analysis. Furthermore, our technique encompasses dissipation potentials possibly taking also the value $+\infty$ (with $0$ in the interior of their domain), thus including evolutions with velocity constraints. More significantly, the method extends naturally to non-autonomous problems with state-dependent dissipation $\Diss(t, x, v)$ and time-dependent energy $\varphi(t,x)$---a setting where time-discretization schemes encounter technical difficulties, as each implicit step couples the unknown state and velocity in a nonlinear fashion. We refer to Remark~\ref{rmk:nonautonomous} for a simple example and to Section~\ref{sec:nonautonomous} for the precise framework. 

In order to keep the presentation light and clear, we now proceed explaining our method in a smooth finite-dimensional setting, postponing the discussion regarding the main technical difficulties we have to face in general Banach spaces. It essentially relies on three ingredients: the application of a superposition principle in order to convexify and relax the De Giorgi's functional to spaces of measures, the use of a minimax theorem allowing to pass to a dual formulation, and the validation of a comparison principle for the Hamilton-Jacobi inequality arising as a constraint in the dual problem. We believe that this original methodology represents a robust and versatile approach, sufficiently flexible to be adapted to a wide range of other problems and frameworks, as gradient flows in a metric space. We plan to address this case in a forthcoming paper.

\subsection*{Outline of the strategy in Euclidean spaces}
Let us consider a convex, smooth, superlinear dissipation potential $\Diss\colon \R^n\to [0,+\infty)$ satisfying $\Diss(0)=0$ and a smooth energy $\varphi\colon \R^n\to [0,+\infty)$ such that 
\begin{equation}\label{eq:coerc}
    \lim\limits_{|x|\to +\infty}\varphi(x)=\lim\limits_{|x|\to +\infty}|\mathrm D\varphi(x)|=+\infty.
\end{equation}
In this setting, the Cauchy problem under investigation is the ordinary differential equation
\begin{equation}\label{eq:finitedim}
     \begin{cases}
        \mathrm D \Diss(\dot {\xx}(t))+\mathrm D\varphi(\xx(t))= 0,\\
        \xx(0)=x_0\in \R^n,
    \end{cases}
\end{equation}
and the corresponding De Giorgi's functional \eqref{eq:DeGiorgi} takes the form
\begin{equation*}
    \mathcal J(T,\xx)=\varphi(\xx(T))+\int_0^T \Big(\Diss(\dot {\xx}(\tau))+\Diss^*(-\mathrm D\varphi(\xx(\tau)))\Big)\,d\tau-\varphi(x_0).
\end{equation*}

Recalling De Giorgi's principle, the result we obtain can be stated as follows.

\begin{teorema}\label{thm:intro}
    Under the previous assumptions, for all initial datum $x_0\in \R^n$ there holds
    \begin{equation*}
          \min\left\{{\mathcal{J}}(T,\xx):\, \xx\in AC([0,T];\R^n),\, \xx(0)=x_0\right\}= 0.
      \end{equation*}
      In particular, any minimizer solves \eqref{eq:finitedim}.
\end{teorema}

The existence of the minimum is ensured by standard arguments, while the chain-rule in \eqref{eq:cr}, which clearly holds in the current smooth context, implies that it is nonnegative. Thus, the only nontrivial fact and the core of our argument is the validity of the opposite inequality
\begin{equation}\label{eq:minintro}
          \min\left\{{\mathcal{J}}(T,\xx):\, \xx\in AC([0,T];\R^n),\, \xx(0)=x_0\right\}\le 0.
      \end{equation}
Its proof follows three main steps.\medskip

\textbf{Step 1.} \textbf{Convexification and relaxation of the functional in spaces of measures.}
    In order to make the problem convex, we relax $\mathcal J(T,\cdot)$ by passing from curves to measures. We associate with the curve $\xx$ a nonnegative measure $\mu$ over the space-time $X_T:=[0,T]\times \R^n$, with the final position $\xx(T)$ a nonnegative measure $m$ over $\R^n$, and with the velocity $\dot {\xx}$ a vector-valued measure $\nu$ over $X_T$. The link between the three measures $m$, $\mu$, $\nu$ and the initial datum is encoded in the \emph{continuity equation}
\begin{equation}\label{eq:conteqintro}
        \begin{cases}    \partial_t\mu+\operatorname{div}\nu=0,\qquad \text{in }X_T,\\
            \mu(T,\cdot)=m,
            \\
            \mu(0,\cdot) = \delta_{x_0},
        \end{cases}
    \end{equation}
    meant in duality with the class of smooth and bounded functions $C_b^1(X_T)$. The choice of such class is useful for future considerations. Anyway, we point out that the formulation in duality with such functions is equivalent to the usual interpretation in the sense of distributions, whenever $\nu$ is absolutely continuous with respect to $\mu$ (as it will be).

    The relaxed problem consists in minimizing the functional
    \begin{equation}\label{eq:E}
        \mathcal E(m,\mu,\nu):=\int_{\R^n}\varphi(x)\, dm(x)+\int_{X_T}\left(\Diss\left(\frac{d\nu}{d\mu}(t,x)\right)+\Diss^*(-\mathrm D\varphi(x))\right)\, d\mu(t,x)-\varphi(x_0),
    \end{equation}
    among all the triples $(m,\mu,\nu)$ subject to the constraint \eqref{eq:conteqintro}. In \eqref{eq:E}, we are tacitly imposing that the functional is $+\infty$ if $\nu$ is not absolutely continuous with respect to $\mu$, i.e. when the density $\frac{d\nu}{d\mu}$ is not defined. This problem is now convex: indeed, the constraint is convex and the functional $\mathcal{E}$ is affine, except for the term $\int_{X_T}\Diss\left( \frac{d\nu}{d\mu} \right) d\mu$, usually called \emph{action}, which however turns out to be convex as shown in \cite[Chapter 2]{ambrosio2000functions}.
    
    In order to prove that the original and the relaxed problem are equivalent, so that their minimum values coincide, we argue as follows. It is immediate to check that, given an absolutely continuous curve $\xx$ with $\xx(0)=x_0$, the triple given by 
    \begin{equation*}
    m=\delta_{\xx(T)},\quad \mu= \mathcal{L}^1\otimes\delta_{\xx(t)}, \quad \nu=\dot {\xx}(t)\mathcal{L}^1\otimes\delta_{\xx(t)}
    \end{equation*}
    satisfies \eqref{eq:conteqintro} and $\mathcal J(T,\xx)=\mathcal E(m,\mu,\nu)$, whence the original problem is larger than the relaxed one. The opposite inequality (see \eqref{eq:compare}), which is the key one that actually allows us to prove \eqref{eq:minintro} due to \eqref{eq:minweak} below, instead is more involved, and requires the use of the so-called \emph{superposition principle}, see \cite[Theorem 8.2.1]{ambrosio2005gradient}. Consider a triple $(m,\mu,\nu)$ solving \eqref{eq:conteqintro}: it is well-known, see \cite[Lemma 8.1.2]{ambrosio2005gradient} that this forces the first marginal of $\mu$ to be the Lebesgue measure on $[0,T]$, providing a continuous curve of probability measures $t\mapsto \mu_t \in \PP(\R^n)$, with $\mu_0 = \delta_{x_0}$ and $\mu_T = m$. Moreover, if $\nu$ is absolutely continuous with respect to $\mu$, the system \eqref{eq:conteqintro} can be rewritten in the more common way
    \[ \partial_t \mu_t + \operatorname{div}(v_t \mu_t) = 0 \qquad \text{where} \qquad v(t,x) = \frac{d\nu}{d\mu}(t,x).\] The superposition principle now ensures the existence of a probability measure on the space of continuous curves, $ \lambda \in \PP(C([0,T],\R^n))$, with assigned time-marginals $(\mu_t)_{t\in[0,T]}$ and concentrated on absolutely continuous curves satisfying $\dot{\xx}(t) = v(t,\xx(t))$. This finally yields
    \begin{equation}\label{eq:compare}
        \mathcal{E}(m,\mu,\nu) = \int_{AC([0,T],\R^n)} \mathcal{J}(T,\xx)\, d\lambda(\xx) \geq \min\{ \mathcal{J}(T,\xx):\, \xx\in AC([0,T];\R^n),\, \xx(0)=x_0\}.
    \end{equation}
    Thus, in order to prove \eqref{eq:minintro}, it is enough to show
    \begin{equation}\label{eq:minweak}
        \min\left\{ \mathcal{E}(m,\mu,\nu) \ : \ \eqref{eq:conteqintro} \text{ is satisfied} \right\} \leq 0.
    \end{equation}
      
    \textbf{Step 2.} \textbf{Use of a minimax theorem and the dual problem.}
    We now observe that the constraint introduced by the continuity equation \eqref{eq:conteqintro} can be explicitly formulated as
    \begin{equation}\label{eq:constraintintro}
    \mathcal C(\xi, m,\mu,\nu)=0,\qquad \text{for all $\xi\in C^1_b(X_T)$},
\end{equation}   
where
\begin{equation*}
    \mathcal C(\xi, m,\mu,\nu):= \int_{X_T} \partial_t \xi\, d\mu + \int_{X_T}   \mathrm D_x\xi\cdot d\nu  - \int_{\R^n}  \xi(T)\, d m +  \xi(0,x_0).
\end{equation*} 
    Thus, introducing the Lagrangian $\mathcal L(\xi,m,\mu,\nu)=\mathcal E(m,\mu,\nu)+\mathcal C(\xi,m,\mu,\nu)$, the linearity of $\mathcal C$ with respect to $\xi$ allows to rewrite the relaxed (constrained) optimization problem as a saddle point problem
    \begin{equation*}
        \min\limits_{(m,\mu,\nu)} \sup\limits_\xi \mathcal L(\xi,m,\mu,\nu).
    \end{equation*}
    As previously mentioned, the functional $\mathcal L$ is convex in the variables $(m,\mu,\nu)$ and clearly concave in $\xi$. This structure permits to interchange the order of $\inf$ and $\sup$, as stated by the celebrated Von Neumann minimax theorem \cite[Theorem 3.1]{Sim}. The following recent general version, essential for our purposes, is taken from \cite{OrPorSav}.
    \begin{teorema}[Von Neumann]\label{thm:VonNeumann}
    Let $A$ and $B$ be convex subspaces of some vector spaces, and let $B$ be endowed with some Hausdorff topology. Assume that $\mathcal L\colon A\times B\to \R$ satisfies
    \begin{itemize}
        \item[$\bullet$] $a\mapsto \mathcal L(a,b)$ is concave in $A$ for all $b\in B$;
        \item[$\bullet$] $b\mapsto \mathcal L(a,b)$ is convex and lower semicontinuous in $B$ for all $a\in A$.
    \end{itemize}
    If there exist $\overline C>\sup\limits_{a\in A}\inf\limits_{b\in B} \mathcal L(a,b)$ and $\overline a\in A$ such that
    \begin{equation*}
        \{b\in B:\, \mathcal L(\overline a,b)\le \overline C\}\quad \text{ is nonempty and compact},
    \end{equation*}
    then one has
    \begin{equation*}
        \min\limits_{b\in B} \sup\limits_{a\in A}\mathcal L(a,b)=\sup\limits_{a\in A}\inf\limits_{b\in B} \mathcal L(a,b).
    \end{equation*}
\end{teorema}
In our setting the obvious option is $A=C^1_b(\R^n)$ and $B$ corresponding to the set of triples $(m,\mu,\nu)$ such that the energy \eqref{eq:E} is finite. We also stress how the choice of a suitable topology on $B$, granting both lower semicontinuity and compactness, is pivotal for a correct application of the above minimax result. Due the expression of $\mathcal E$, the natural topology to consider is the narrow topology in each component. 

Let us briefly check that we fall in the Theorem's assumptions. For a given $\xi\in C^1_b(X_T)$ the functional $\mathcal C(\xi,\cdot,\cdot,\cdot)$ is even continuous, while the only term of $\mathcal E$ for which lower semicontinuity is nontrivial is the action $$\int_{X_T}\Diss\left(\frac{d\nu}{d\mu}(t,x)\right) \, d\mu(t,x),$$
which turns out to be lower semicontinuous by \cite[Chapter 2]{ambrosio2000functions}. Now, setting $\overline{\xi}(t,x):= (t-T-1)$, and considering the set
$${\Sigma}_{\overline C} := \{(m,\mu,\nu) \ : \ \mathcal{L}(\overline{\xi},m,\mu,\nu) \leq\overline{C}\},$$
the coercivity of $\varphi$ and $|\mathrm D \varphi|$ assumed in \eqref{eq:coerc} yield, respectively, compactness of $m$ and $\mu$ by Prokhorov theorem \cite{dellacherie1978probabilities} (note that sublevels of $\Diss^*$ are bounded by the continuity of $\Diss$ in $0$, thus $\Diss^*(-\mathrm D\varphi)$ is coercive as well). Regarding $\nu$, both equi-boundedness and equi-tightness are consequences of the superlinearity of $\Diss$ together with the previously obtained compactness of $\mu$.

    After switching $\inf$ and $\sup$ by applying Von Neumann's Theorem, simple computations yield
    \begin{align*}
        &\sup\limits_\xi \inf\limits_{(m,\mu,\nu)}\mathcal L(\xi,m,\mu,\nu)\\
        =& \sup\limits_\xi\Bigg( \xi(0,x_0)-\varphi(x_0)+\inf\limits_m\int_{\R^n}(\varphi-\xi(T))\, d m+\inf\limits_\mu\Bigg(\int_{X_T}\partial_t\xi+\Diss^*(-\mathrm D\varphi)\,d\mu\\
        &\qquad+\inf\limits_{|\nu|\ll\mu}\int_{X_T}\left(\Diss\left(\frac{d\nu}{d\mu} \right)+\mathrm D_x\xi\cdot \frac{d\nu}{d\mu}\right) \, d \mu\Bigg)\Bigg)\\
        =& \sup\limits_\xi\Bigg( \xi(0,x_0)-\varphi(x_0)+\inf\limits_m\int_{\R^n}(\varphi-\xi(T))\, d m+\inf\limits_\mu\int_{X_T}(\partial_t\xi+\Diss^*(-\mathrm D\varphi)-\Diss^*(-\mathrm D_x\xi))\, d \mu\Bigg).
    \end{align*}
    Using the general fact that 
    \begin{equation}\label{eq:generalintro}
        \inf\limits_{\rho\in \mathcal M_+(Y)}\int_Y f d\rho=\begin{cases}
            0,&\text{if }f(y)\ge 0\text{ for all }y\in Y,\\
            -\infty, &\text{otherwise,}
        \end{cases}
    \end{equation}
    we finally infer
    \[\min\left\{ \mathcal{E}(m,\mu,\nu) \ : \ \eqref{eq:conteqintro} \text{ is satisfied} \right\} =\sup\limits_\xi\left\{\xi(0,x_0)-\varphi(x_0):\, \xi\in \HJ\right\},\]    
    where the dual constraint $\xi\in \HJ$ means that the function $\xi$ solves the following Hamilton-Jacobi-type inequality for all $(t,x)\in X_T$: 
    \begin{equation}\label{eq:HJintro}
        \begin{cases}
            - \partial_t\xi(t,x) + \Diss^*(-\mathrm D_x\xi(t,x)) \le \Diss^*(-\mathrm D\varphi(x)),\\
            \xi(T,x)\le \varphi(x).
        \end{cases}
    \end{equation}
    
    \textbf{Step 3.} \textbf{Bounding of the dual problem by zero.}
    We conclude the argument by proving that \eqref{eq:HJintro} possesses a \lq\lq backward boundedness\rq\rq property. Such property admits a natural interpretation as a \emph{comparison principle}: since formally the function $\varphi$ itself solves the associated equation, if a subsolution $\xi$ satisfies the terminal bound $\xi(T, x) \leq \varphi(x)$, then the Hamilton-Jacobi dynamics propagates this bound to all earlier times $t \in [0, T]$. Hence, we obtain $\xi(0,x)\le\varphi(x)$ for all $x\in \R^n$ and in particular we infer that
    \[\sup\limits_\xi\left\{\xi(0,x_0)-\varphi(x_0):\, \xi\in \HJ\right\}\le 0.\] 
    Collecting all the previous considerations we finally obtain \eqref{eq:minintro} and so Theorem \ref{thm:intro} is proved.
 
    The proof of the backward boundedness property proceeds by contradiction, analyzing the first time at which the bound is violated. For simplicity, let us assume that $\xi$ fulfils \eqref{eq:HJintro} with strict inequalities; this can be made general by perturbing the function $\xi$. Consider
    \begin{equation*}
        \overline{t}:=\max\{t\in [0,T):\, \text{there exists } x\in \R^n \text{ for which }\xi(t,x) \ge \varphi (x)\}.
    \end{equation*}
    From the definition of $\bar t$, exploiting smoothness of $\xi$ and $\varphi$, it immediately follows that there exists $\bar x$ such that $\xi(\bar t,\bar x)=\varphi(\bar x)$, and so necessarily $\partial_t\xi(\bar t,\bar x)\le 0$ and $\mathrm D_x \xi(\bar t,\bar x)=\mathrm D\varphi(\bar x)$.

    By plugging $(\overline t,\overline x)$ into the first inequality in \eqref{eq:HJintro}, we finally reach a contradiction:
    \begin{equation*}
        0\le -\partial_t\xi(\bar t,\bar x)=-\partial_t\xi(\bar t,\bar x)+   \Diss^*(-D_x\xi(\bar t,\bar x))- \Diss^*(-\mathrm D\varphi(\bar x))<0.
    \end{equation*} 
    
    \begin{oss}[\textbf{Time- and space-dependent dissipation}]\label{rmk:nonautonomous}
    Here we briefly discuss how our approach naturally handles time and space dependencies of the dissipation potential by illustrating the simple quadratic example
\[
\Diss(t,x, v) = \frac{1}{2} a(t,x) |v|^2,
\]
where $a \colon [0,T]\times\R^n \to (0, +\infty)$ represents a varying friction coefficient. The doubly nonlinear equation becomes
\[
a(t,\xx(t)) \dot{\xx} + \mathrm D\varphi(\xx) = 0,
\]
i.e. a gradient flow with time- and state-dependent metric, while the De Giorgi's functional reads
\begin{equation*}
    \mathcal J(T,\xx)=\varphi(\xx(T))+\int_0^T \Big(\frac 12 a(\tau,\xx(\tau))|\dot {\xx}(\tau)|^2+\frac{1}{2 a(\tau,\xx(\tau))}|\mathrm D\varphi(\xx(\tau))|^2\Big)\,d\tau-\varphi(x_0).
\end{equation*}  
Our method applies directly, indeed the three-step proof carries over with no modifications. The relaxation procedure is unchanged, since continuity equation and superposition principle are not affected by $a$. In the relaxed formulation the action functional
\[
(\mu, \nu)\mapsto
 \int_{X_T} \frac 12 a(t,x) \left| \frac{d\nu}{d\mu}(t,x) \right|^2 d\mu(t,x)
\]
remains convex despite the $(t,x)$-dependence of $a$, so that Von Neumann's result still applies. The Hamilton-Jacobi inequality takes now the form
\[
-\partial_t \xi(t,x) + \frac1{2a(t,x)}
|\mathrm D \xi(t,x)|^2\leq \frac1{2a(t,x)}
|\mathrm D \varphi(x)|^2,
\]
and the backward boundedness property can be checked exactly as before.

Considering a time-dependent energy $\varphi(t,x)$ is also straightforward. This setting only introduces an additional term $-\int_0^T\partial_t\varphi(\tau, \xx(\tau))\, d\tau$ in the De~Giorgi's functional, accounting for the explicit variation of the energy landscape. Such term can be easily tackled under a natural controlled growth condition (see \eqref{eq:powerbound}), without affecting the overall analysis. 

We refer to Section~\ref{sec:nonautonomous} for the precise assumptions.
    \end{oss}

\begin{oss}[\textbf{A Dynamic Programming viewpoint}]
A connection between the De Giorgi's functional \eqref{eq:DeGiorgi} and the emergence of the Hamilton-Jacobi system \eqref{eq:HJintro} could be recovered within Control Theory and the Dynamic Programming Principle. Indeed, the minimum of the energy-dissipation functional $\mathcal J(T,\cdot)$ is naturally related to a particular \emph{value function} defined as
\begin{equation*}
    V(t,x):=\inf\limits_{\mathrm{v}\in L^1(t,T;\R^n)}\left\{\int_t^T L(\xx(\tau),\mathrm{v}(\tau))\, d\tau+\varphi(\xx(T))\,:\, \xx\text{ solves }\dot {\xx}=\mathrm{v},\text{ and }\xx(t)=x\right\},
\end{equation*}
where the Lagrangian is given by
\begin{equation*}
    L(x,v):=\Diss(v)+\Diss^*(-\mathrm D\varphi(x)).
\end{equation*}
Introducing the associated Hamiltonian function
\begin{align}\label{eq:Hamiltonian}
     H(x,p):=& \inf\limits_{v\in \R^n}\left\{p\cdot v+L(x,v)\right\}\nonumber\\
    =& -\Diss^*(-p)+\Diss^*(-\mathrm D\varphi(x)),
\end{align}
it is well known---see e.g. \cite[Chapter 10]{Evans} or \cite[Chapter III.3]{BardiCapuzzoDolcetta}---that, under suitable smoothness and compactness assumptions, the value function $V$ is the unique (viscosity) solution of the Hamilton--Jacobi equation
\begin{equation}\label{eq:HJV}
    \begin{cases}
        -\partial_t V(t,x)-H(x,\mathrm D_x V(t,x))=0,\\
        V(T,x)=\varphi(x).
    \end{cases}
\end{equation}
Exploiting the expression \eqref{eq:Hamiltonian}, the above system directly corresponds to \eqref{eq:HJintro}. As previously observed, $\varphi$ trivially solves \eqref{eq:HJV} and so uniqueness yields $V(t,x)\equiv \varphi(x)$. In particular there holds
\begin{equation*}
    \min\left\{{\mathcal{J}}(T,\xx):\, \xx\in AC([0,T];\R^n),\, \xx(0)=x_0\right\}=V(0,x_0)-\varphi(x_0)=0,
\end{equation*}
formally suggesting the validity of Theorem \ref{thm:intro}.

The previous argument is not rigorous since in general the value function $V$ is not differentiable, and in fact this lack of regularity is expected if $\varphi$ is not smooth. Our approach bypasses this issue, thus providing a successful theory even in an infinite-dimensional nonsmooth setting, identifying in a proper way the dual problem. By passing to the weaker framework of measures, one indeed gains regularity of the dual \lq\lq variables\rq\rq, which turn out to be smooth subsolutions of the Hamilton-Jacobi equation, for which now a comparison principle (essentially replacing the role played by uniqueness) is easy to prove.
\end{oss}

\subsection*{Issues arising in the infinite-dimensional setting}
The paper presents in a rigorous and complete way the previously described idea extending it to a genuine abstract infinite-dimensional framework, which of course brings additional difficulties in the analysis. First, besides the highly nonsmoothness of the involved energies and potentials, which requires the tool of Fréchet subdifferentials in place of standard derivatives, assuming coercivity of the slope \eqref{eq:coerc} is too restrictive, since we want to cover cases where $\varphi$ is not $\lambda$-convex. In order to tackle this problem we will introduce an exponentially weighted version of the De Giorgi's functional, see \eqref{eq:Ja} and Section \ref{sec:reduction}. The recovery of null-minimizers for this new functional still turns out to be equivalent to finding solutions to the doubly nonlinear equation \eqref{eq:DNEintro}, under the validity of a proper chain-rule formula. We refer to Proposition \ref{prop:DeGiorgi} for details.

Moreover, in order to introduce the relaxed functional $\mathcal E$ in the spaces of measures, it is important to give a proper meaning to the continuity equation \eqref{eq:conteqintro} in general Banach spaces. In particular, the choice of an appropriate subclass of bounded smooth functions (the so-called cylinder functions, Definition\ref{def:cylinder}) is crucial to identify the appropriate form of the dual problem. Such test functions are rich enough to recover the validity of the superposition principle in Banach spaces, as noticed in \cite[Theorem~5.2]{ambrosio2021spatially}.

Finally, the last subtle issue we have to face in infinite dimension is to determine a good topology for the triple $(m,\mu,\nu)$ in order to apply Von Neumann's theorem. Indeed, although $m$ and $\mu$ are still nonnegative measures, and the classical narrow topology remains a natural candidate, the measure $\nu$ now takes values in the ambient Banach space, and for such measures the choice of the proper narrow topology is delicate: it must be weak enough for compactness yet strong enough for lower semicontinuity. To this aim, inspired by the theory of vector-valued distributions \cite{schwartz1,schwartz2}, in Section~\ref{subsec: narrow topology} we adopt the \emph{narrow topology} for Banach-valued measures in duality with \emph{finite} sums of simple elements $\zeta(t,x)z$, with $\zeta$ continuous and bounded, and $z$ living in the dual space. We also prove a {general metrizability result} of such topology (Proposition~\ref{prop:localmetrnarrow}), a general compactness criterion (Theorem~\ref{thm: compactness for nu}), and a lower-semicontinuity result (Proposition~\ref{repr formula proposition}) tailored to this setting. Apart from the application to the present work, we strongly believe that these results may have an independent interest and may be useful for different purposes. 

\bigskip
    
\textbf{Plan of the paper.} The paper is organized as follows. Section \ref{sec:Notation} fixes the notation and lists some known notions we will use throughout the paper. Section \ref{sec: setting} is devoted to a rigorous discussion on the De Giorgi's principle for doubly nonlinear equations and to the statement of our main result, regarding the existence of null-minimizers of (a suitable variant of) \eqref{eq:DeGiorgi}. In Section \ref{subsec: narrow topology} we introduce useful tools we will employ for the proof of the main theorem. We first describe the narrow topology for Banach-valued measures, providing local metrizability and a compactness result; we then present the \emph{action functional} $\displaystyle (\mu,\nu)\mapsto\int_{X_T} \Diss\left(\frac{d\nu}{d\mu}\right)\, d\mu$ in an abstract and general setting, discussing its lower semicontinuity and convexity properties and presenting some duality formulas. The content of Section \ref{sec: convexification} is the convexification and relaxation procedure of the De Giorgi's functional to spaces of measures, exploiting the superposition principle. In Section \ref{sec: proof} we then apply the Von Neumann minimax theorem in order to pass to the dual problem, in which the Hamilton-Jacobi inequality appears as a constraint. We also show how to correctly  bound the dual problem, proving the backward boundedness property, thus concluding the argument. Finally, Section \ref{sec:nonautonomous} contains an extension of the result to the case of time- and space-dependent dissipation potentials and nonautonomous energies. At the end of the paper we attach three appendices in which, respectively, we prove the stated properties on narrow topology, we discuss the technical relations between our notion of continuity equation with the more classic one described by a vector field and a curve of probability measures, together with the superposition principle, and we recall suitable measurability properties of Bochner integration.

\section{Notation and preliminaries}\label{sec:Notation}
		
	For any normed space $(X,\|\cdot\|)$ we denote by $(X^*,\|\cdot\|_{*})$ its topological dual and by $\langle z, x\rangle$ the duality product between $z\in X^*$ and $x\in X$. The symbols $B_R$ and $B_R^*$ indicate the closed ball of radius $R>0$ centered at the origin, respectively in in $X$ and $X^*$. We use standard notation for Lebesgue, Sobolev and Bochner spaces. By $AC([a,b];X)$ we mean the set of absolutely continuous functions from $[a,b]$ to $X$.
    
    \subsection{Finite measures and narrow topology}
    Given a Polish space $(Y,\tau)$ and a separable Banach space $X$, we denote by $\mathcal M_+(Y)$ and $\mathcal M(Y;X)$, respectively,
    the sets of nonnegative finite measures and $X$-valued measures $\nu$ with finite total variation over $Y$,
     that is $|\nu|\in \mathcal{M}_+(Y)$, where
     \[|\nu|(A) := \sup\left\{ \sum_{n=1}^{+\infty} \|\nu(E_n)\| \ : \ \bigcup_n E_n = A, \ E_i \cap E_j = \emptyset \text{ as } i\neq j \right\}.\]
    By definition, a vector-valued measure $\nu\in \mathcal M(Y;X)$ is absolutely continuous with respect to $\mu\in \mathcal M_+(Y)$ if its total variation $|\nu|$ is absolutely continuous with respect to $\mu$ in the standard sense; in this case, we write $|\nu|\ll \mu$. 
    
    We recall that the \emph{narrow topology} on $\mathcal{M}_+(Y)$ is the coarsest
    topology for which the functionals $\displaystyle \mathcal{M}_+(Y) \ni \mu\mapsto \int_Y \zeta \,d\mu$ are continuous for all bounded and continuous functions $\zeta\in C_b(Y)$. Compactness in $\mathcal{M}_+(Y)$ with respect to this topology is established in the well-known Prokhorov theorem.
\begin{teorema}[\textbf{Prokhorov} \cite{dellacherie1978probabilities}]\label{thm: prohorov}
    Let $(Z,d)$ be a metric space. Let $\mathfrak{F}\subseteq \mathcal{M}_+(Z)$ be a family of non-negative measures with equi-bounded mass, that is
    \begin{equation}\label{eq: mass bound prohorov}
        \sup_{\mu\in \mathfrak{F}} \mu(Z) <+\infty.
    \end{equation}
        If $\mathfrak{F}$ is uniformly tight, namely
        \begin{equation}\label{eq: tightness}
            \text{for all } \varepsilon>0 \text{ there exists a compact set } K_{\varepsilon}\subset Z \text{ such that } \sup_{\mu\in \mathfrak{F}}\mu(Z\setminus K_\varepsilon) <\varepsilon,
        \end{equation}
        then $\mathfrak{F}$ is relatively compact in $\mathcal{M}_+(Z)$ with respect to the narrow topology.

        Conversely, if $(Z,d)$ is complete and separable, the opposite holds as well, namely relatively compact subsets of $\mathcal{M}_+(Z)$ satisfy \eqref{eq: mass bound prohorov} and \eqref{eq: tightness}.
    
\end{teorema}

\subsection{Radon-Nikod\'ym spaces}
A Banach space $X$ is said to be a \emph{Radon-Nikod\'ym space}, or to possess the Radon-Nikod\'ym property, if for any finite measure space $(\Omega,\mathcal{A},\mu)$ and any $X$-valued measure $\nu:\mathcal{A}\to X$ which is absolutely continuous with respect to $\mu$, the measure $\nu$ admits density $\displaystyle \frac{d \nu}{d\mu}\in L^1_\mu(\Omega;X)$ with respect to $\mu$. Standard known examples of Radon-Nikod\'ym spaces are reflexive spaces or separable dual spaces (see \cite[pp. 79]{DiestelUhl}).

An important property of Radon-Nikod\'ym spaces is that any curve $\xx\in AC([a,b];X)$ admits a summable derivative $\dot {\xx}\in L^1(a,b;X)$ and satisfies the identity 
\begin{equation*}
    \xx(t)-\xx(s)=\int_s^t \dot {\xx}(\tau)\,d\tau,\qquad\text{for all }s,t\in [a,b].
\end{equation*} 
 
\subsection{Fréchet and limiting subdifferentials}
	We briefly recall some basic definitions in convex  and nonsmooth  analysis (see for instance \cite{Borwein-Zhu,Mordukhovich18,Rocka,Rockafellar-Wets98}). Given a function $\Phi\colon X\to(-\infty,+\infty]$, its (viscosity) \emph{Fréchet subdifferential} 
    $\partial\Phi\colon X\rightrightarrows X^*$ at a point $x\in 
    \rm dom (\Phi):=\{\Phi<+\infty\}$ 
    is defined as the set of $z\in X^*$
    such that 
    there exists a $C^1$ function
    $g:X\to \R$ such that 
    $\mathrm D g(x)=z$ and 
    $\Phi-g$ has a local minimum at $x$.
    It is also not restrictive to assume that $g$ is concave.
    When $X$ is reflexive 
    we have the equivalent condition
	\begin{equation*}
		z\in \partial \Phi(x)
        \quad\Leftrightarrow\quad
        \liminf\limits_{y\to x} \frac{\Phi(y)-\Phi(x)-\langle z,y-x\rangle}{\|y-x\|}\ge 0.
	\end{equation*}
     If $x\notin{\rm dom}\Phi$, one convenes that $\partial \Phi(x)=\emptyset$.   It may happen that 
    the graph of $\partial\Phi$
    is not strongly-weakly$^*$ sequentially closed,
    so it is natural to introduce
    suitable notions of
    limiting subdifferentials
    arising from $\partial\Phi$
    \cite{mordukovich1980nonsmooth,Mordukhovich18}.
    We adopt the following natural one
    \cite{MRSNonsmooth,rossi2006gradient}:
    \begin{df}
        \label{def:limiting-subdifferential}
        Let $x\in \overline{\operatorname{dom}(\partial\Phi)}$. We say that 
        $z\in X^*$
        belongs to the 
        \emph{limiting subdifferential}
        $\partial_\ell\Phi(x)$
        of $\Phi$
        at $x$ if there
        exist sequences
        $(x_n)_n$ in ${\operatorname{dom}(\partial\Phi)}$
        and $(z_n)_n$ in $X^*$ such that 
        \begin{equation*}
            z_n\in \partial\Phi(x_n),
            \quad 
            x_n\to x,\quad
            z_n\stackrel{*}{\rightharpoonup}z,\quad
            \text{as }n\to\infty.
        \end{equation*}
        We say that $\partial\Phi$
        is strongly-weakly$^*$ sequentially closed if
        $\partial_\ell\Phi=\partial\Phi$, i.e.~
        \begin{equation*}
         z_n\in \partial\Phi(x_n),
            \quad 
            x_n\to x,\quad
            z_n\stackrel{*}{\rightharpoonup}z,\quad
            \text{as }n\to\infty
            \qquad\Rightarrow
            \qquad
            z\in \partial\Phi(x).
        \end{equation*}
    \end{df}
    If $\Phi$ is convex (and lower semicontinuous)
    then the Fréchet subdifferential coincides with the 
    usual  \emph{subdifferential of convex analysis}, namely
	\begin{equation*}
	\partial\Phi(x)=\left\{z\in X^*:\,\Phi(y)\ge \Phi(x)+\langle z,y-x\rangle,  \quad\text{for every $y\in X$}\right\},
\end{equation*}
and $\partial\Phi=\partial_\ell\Phi.$

	\subsection{Fenchel conjugate}
	We recall
    the \emph{Fenchel conjugate} of the (proper) function $\Phi\colon X\to (-\infty,+\infty]$, namely the convex, weak$^*$ lower semicontinuous function
	\begin{equation*}
		\Phi^*\colon X^*\to (-\infty,+\infty],\quad \text{defined as }\quad \Phi^*(z):=\sup\limits_{x\in X} \{\langle z,x\rangle-\Phi(x)\}.
	\end{equation*}
	If $\Phi$ is convex, the very definition yields that for every $ z\in X^*$ and $x\in X$ the Fenchel conjugate $\Phi^*$ satisfies
	\begin{equation}\label{Fenchelprop}
		\Phi^*(z)+\Phi(x)\ge \langle{z},{x}\rangle,\quad\text{ with equality if and only if }\quad z\in \partial \Phi(x).
	\end{equation}

 The following elementary properties can be directly proved from the definition of $\Phi^*$.
 \begin{lemma}\label{lemma:propFenchel}
    If $\Phi$ is convex and has superlinear growth at infinity, then $\Phi^*$ is locally bounded. In particular, $\operatorname{dom}\Phi^*=X^*$ and $\Phi^*$ is continuous on $X^*$.
   
    If $0\in \operatorname{dom}\Phi$ and $\Phi$ is convex and continuous at $0$, then $\Phi^*$ has bounded sublevels.
\end{lemma}
\begin{proof}
    Let $\bar x\in \operatorname{dom}\Phi$, then $\Phi^*(z)\ge \langle z,\bar x\rangle-\Phi(\bar x)$ for all $z\in X^*$, namely $\Phi^*$ is locally bounded from below. By superlinearity we also know that for all $R>0$ there exists $c_R>0$ such that $\Phi(x)\ge R\|x\|-c_R$ for all $x\in X$, whence for all $z\in B_R^{*}$ we deduce $$\Phi^*(z)\le \sup\limits_{x\in X} \{R\|x\|-\Phi(x)\}\le c_R.$$
    Thus $\Phi^*$ is locally bounded, and since it is convex we infer it is continuous in the whole of $X^*$.

    Assume now that $0\in \operatorname{dom}\Phi$ and $\Phi$ is continuous at $0$. By convexity, this implies boundedness of $\Phi$ in a neighborhood of the origin, say $|\Phi(x)|\le M$ for $x\in B_r$. Consider $z\in \{\Phi^*\le C\}$, then by \eqref{Fenchelprop} we deduce
    \begin{equation*}
        \|z\|_*=\sup\limits_{x\in B_1}\langle z,x\rangle \le \sup\limits_{x\in B_1}\frac{\Phi^*(z)+\Phi(rx)}{r}\le \frac{C+M}{r},
    \end{equation*}
    and we conclude.
\end{proof}

\subsection{Continuity equation and cylinder functions}
    Let $X$ be a Radon-Nikod\'ym space.
    For a horizon time $T>0$ we introduce the notation $X_T:=[0,T]\times X$ for the space-time domain. Given a vector-field $v\colon X_T\to X$, we consider the \emph{continuity equation}
    \begin{equation}\label{eq:conteq}
        \partial_t\mu + \operatorname{div}_x(v\mu) = 0, \qquad\text{in }X_T.
    \end{equation}
    Since $X$ is infinite-dimensional, a good set of test functions to give a meaning to the above equation is provided by the so-called \emph{cylinder functions}.
\begin{df}[\textbf{Cylinder functions}]\label{def:cylinder} 
A function $\xi:X_T\to\R$ is called a cylinder function 
if there exist 
    $k\in \N$, $\boldsymbol{z}=(z_1,\dots,z_k) \in (X^*)^k$ and $\zeta \in C^1_c([0,T]\times\R^k)$  such that 
    \begin{equation*}
        \xi(t,x) = \zeta\left(t,\langle \boldsymbol{z},x \rangle_k\right),\qquad \text{where }\langle\boldsymbol{z}, x\rangle_k := \left(\langle z_1,x \rangle, \dots , \langle z_k,x \rangle\right).
    \end{equation*}
    In such case, we write $\xi \in \operatorname{Cyl}_c(X_T)$. If the function $\zeta$ belongs to $ C_b^1([0,T]\times\R^k)$, namely it is bounded with bounded derivatives, we instead write $\xi \in \operatorname{Cyl}_b(X_T)$.
\end{df}
Observe that $\operatorname{Cyl}_c (X_T) \subset \operatorname{Cyl}_b(X_T)$, and that any cylinder function $\xi \in \operatorname{Cyl}_b(X_T)$ is Fréchet differentiable for all $(t,x)\in X_T$ with
\begin{equation*}
    \mathrm D\xi(t,x) =(\partial_t \xi(t,x), \mathrm D_x \xi(t,x))=\left(\partial_t \zeta(t,\langle \boldsymbol{z},x\rangle_k ), \sum_{i=1}^k \mathrm D_i \zeta(t,\langle \boldsymbol{z},x\rangle_k ) z_i\right)\in \R\times X^*.
\end{equation*}

\begin{df}\label{continuity equation over X def}
    Let $\mu \in \mathcal{M}_+(X_T)$ and $v\in L^1_\mu(X_T;X)$. We say that the pair $(\mu,v)$ solves the continuity equation \eqref{eq:conteq} with starting measure $m_0\in \mathcal{M}_+(X)$ and ending measure $m\in \mathcal{M}_+(X)$ if for all $\xi\in \operatorname{Cyl}_c(X_T)$ it holds
    \begin{equation}\label{eq: CE test}
     \int_{X_T} \partial_t \xi d\mu + \int_{X_T} \langle \mathrm D_x\xi, v\rangle d\mu = \int_X \xi(T) dm - \int_X \xi(0) dm_0.
    \end{equation}
\end{df}
A simple approximation procedure shows that \eqref{eq: CE test} is still valid if we consider a more general $\xi \in \operatorname{Cyl}_b(X_T)$ as a test function.
\begin{lemma}
    Let $\mu \in \mathcal{M}_+(X_T)$ and let $v\in L^1_\mu(X_T;X)$. Then $(\mu,v)$ solves the continuity equation \eqref{eq:conteq} starting from $m_0$ and ending at $m$ if and only if \eqref{eq: CE test} is satisfied for all $\xi \in \operatorname{Cyl}_b(X_T)$.
\end{lemma}
    
    Since $X$ has the Radon-Nikod\'ym property, given $\mu\in \mathcal{M}_+(X_T)$, there is a one-to-one correspondence between vector fields $v\in L^1_\mu(X_T;X)$ and measures $\nu \in \mathcal{M}(X_T ; X)$ that are absolutely continuous with respect to $\mu$. Indeed, with any $v\in L^1_\mu(X_T;X)$ one can associate the measure $\nu:=v\mu\in \mathcal{M}(X_T ; X)$; on the other hand, given $\nu \in \mathcal{M}(X_T ; X)$ such that $|\nu| \ll \mu$, one can consider its Radon-Nikod\'ym derivative $\displaystyle v:=\frac{d \nu}{d \mu}\in L^1_\mu(X_T;X)$. This observation leads to the following (equivalent but useful) definition.

    \begin{df}\label{def:CEsol} We say that a pair of measures $(\mu,\nu)\in \mathcal{M}_+(X_T) \times \mathcal{M}(X_T;X)$ solves the continuity equation \eqref{eq:conteq} starting from $m_0\in \mathcal M_+(X)$ and ending at $m\in \mathcal M_+(X)$ if 
    \begin{equation}\label{eq:CE}
        \int_{X_T} \partial_t \xi d\mu + \int_{X_T} \langle \mathrm D_x\xi, d\nu\rangle  = \int_X \xi(T) dm - \int_X \xi(0) dm_0,\quad\text{for all $\xi \in \operatorname{Cyl}_b(X_T)$.}
    \end{equation}
    In this case, we write $$(\mu,\nu)\in \CE(m_0,m).$$
\end{df}

We refer to Appendix \ref{app: CE} to relate these definitions to other standard approaches to the continuity equation.

\section{Differential inclusions, De Giorgi's principle, and 
main result}\label{sec: setting}

In all this section 
we will assume that 
$(X,\|\cdot\|)$ is a reflexive and separable Banach space. In particular, we recall that $X$ satisfies the Radon-Nikod\'ym property. 
\subsection{Doubly nonlinear differential inclusions
and De Giorgi's principle}
We are interested in the 
absolute continuous curves 
$\xx:[0,T]\to X$ 
solving the 
Cauchy problem for the relaxed doubly nonlinear differential inclusion 
\begin{equation}
\label{eq:DNE}
    \begin{cases}
        \partial \Diss(\dot {\xx}(t))+
         \partial_\ell 
         \varphi(\xx(t))\ni 0\quad\text{in }X^*,\quad\text{for a.e. }t\in(0,T),\\
        x(0)=x_0\in X,
    \end{cases}
\end{equation}
 where $\partial_\ell\varphi$
denotes the limiting subdifferential, introduced in Definition
\ref{def:limiting-subdifferential}.

We will assume 
that the proper \emph{energy} $\varphi\colon X\to (-\infty,+\infty]$ and the \emph{dissipation potential} $\Diss\colon X\to (-\infty,+\infty]$ satisfy the following basic 
structural properties:

\begin{enumerate}[label=\textup{($\varphi$\arabic*)}]
  \item \label{hyp:phi1} sublevels of $\varphi$ are boundedly compact, namely
  \begin{equation*}
      \text{for all $C\in\mathbb{R}$ and $R>0$ the set $\{\varphi\le C\}\cap B_R$ is compact in $X$;}
  \end{equation*} 
	\item \label{hyp:phi4} $\varphi$ 
    is linearly bounded from below, namely
    there exist constants
    $c_1,c_2\ge0$ such that 
    \begin{equation*}
        \varphi(x)\ge -c_1-c_2\|x\|\quad\text{for all }x\in X.
    \end{equation*}
\end{enumerate}
Notice that \ref{hyp:phi1} implies that $\varphi$ is lower semicontinuous. As regards the dissipation potential, which we stress could take the value $+\infty$ (but see \ref{psi4} in Proposition \ref{prop:F*nabla}), we require:
\begin{enumerate}[label=\textup{($\Diss$\arabic*)}]
    \item \label{hyp:F1} $\Diss$ is convex and lower semicontinuous;
  \item \label{hyp:F2} $\Diss$ has superlinear growth at infinity, i.e. $\lim\limits_{\|v\|\to \infty}\frac{\Diss(v)}{\|v\|}=+\infty$;
  \item \label{hyp:F3} $\Diss(0)<+\infty$.
\end{enumerate}
We observe that this latter set of assumptions ensures that 
\begin{equation}\label{eq:Fboundbelow}
    \text{both $\Diss$ and $\Diss^*$ are bounded below, with $\Diss^*(z)\ge -\Diss(0)$ for all $z\in X^*$.}
\end{equation}
We also point out that condition \ref{hyp:F2} is natural if one looks for absolutely continuous solutions to \eqref{eq:DNE}. 
Indeed, the linear growth regime leads to the framework of \emph{rate-independent systems} \cite{MielkRoubbook}, where discontinuous solutions are expected to appear. 
\begin{oss}
     The growth conditions \ref{hyp:phi4} and \ref{hyp:F2} are strictly related. Indeed, our results still hold true strenghtening \ref{hyp:F2} while weakening \ref{hyp:phi4}:
    \begin{itemize}
        \item [$(\Diss2')$] $\Diss$ has super $p$-growth at infinity, i.e. $\lim\limits_{\|v\|\to \infty}\displaystyle\frac{\Diss(v)}{\|v\|^p}=+\infty$, for some $p\ge 1$;
        \item[$(\varphi 2')$] $\varphi(x)\ge -c_1-c_2\|x\|^p$ for all $x\in X$. 
    \end{itemize}
\end{oss}
In order to state De Giorgi's principle, for $a\ge 0$ we consider the energy-dissipation functional $\calJ^a\colon [0,T]\times AC([0,T];X)\to (-\infty,+\infty]$ defined as
\begin{equation}\label{eq:Ja}
    \calJ^a(t,\xx):=e^{-at}\varphi(\xx(t))+\int_0^t e^{-a\tau}\Big(\Diss(\dot{\xx}(\tau)) + \rS(\xx(\tau))+a\varphi(\xx(\tau))\Big)\, d\tau-\varphi(x_0),
\end{equation}
where we introduced the 
{\em relaxed $\Diss^*$-slope function}
$\rS\colon X\to (-\infty,+\infty]$, namely
\begin{equation*}
      \rS(x):=      
       \inf\Big \{
        \liminf_{n\to\infty}
        \Diss^*(-z_n):\, z_n\in\partial\varphi(x_n),
        \
        x_n\to x\text{ in }X
        \Big\}
        \ \text{if }x\in 
        \overline{\operatorname{dom}\varphi},
  \end{equation*}
with the usual convention 
$\rS(x):=+\infty$ if $x$ does not belong to the closure of the proper domain of $\varphi$.
The relaxed slope $\rS$ is the largest lower semicontinuous function satisfying
\begin{equation}
    \label{eq:obvious-relaxation}
    \rS(x)\le \Diss^*(-z)
    \quad
    \text{for every }z\in \partial\varphi(x).
\end{equation}
Notice that \eqref{eq:Ja} coincides with \eqref{eq:DeGiorgi} for $a=0$, up to relaxing the slope \eqref{eq:slope}. Such a relaxation 
plays an important role,
since in general 
the slope \eqref{eq:slope}
may lack lower semicontinuity (but see Proposition \ref{prop:F*nabla}). 
One of the roles of the parameter $a$, as we will see, consists in gaining coercivity of the $x$-dependent terms inside the integral in \eqref{eq:Ja},
a crucial property for 
any variational argument. 

It will be useful to introduce the set 
\begin{equation}\label{eq:nablaF}
      \Dpartial \varphi(x):=
      \Big\{z\in {\partial_\ell}\varphi(x):
      \Diss^*(-z)\le \rS(x)\Big\} \subseteq {\partial_\ell}\varphi(x).
  \end{equation}
Notice that $\Dpartial\varphi(x)$ may be empty even at points where $S^-(x)<\infty$ 
since the sublevels of $\Diss^*$ 
could be unbounded.
When $\psi$ is continuous at $0$
the proper domain of $\Dpartial\varphi$ coincides 
with the proper domain of $\partial_\ell\varphi$
(see Proposition \ref{prop:F*nabla}). 

Next proposition finally states the De Giorgi's principle for the \textit{relaxed doubly non linear equation}
\eqref{eq:DNE} 
relating its solutions to null-minimizers of \eqref{eq:Ja}. We point out that the validity of the chain rule \eqref{eq:chainrule} below is required just for elements of the smaller set $\partial^*\varphi$, and not of the whole $\partial_\ell\varphi$.

\begin{prop}[\textbf{De Giorgi's principle}]\label{prop:DeGiorgi} 
    Assume the validity of the following two properties:
    \begin{enumerate}[label=\textup{(S)}]
    \item\label{M} 
    $\Dpartial\varphi(x)$ is nonempty for every 
    $x\in \operatorname{dom}\rS$;
    \end{enumerate}
    \begin{enumerate}[label=\textup{(CR)}]
    \item\label{CR} the following \emph{chain-rule} holds: for all $\xx\in AC([0,T];X)$ such that $$\int_0^T 
    \Big(\Diss(\dot{\xx}(\tau)) + \rS(\xx(\tau))
     \Big)\, d\tau<+\infty,$$ the map $t\mapsto \varphi(\xx(t))$ is absolutely continuous and 
    \begin{equation}\label{eq:chainrule}
        \frac{d}{dt}\varphi(\xx(t))=\langle \mathrm z(t), \dot \xx(t)\rangle,\quad\text{for a.e. }t\in(0,T),
    \end{equation}
    for any measurable selection $\mathrm z\colon [0,T]\to X^*$ satisfying $\mathrm z(t)\in {\Dpartial}\varphi(\xx(t))$
    a.e.~in $[0,T]$.
     \end{enumerate}
    Then, for every initial datum $x_0\in \operatorname{dom}\varphi$ and trajectory $\bar {\xx}\in AC([0,T];X)$ with $\bar {\xx}(0)=x_0$ we have
    \begin{equation}
    \label{eq:DGpositive}
        \calJ^a(t,\bar{\xx})\ge 0, \quad\text{for all }t\in [0,T] \text{ and }a\ge 0.
    \end{equation}
    Moreover, the following conditions are equivalent:
    \begin{itemize}
        \item [(a)] $\bar {\xx}$ solves the 
        \textit{relaxed doubly nonlinear equation} \eqref{eq:DNE},
        in the sense that there exists a measurable map $\bar{\mathrm z}\colon [0,T]\to X^*$ 
        such that
        \begin{equation}\label{eq:sol}
            -\bar {\mathrm z}(t)\in \partial \Diss(\dot{\bar {\xx}}(t)),\quad 
            \bar{\mathrm z}(t)\in 
             {\partial_\ell}\varphi(\bar{\xx}(t))
            \qquad\text{for a.e. }t\in (0,T),
        \end{equation}
        and 
        \begin{equation}
        \label{eq:sol2} \displaystyle 
        \psi^*(-\bar{\mathrm z}(t))=S^-(\bar{\xx}(t))
        \quad\text{
        a.e.~in }(0,T),
        \quad \int_0^T \Diss(\dot{\bar {\xx}}(\tau)) + \rS(\bar {\xx}(\tau))\, d\tau<+\infty;
        \end{equation}
        \item[(b)] $\calJ^a(t,\bar {\xx})\le 0$ for all $t\in [0,T]$ and for some $a\ge 0$;
        \item[(c)] $\calJ^a(t,\bar {\xx})= 0$ for all $t\in [0,T]$ and for some $a\ge 0$;
        \item[(d)] $\calJ^a(T,\bar {\xx})
        \le 0$ for some $a\ge 0$.
    \end{itemize}
    Moreover, if one of the above is in force, then $(b)$, $(c)$ and $(d)$ hold true for all $a\ge 0$, and $\bar{\xx}$ 
    attains the minimum of $\mathcal J^a(T,\cdot)$, that is
    $\mathcal J^a(T,\bar{\xx})=
    \min\left\{\calJ^a(T, \xx):\, \xx\in AC([0,T];X),\, \xx(0)=x_0 \right\}$.
\end{prop}
\begin{proof}
    Let us first prove \eqref{eq:DGpositive}.
    Indeed, without loss of generality assume $\calJ^a(t,\bar \xx)<+\infty$, whence also the integral $\int_0^t \Diss(\dot{\bar\xx}(\tau)) + \rS(\bar \xx(\tau))\, d\tau$ is finite. So, by exploiting Fenchel's inequality \eqref{Fenchelprop} and the chain-rule \eqref{eq:chainrule}, for any Lebesgue-measurable function 
    $\mathrm z\colon [0,T]\to X^*$ satisfying $\mathrm z(t)\in \Dpartial\varphi(\bar \xx(t))$ almost everywhere we deduce
    \begin{align*}
        \calJ^a(t,\bar \xx)&\ge e^{-at}\varphi(\bar \xx(t))+\int_0^t e^{-a\tau}(-\langle \mathrm z(\tau),\dot {\bar{\xx}}(\tau)\rangle+a\varphi(\bar\xx(\tau)) \, d\tau-\varphi(x_0)\\
        &=e^{-at}\varphi(\bar \xx(t))-\int_0^t \frac{d}{d\tau}e^{-a\tau}\varphi(\bar \xx(\tau) \, d\tau-\varphi(x_0)=0.
    \end{align*}    
    We recall that at least one measurable selection ${\mathrm{z}}(t) \in \Dpartial\varphi(\bar \xx(t))$ exists. Indeed, due to Proposition \ref{prop: meas lim}, the set $\Dpartial \varphi$ is Borel measurable, and thanks to \ref{M} a classical measurable selection argument can be performed (see e.g. \cite[Theorem 6.9.1]{Bogachev07})

    Let us now prove that $(d)\implies (a)$. Since $\calJ^a(T,\bar {\xx})\le 0$, observe that $\int_0^T \Diss(\dot{\bar{\xx}}(\tau)) + \rS(\bar {\xx}(\tau))\, d\tau$ is finite, so the chain rule may be applied. We fix any measurable function $\mathrm z\colon [0,T]\to X^*$ satisfying $\mathrm z(t)\in \Dpartial\varphi(\bar {\xx}(t))$ almost everywhere, and arguing as before we obtain
    \begin{align*}
        0&\ge 
        \calJ^a(T,\bar {\xx}) 
        =  e^{-aT}\varphi(\bar {\xx}(T))+\int_0^T e^{-a\tau}(\Diss(\dot{\bar {\xx}}(\tau)) + S^-(\bar \xx(\tau))+a\varphi(\bar {\xx}(\tau)) \, d\tau-\varphi(x_0) 
        \\&
        \ge  e^{-aT}\varphi(\bar {\xx}(T))+\int_0^T e^{-a\tau}(\Diss(\dot{\bar {\xx}}(\tau)) + \Diss^*(-\mathrm z(\tau))+a\varphi(\bar {\xx}(\tau)) \, d\tau-\varphi(x_0) \\
        &\ge e^{-aT}\varphi(\bar {\xx}(T))+\int_0^T e^{-a\tau}(-\langle \mathrm z(\tau),\dot{\bar {\xx}}(\tau)\rangle+a\varphi(\bar {\xx}(\tau)) \, d\tau-\varphi(x_0)=0.
    \end{align*}
    This implies that $\displaystyle \int_0^T e^{-a\tau}(\Diss(\dot{\bar {\xx}}(\tau)) + \Diss^*(-\mathrm z(\tau))+\langle \mathrm z(\tau),\dot{\bar {\xx}}(\tau)\rangle)\, d\tau=0$, whence (recalling \eqref{Fenchelprop}) $-\mathrm z(t)\in \partial \Diss(\dot{\bar {\xx}}(t))$ for a.e. $t\in (0,T)$, namely $\bar {\xx}$ satisfies \eqref{eq:sol}
    and \eqref{eq:sol2}.

    Finally, we show that $(a)\implies (c)$ (which clearly implies $(b)$ and $(d)$). Pick any $a\ge 0$ and fix $t\in [0,T]$. From \eqref{eq:sol} and \eqref{eq:sol2}, by using Fenchel's equality and exploiting again the chain-rule we infer
    \begin{align*}
        \calJ^a(t,\bar {\xx})= e^{-at}\varphi(\bar {\xx}(t))+\int_0^t e^{-a\tau}(-\langle \bar {\mathrm z}(\tau),\dot{\bar {\xx}}(\tau)\rangle+a\varphi(\bar {\xx}(\tau)) \, d\tau-\varphi(x_0)=0,
    \end{align*}
    and we conclude.
\end{proof}
In the literature, conditions $(b)$ and $(c)$ have been adopted as a definition for the so-called EDI (energy-dissipation inequality) and EDE (energy-dissipation equality) solutions, widely employed even in metric setting \cite{ambrosio2005gradient}. In this paper we are instead interested in condition $(d)$.

We stress that the chain-rule assumption \ref{CR} is crucial to have equivalence between the four previous conditions. Without it, the minimal request $(d)$ presents some pathological behaviour, as explained even by the following one-dimensional example.
\begin{es}
    For $X=\R$, consider the quadratic dissipation potential $\Diss(v)=\frac 12 v^2$ and the energy
\begin{equation}\label{eq:weird}
    \varphi(x)=\begin{cases}
        0,&\text{if }x\neq 1,\\
        -\lambda,&\text{if }x= 1,
    \end{cases}
    \qquad\text{for some }\lambda>0.
\end{equation}
In this setting, it is immediate to check that the slope is identically zero (also in the \lq\lq problematic\rq\rq point $x=1$) and that $\partial^*\varphi(x)=\{0\}$ for all $x\in \R$. However, clearly $\varphi$ does not comply with the chain-rule due to the discontinuity at $x=1$.

Let now $x_0=0$ and, for simplicity, we fix $a=0$ so that
\begin{equation*}
    \mathcal J^0(t,\xx)=\varphi(\xx(t))+\int_0^t\frac 12 \dot {\xx}(\tau)^2\, d\tau.
\end{equation*}

Then, the unique solution of the gradient flow in the sense of $(a)$, $(b)$ and $(c)$ is the constant curve $\xx(t)\equiv0$. However, if $2\lambda T>1$, the minimum of $\mathcal J^0(T,\cdot)$ is attained at the (unreasonable) curve $\bar {\xx}(t)=\frac tT$ with value
\begin{equation*}
    \mathcal J^0(T,\bar {\xx})=-\lambda+\frac{1}{2T}<0.
\end{equation*}
The example can be easily extended to all values $a>0$.
\end{es} 

Sufficient conditions for an energy $\varphi$ to comply with the chain-rule have been thoroughly discussed in \cite[Section 2.2]{MRSNonsmooth}, to which we refer for details. We mention that they are related to \emph{uniform subdifferentiability} properties of $\varphi$, which \eqref{eq:weird} clearly lacks at $x=1$.

On the other hand, assumption \ref{M} is needed to select a descent direction, see Proposition \ref{prop:F*nabla} for a sufficient condition, namely to have solutions to \eqref{eq:DNE} in the sense of (a). Indeed, consider the following example, again in dimension one.
\begin{es}
     For $X=\R$, consider the convex energy $\varphi(x)=
     \mathrm I_{[-1,1]}(x)$, where $\mathrm I_{[-1,1]}$ denotes the indicator function of the interval $[-1,1]$; the subdifferential of $\mathrm I_{[-1,1]}$ is the maximal monotone graph
     \begin{equation*}
         \partial\mathrm I_{[-1,1]}(x)=
         \begin{cases}
             (-\infty,0]&\text{if }x=-1,\\
             \{0\}&\text{if }x\in (-1,1),\\
             [0,+\infty)&\text{if }x=1,\\
             \emptyset&\text{otherwise}.
         \end{cases}
     \end{equation*}
     We also consider the dissipation potential
\begin{equation*}
    \Diss(v)=\begin{cases}
    v(\log v-1),&\text{if }v>0,\\
        0,&\text{if }v=0,\\
        +\infty,&\text{if }v<0.
    \end{cases}
\end{equation*}
Notice that \ref{psi4} is not fulfilled. For $x_0\in [-1,1]$, simple computations show that the corresponding equation reads as
\begin{equation}\label{eq:esM}
    \begin{cases}
        -\log(\dot {\xx}(t))\in \partial\mathrm I_{[-1,1]}(\xx(t)),&\text{for a.e. }t\in (0,T),\\
        \dot {\xx}(t)>0,&\text{for a.e. }t\in (0,T),\\
        \xx(0)=x_0.
    \end{cases}
\end{equation}
Moreover, observing that $\Diss^*(z)=e^z$, the relaxed 
slope and $\Dpartial\varphi$ can be explicitely computed
\begin{equation*}
    \rS(x)=
    \begin{cases}
        1,&\text{if }x\in[-1,1),\\
        0,&\text{if }x=1,\\
        +\infty,&\text{otherwise,}
    \end{cases}
    \qquad
    \Dpartial\mathrm I_{[-1,1]}(x)=
         \begin{cases}
             \{0\}&\text{if }x\in [-1,1),\\
         \emptyset&\text{otherwise}.
         \end{cases}
\end{equation*}
In particular, $\Dpartial\mathrm I_{[-1,1]}(1)$ is empty even if $\partial\mathrm I_{[-1,1]}(1)$ is not, and \ref{M} fails at $x=1$.

Notice that in this setting the functional $\mathcal J^a$, at least when restricted to curves lying in the relevant set $[-1,1]$, reads as
\begin{equation*}
    \mathcal J^a(t,\xx)=\int_0^te^{-a\tau} (\Diss(\dot {\xx}(\tau))+\rS(\xx(\tau))) \,d\tau.
\end{equation*}
It is then immediate to check that the (unique) solution in the sense of $(b)$, $(c)$ and $(d)$ is given by
\begin{equation*}
    \bar {\xx}(t)=\begin{cases}
        x_0+t,&\text{if }t\in [0,1-x_0),\\
        1,&\text{if }t\in [1-x_0,T],
    \end{cases}
\end{equation*}
which solves the equation \eqref{eq:esM} (i.e. fulfils condition $(a)$) just in $(0,1-x_0)$, namely before it reaches the position $x=1$, where a singularity occurs.
\end{es}
\subsection{Main result.}
In this paper we would like to find a direct method which shows that 
every minimizer of $\ucalJ^a(T,\cdot)$ is in fact a null-minimizer, thus verifying condition $(d)$ in the De Giorgi's principle.

Our main results can be stated as follows. The proof 
of the first statement 
will be the content of Sections \ref{sec: convexification} and \ref{sec: proof}.
  \begin{teorema}\label{thm:main}
 Let us assume
  \ref{hyp:phi1}--\ref{hyp:phi4} and \ref{hyp:F1}--\ref{hyp:F3} 
  and let $a\ge0$,
  $T>0$, and $x_0\in \operatorname{dom}\varphi$ be given.
  
  If $a>0$, 
      or if $a=0$ and $\rS$ has boundedly
      compact sublevels, 
      then the functional $\ucalJ^a(T,\cdot)$ 
      attains the minimum 
      in the class
      of absolutely continuous curves starting from $x_0$ and 
      \begin{equation}\label{eq:prob}
          \min\Big\{\ucalJ^a(T,\xx):\, \xx\in AC([0,T];X),\, \xx(0)=x_0\Big\}\le 0.
      \end{equation} 
  \end{teorema}
    \begin{co}
      Under the same assumptions
      of Theorem \ref{thm:main},
    suppose      
      in addition 
      that conditions \ref{M} and \ref{CR} of Proposition \ref{prop:DeGiorgi} are in force (see Proposition \ref{prop:F*nabla} for a sufficient condition
      ensuring \ref{M}).
      
      Then 
      the functional $\ucalJ^a(T,\cdot)$ 
      attains the minimum 
      in the class
      of absolutely continuous curves starting from $x_0$,
      the 
      set of minimizers 
      does not depend on $a$,
      and its elements are 
      solutions of the relaxed doubly nonlinear equation \eqref{eq:DNE}. 
      
      If in particular the 
      Fréchet subdifferential
      of $\varphi$ 
      is strongly-weakly sequentially closed (so that $\partial\varphi$ concides with
      $\partial_\ell\varphi$), 
      the minimizers are actually
      solutions of 
      \begin{equation*}
    \begin{cases}
        \partial \Diss(\dot {\xx}(t))+ \partial \varphi(\xx(t))\ni 0\quad\text{in }X^*,\quad\text{for a.e. }t\in(0,T),\\
        x(0)=x_0\in X,
    \end{cases}
\end{equation*}
  \end{co}
  \begin{proof}
      If $a>0$ 
      (or $a=0$ and $\rS$ 
      has boundedly compact sublevels), 
      the statement is an immediate consequence of 
      Theorem \ref{thm:main} 
      and Proposition 
      \ref{prop:DeGiorgi}.

      In the general case when 
      $a=0$ 
      we already know 
      (by the previous discussion of the case $a>0$ and 
      Proposition 
      \ref{prop:DeGiorgi})
      that the minimum of $\ucalJ^0$ 
      is attained and is null.
      A further application of Proposition 
      \ref{prop:DeGiorgi}
      yields that 
      every minimizer of 
      \eqref{eq:prob}
      for $a=0$
      is a solution to 
      \eqref{eq:DNE}
      so that it 
      also a minimizer
      of \eqref{eq:prob} for every $a>0$.      
  \end{proof}

We conclude this section by providing sufficient conditions on the energy $\varphi$ and the dissipation potential $\Diss$ for the validity of condition \ref{M}.
\begin{prop}\label{prop:F*nabla}
      In addition to \ref{hyp:phi1}-\ref{hyp:phi4} and \ref{hyp:F1}-\ref{hyp:F3}, assume that 
      \begin{enumerate}[label=\textup{($\Diss4$)}]
  \item\label{psi4} $\Diss$ is continuous at $0$.
\end{enumerate} 
Then for every $x\in \operatorname{dom}\rS$, the set $\Dpartial\varphi(x)$ is nonempty 
      so that \ref{M}
      holds.

If moreover 
$\partial\varphi$
is strongly-weakly sequentially closed,
then
      \begin{equation}
      \label{eq:rS-representation}
          \rS(x)=
          \min\Big\{\Diss^*(-z):
          z\in \partial\varphi(x)\Big\}.
      \end{equation}
  \end{prop}
 \begin{proof}
      Let us first observe that 
  $\Diss^*$ has bounded sublevels by $(\Diss4)$
  (see Lemma \ref{lemma:propFenchel}). If $S^-(x)<\infty$
  then we can find 
  a sequence 
  $x_n\in \operatorname{dom}(\partial\varphi)$
  and $z_n\in \partial\varphi(x_n)$
  such that 
  $$x_n\to x,\quad
  \psi^*(-z_n)\to S^-(x).$$
  Since $z_n$ is bounded in $X^*$,
  up to extracting a further (not relabeled) subsequence it is not restrictive to assume that 
  $z_n \rightharpoonup z$
  for some $z\in X^*.$
  By the very definition of limiting subdifferential there holds
  $z\in \partial_\ell\varphi(x)$,
  and by the weak lower semicontinuity of $\psi^*$, we deduce
  $\psi^*(-z)\le S^-(x),$
  so that $z\in \Dpartial\varphi(x).$

    When $\partial\varphi$
    is strongly-weakly sequentially closed,
    we immediately deduce 
    \eqref{eq:rS-representation}.
 \end{proof}

\subsection{ 
Reduction to a coercive and simpler setting}\label{sec:reduction} 

We first show that the proof of Theorem \ref{thm:main} 
can be carried out under 
three simplifying assumptions:
\begin{enumerate}
    \item the sublevels of 
    $\varphi$ are
    compact and
    the sublevels of $\rS$ are bounded; 
    \item $\varphi$ is nonnegative;
    \item $\Diss$ is nonnegative. 
\end{enumerate}
Taking into accounts the assumptions of Theorem \ref{thm:main}, observe that (1) implies that the function $x\mapsto S^-(x)+a\varphi(x)$ has compact sublevels, retrieving coercivity of the De Giorgi's functional $\ucalJ^a$. Let us first address the first
claim. 
For $R>\|x_0\|$ consider the perturbed energy $\varphi_R(x):=\varphi(x)+\mathrm I_{R}(x)$, where $\mathrm I_{R}$ denotes the indicator function of the \emph{closed} ball of radius $R$ centered at the origin. Clearly $\varphi_R$ still fulfils \ref{hyp:phi1}. 
Moreover, it is easy to see that sublevels of $\varphi_R$ are compact, since $\{\varphi_R\le C\}\subseteq \{\varphi\le C\}\cap \overline{B_{R}}$.
We denote by $\rSR$
    the relaxed slope associated with 
    $\varphi_R$. 
Clearly $\rSR(x)=+\infty$
if $\|x\|>R$ so that 
the sublevels of $\rSR$ are bounded.

 The next Lemma shows that 
it is not restrictive to prove Theorem 
\ref{thm:main} 
for $\varphi_R$.
\begin{lemma}
    There exists a constant $C>0$
    such that 
    every function $\xx\in AC([0,T];X)$ with $\xx(0)=x_0$
    and $\min\{\ucalJ^a_R(T,\xx), \ucalJ^a(T,\xx)\}\le 0$ 
    satisfies $
    \max_{t\in [0,T]}\|\xx(t)\|\le C$. 
    
    In particular, 
    for every $R>C$ 
    and every
    $\xx\in AC([0,T];X)$ with $\xx(0)=x_0$
    \begin{align}
    \label{eq:coincide}
        \min\Big\{\ucalJ^a_R(T,\xx),
        \ucalJ^a(T,\xx)\Big\}\le 0
        \quad&
        \Rightarrow\quad
        \ucalJ^a(T,\xx)=
        \ucalJ^a_R(T,\xx)\le 0
    \end{align}
    so that 
    if
    $\min\Big\{\ucalJ^a_R(T,\xx):
    \xx\in AC([0,T];X),\ 
    \xx(0)=x_0\Big\}\le 0$
    then 
    the sets of minimizers
    of $\ucalJ^a$
    and $\ucalJ^a_R$ coincide and 
    $\min\Big\{\ucalJ^a(T,\xx):
    \xx\in AC([0,T];X),\ 
    \xx(0)=x_0\Big\}\le 0$.
\end{lemma}
\begin{proof}
    By \ref{hyp:F2} we know that for all $M>0$ there exists $c_M>0$ such that $\Diss(v)\ge M\|v\|-c_M$; also using the inequalities $\varphi_R\ge \varphi$ and \eqref{eq:Fboundbelow} (whence $\rS\ge -\Diss(0)$) and recalling \ref{hyp:phi4} we thus deduce 
    \begin{align*}
        & \ucalJ^a_R(T,\xx)\ge e^{-aT}\varphi(\xx(T))+\int_0^T e^{-a\tau}(M\|\dot {\xx}(\tau)\|-c_M-\Diss(0)+a\varphi(\xx(\tau))) \, d\tau-\varphi(x_0)\\
        \ge& M\int_0^T e^{-a\tau}\|\dot {\xx}(\tau)\| \, d\tau -\varphi(x_0) \\&-\left[e^{-aT}(c_2\|\xx(T)\|+c_1)+(c_M+\Diss(0)+ac_1)\int_0^T e^{-a\tau} \, d\tau+ac_2\int_0^T e^{-a\tau}\| \xx(\tau)\| \, d\tau\right].
    \end{align*}
    The very same estimate can be obtained starting from $\ucalJ^a(T,\xx)$. So, in both cases we infer
    \begin{align*}
       &Me^{-aT}\int_0^T \|\dot {\xx}(\tau)\| \, d\tau \\
       \le& \varphi(x_0){+}(c_M{+}|\Diss(0)|{+}ac_1){+}c_1{+}c_2\left(\|x_0\|{+}\int_0^T \|\dot {\xx}(\tau)\| \, d\tau\right)+ac_2T\left(\|x_0\|{+}\int_0^T \|\dot {\xx}(\tau)\| \, d\tau\right)\\
       =&C_M+c_2(1+aT)\int_0^T \|\dot{\xx}(\tau)\| \, d\tau.
    \end{align*}
    By choosing $M:=
     [1+c_2(1+aT)]e^{aT}$  we obtain that $\int_0^T \|\dot {\xx}(\tau)\| \, d\tau\le  C_M$ 
    so that 
    $$\sup_{t\in [0,T]}\|\xx(t)\|\le 
    C
    \quad\text{with }C:=C_M+\|x_0\|.$$

    Let us now choose $R>C$.
    Since for every 
    $x\in X$ with norm $\|x\|<R$
    the Fréchet subdifferential of
    $\varphi$ 
coincides with
the Fréchet subdifferential of $\varphi_R$, it is easy to check that 
\begin{equation*}
    \rSR(x)=\rS(x)\quad\text{for every }x\in X,\ \|x\|<R.
\end{equation*}
It follows that 
$\ucalJ^a_R(\xx)$ and $\ucalJ^a(\xx)$ 
coincide along every absolutely continuous curve $\xx$ starting from
$x_0$ and satisfying
$\sup_{t\in [0,T]}\|\xx(t)\|\le C<R.$
In particular, using the previous estimate, we deduce 
\eqref{eq:coincide}.
\end{proof}

\begin{oss}
    The very same proof shows that the result is still valid under the $p$-growth conditions $(\Diss2')$ and $(\varphi2')$.
\end{oss}

 Let us now 
show that it is not restrictive to assume $\varphi$ nonnegative. 

It is sufficient to observe that the functionals $\ucalJ^a$ 
are invariant under vertical shifts of 
$\varphi$ by constants $\varphi\mapsto \varphi+k$, for $k\in \R$.
Since by the previous argument
we can assume that $\varphi$ has compact sublevels, $\varphi$ 
is bounded from below and therefore, without loss of generality we may also assume that
    \begin{equation}\label{eq:nonneg1}
          \varphi(x)\ge 0\qquad\text{for all }x\in X.
      \end{equation}
We conclude this section by addressing the last claim and showing that 
it is not restrictive to assume that 
$\Diss$ is nonnegative as well.

We first observe that 
a vertical shift of $\Diss$ yields
\begin{equation*}
    (\Diss+k)^*(z)=\Diss^*(z)-k
\end{equation*}
so that 
denoting by $\rSk$ the relaxed slope
of $\varphi$ induced by the shifted
dissipation $\Diss+k$
we have
$$(\Diss+k)(v)+\rSk(x)=
\Diss (v)+\rS(x)\quad\text{for every }
x,v\in X$$
We deduce that 
$\ucalJ^a$ is invariant
by vertical shifts of $\Diss$; 
since $\Diss$ is bounded from below, 
without loss of generality we may also assume that      \begin{equation}\label{eq:nonneg2}
          \Diss(v)\ge 0,\qquad\text{for all }v\in X.
      \end{equation}
  In the proof of Theorem \ref{thm:main}, which we develop in Sections \ref{sec: convexification} and \ref{sec: proof}, we will thus 
      assume 
\eqref{eq:nonneg1}, \eqref{eq:nonneg2},
and that $\varphi$ has compact sublevels.

\section{Narrow topology for Banach-valued measures and action functional}\label{subsec: narrow topology}

In the first part of this section we introduce the proper notion of \emph{narrow topology} for vector-valued measures $\mathcal{M}(Y;X)$, where $(Y,\tau)$ and $(X,\|\cdot\|)$ are, respectively, a Polish space and a reflexive and separable Banach space, and we state important (local) metrizability and compactness results. Then, we consider the so-called \emph{action functional}, namely the lifting of the dissipation integral $\displaystyle \int_0^T \Diss(\dot {\xx}(\tau))\, d\tau$ to spaces of measures, and we discuss its main convexity and lower semicontinuity properties.

\subsection{Narrow topology and compactness result}

\begin{df}\label{def: narrow topology}
    The narrow topology on $\mathcal{M}(Y;X)$ is the smallest topology for which the functionals 
    \begin{equation}\label{eq: narrow cont functions}
    \mathcal{M}(Y;X) \ni \nu\mapsto \int_Y \langle \boldsymbol{\zeta}, d\nu\rangle:= \sum_{i=1}^N \langle z_i, \int_Y  \zeta_i d\nu\rangle
    \end{equation}
    are continuous for all functions $\boldsymbol{\zeta}\colon Y\to X^*$ of the form 
    \begin{equation}\label{eq:psiform}
        \boldsymbol{\zeta}(y) = \sum_{i=1}^N \zeta_i(y)z_i,
    \end{equation}
    for some $N\in \N$, $z_i \in X^*$ and $\zeta_i \in C_b(Y)$, as $i=1,\dots, N$.
    
    The collection of all the functions $\boldsymbol{\zeta}$ satisfying \eqref{eq:psiform} will be denoted by the symbol $\mathfrak{Z}_b(Y;X^*)$.
\end{df}

A first property of such topology, as in the finite-dimensional case, is {metrizability of bounded and uniformly tight sets}. We postpone the proof of the following proposition to Appendix \ref{sect app: narrow}.

\begin{prop}\label{prop:localmetrnarrow}
    The narrow topology over $\mathcal{M}(Y;X)$ is metrizable on bounded {and uniformly tight sets, namely satisfying \eqref{eq: mass bound prohorov} and \eqref{eq: tightness} with respect to the total variation}. Moreover, {for a uniformly tight set, being relatively compact is equivalent to being relatively sequentially compact}.
\end{prop}

The main result of the section is the following compactness criterion in the space $\mathcal{M}_+(Y)\times \mathcal{M}(Y;X)$, endowed with the product narrow topology.

\begin{teorema}\label{thm: compactness for nu}
    Let $\mathcal{G}\subseteq \mathcal{M}_+(Y)\times \mathcal{M}(Y;X)$ be such that:
    \begin{enumerate}
        \item the projection over the first component, namely the family $\mathcal G_1:=\{\mu \in \mathcal{M}_+(Y) \ : \ (\mu,\nu)\in \mathcal{G} \text{ for some } \nu \in \mathcal{M}(Y;X)\}$ has equibounded mass and is uniformly tight, i.e. it satisfies \eqref{eq: mass bound prohorov} and \eqref{eq: tightness};
        \item for each pair $(\mu,\nu)\in \mathcal{G}$, the measure $\nu$ is absolutely continuous with respect to $\mu$;
        \item there exists a function $G:X\to(-\infty,+\infty]$ with superlinear growth at infinity and bounded below, such that
        \begin{equation}\label{eq: weak tightness}
            \sup_{(\mu,\nu)\in \mathcal{G}} \int_Y G\left(\frac{d\nu}{d\mu}(y)\right)d\mu(y)<+\infty. 
        \end{equation}
    \end{enumerate}
    Then, $\mathcal{G}$ is relatively compact in $\mathcal{M}_+(Y) \times \mathcal{M}(Y;X)$, endowed with the product narrow topology.
\end{teorema}

To prove the above theorem, we introduce a (possibly non-complete) metric over $X$, in the spirit of \cite[Section 5.1]{ambrosio2005gradient}, whose induced topology is globally weaker than the weak topology.

\begin{df}
    Let $D= \{z_n\}_{n\in \N}\subset X^*$ be a countable subset of the unit ball such that $\operatorname{Span} D$ is dense in $X^*$. We define $X_{\overline{\omega}}$ as the normed space $(X,\|\cdot \|_{\overline{\omega}})$, where
    \begin{equation*}
        \|x\|_{\overline{\omega}}:= \left(\sum_{n=1}^{+\infty} \frac{\langle  z_n,x \rangle^2}{n^2}\right)^\frac 12.  
    \end{equation*}
\end{df}

Next lemma collects some useful properties about this new space we will exploit afterwards. Its proof follows the line of \cite[Corollary 2, pp. 101]{schwariz1973radon}, to which we refer for details.

\begin{lemma}\label{lemma: properties of X_overline omega}
    The following properties hold:
    \begin{enumerate}[label=(\roman*)]
        \item the strong topology of $X_{\overline \omega}$ is weaker than the weak topology of $X$. Moreover, their generated Borel $\sigma$-algebras coincide;
        \item the topology of $X_{\overline{\omega}}$ restricted to bounded sets (with respect to the original norm of $X$) coincides with the weak topology of $X$. In particular, closed and bounded sets are compact in $X_{\overline{\omega}}$;
        \item the map $x \mapsto \| x \|$ is lower semicontinuous in $X_{\overline \omega}$.
    \end{enumerate}
\end{lemma}

The proof of Theorem \ref{thm: compactness for nu} is based on the following correspondence between the pair $(\mu,\nu) \in \mathcal{M}_+(Y)\times \mathcal{M}(Y;X)$ such that $|\nu|\ll\mu$, and the nonnegative measure $\theta\in\mathcal{M}_+(Y\times X)$ defined as 
\begin{equation}\label{eq:T}
    \theta = T(\mu,\nu) := \left(\operatorname{id} , \frac{d\nu}{d\mu}\right)_\# \mu,
\end{equation}
namely $\theta$ is the pushforward of $\mu$ with respect to the pair $\left(\operatorname{id} , \frac{d\nu}{d\mu}\right)$.

\begin{proof}[Proof of Theorem \ref{thm: compactness for nu}]
    \textbf{Step 1}: we claim that
    \begin{equation}\label{eq: equi-integrability}
        \text{for all }\varepsilon>0 \ \text{there exists }R>0 \text{ such that }  \sup_{(\mu,\nu)\in \mathcal{G}}\int_{\{\|\frac{d\nu}{d\mu}\|> R\}} \left\|\frac{d\nu}{d\mu}(y)\right\| d\mu(y) <\varepsilon.
    \end{equation}
    Indeed, by the superlinearity of $G$, for all $M>0$ there exists $R>0$ such that $G(v)\ge M\|v\|$ for all $v\in X$ with $\|v\|>R$. Recalling that $G$ is also bounded below we deduce that for all $(\mu,\nu)\in \mathcal{G}$ there holds
    \[\int_{\{\|\frac{d\nu}{d\mu}\|> R\}} \left\|\frac{d\nu}{d\mu}\right\| d\mu\leq \frac{1}{M}\int_{\{\|\frac{d\nu}{d\mu}\|> R\}}  G\left(\frac{d\nu}{d\mu}\right)d\mu\le \frac{1}{M}\left(\int_YG\left(\frac{d\nu}{d\mu}\right)d\mu+C\mu(Y)\right).\]
Claim \eqref{eq: equi-integrability} now follows from \eqref{eq: weak tightness} and from the equiboundedness of $\mu$.

    From \eqref{eq: equi-integrability} we infer that for every Borel set $A\subset Y$ and $(\mu,\nu) \in \mathcal{G}$, we have 
    \[|\nu|(A) \leq R(\varepsilon)\mu(A) + \varepsilon.\]
    Then, the choice $\varepsilon = 1$ and $A=Y$ yields equiboundedness of $\nu$, since $\mathcal{G}_1$ is equibounded. Similarly, exploiting uniformly tightness of $\mu$, let $K\subset Y$ be compact and such that $\mu(Y\setminus K) < \varepsilon/R(\varepsilon)$ for every $\mu \in \mathcal{G}_1$. Then it holds that $|\nu|(Y\setminus K) < 2\varepsilon$, namely also $\nu$ is uniformly tight. Thus, by Proposition \ref{prop:localmetrnarrow} it suffices to show that $\mathcal{G}$ is sequentially relatively compact.    
    
    \noindent \textbf{Step 2}: we define the family
    \begin{equation*}
        \Theta := \{T(\mu,\nu) \ : (\mu,\nu)\in \mathcal{G}\},
    \end{equation*} 
    where the operator $T$ has been introduced in \eqref{eq:T}. We claim that $\Theta$ is a relatively compact subset of $\mathcal{M}_+(Y\times X_{\overline{\omega}})$. We prove this fact by using Prokhorov Theorem \ref{thm: prohorov} setting $Z := Y\times X_{\overline{\omega}}$, thus we need to show equiboundedness of mass and uniform tightness:
    \begin{itemize}
        \item for any $\theta \in \Theta$ there holds $\theta(Y\times X)=\mu(Y)$, which is uniformly bounded by assumption;
        \item for all $\varepsilon>0$, consider $K\subset Y$ and $R\ge 1$ satisfying, respectively, \eqref{eq: tightness} for $\mathcal{G}_1$ and \eqref{eq: equi-integrability} with $\varepsilon/2$. Thanks to Lemma \ref{lemma: properties of X_overline omega}, the set $K\times B_R$ is compact in $Y\times X_{\overline{\omega}}$. Furthermore, for all $\theta \in \Theta$ we have
        \begin{align*}
            \theta \big( & (K\times B_R)^c \big) \leq  \theta(K^c \times X) + \theta(Y \times B_R^c) =  \int_{Y \times X} \mathds{1}_{K^c}(y) + \mathds{1}_{B_R^c}(x) \, d\theta(y,x)
            \\
            = &
            \int_Y \mathds{1}_{K^c}(y) + \mathds{1}_{B_R^c}\left( \frac{d\nu}{d\mu}(y) \right)\,            d\mu(y)
            = 
            \mu(Y\setminus K) + \mu\left( \left\{ y\in Y \ : \ \left\|\frac{d\nu}{d\mu}(y)\right\| > R \right\}\right) <\varepsilon.
        \end{align*}
    \end{itemize}
    \textbf{Step 3}: let $\{(\mu_n,\nu_n)\}_{n\in \N}\subseteq \mathcal{G}$ and define $\theta_n := T(\mu_n,\nu_n)$ for all $n\in \N$. By Step 2 there exists $\theta \in \mathcal{M}_+(Y\times X_{\overline{\omega}})$ such that $\theta_n \to \theta$ narrowly in $\mathcal{M}_+(Y\times X_{\overline\omega})$. We now prove that 
    \begin{equation}\label{eq: limit in theta}
       \lim\limits_{n\to +\infty} \int_{Y\times X} \langle \boldsymbol{\zeta}(y), x \rangle d\theta_n(y,x) = \int_{Y \times X} \langle \boldsymbol{\zeta}(y), x \rangle d\theta(y,x), \quad \text{for all } \boldsymbol{\zeta} \in \mathfrak{Z}_b(Y;X^*).
    \end{equation}
    Notice that the map $Y\times X_{\overline{\omega}}\ni(y,x)\mapsto f_{\boldsymbol{\zeta}}(y,x):= \langle  \boldsymbol{\zeta}(y),x \rangle$ is neither bounded nor continuous, so \eqref{eq: limit in theta} is not a direct consequence of narrow convergence. However, it still holds true due to \cite[Proposition 5.1.10]{ambrosio2005gradient}: indeed, recalling Lemma \ref{lemma: properties of X_overline omega}, for all $\varepsilon>0$ there exists $R>0$ such that the following facts hold:
    \begin{itemize}
        \item the restriction of $f_{\boldsymbol{\zeta}}$ to $Y\times B_R$ is continuous with respect to the subspace topology induced by $Y\times X_{\overline{\omega}}$;
        \item $Y\times B_R$ is a closed subset of   $Y\times X_{\overline{\omega}}$;
        \item $\theta_n(Y\times B_R) < \varepsilon$.
    \end{itemize} 
    
    \textbf{Step 4}: we finally show that there exists $(\mu,\nu) \in \mathcal{M}_+(Y)\times \mathcal{M}(Y;X)$ such that $\mu_n \to \mu$ and $\nu_n \to \nu$ narrowly.  Consider the disintegration of $\theta$ with respect to the projection $p^1$ on the first component, obtaining a family of probability measures $\{\theta_y\}_{y\in Y} \subset \PP(X)$ such that 
    \begin{equation}\label{eq:disintegration}
        d\theta(y,x) =  \int\, d\theta_y(x) \, d\mu(y),\quad \text{where } \mu:= p^1_\#\theta.
    \end{equation}
    We then define a measurable map $v\colon Y\to X$ and a measure $\nu\in \mathcal{M}(Y;X),$ as follows:
    \begin{equation*}
        v(y) := \int_X x \ d\theta_y(x), \qquad \nu := v \mu,
    \end{equation*}
    where the integral has to be meant in the sense of Bochner. They are well defined $\mu$-a.e. since
    \[\int_Y \int_X \|x\| d\theta_y(x) d\mu(y) = \int_{Y\times X} \|x\|d\theta(y,x) \leq \limsup_{n\to+\infty} \int_{Y\times X} \|x\| d\theta_n(y,x),\]
    and the latter is finite because, taking $R>0$ relative to $\varepsilon = 1$ in \eqref{eq: equi-integrability}, we have 
    \begin{align*}
    \int_{Y\times X} \|x\| d\theta_n(y,x) = & \int_Y \left\|\frac{d\nu_n}{d\mu_n}\right\|d\mu_n = \int_{\left\{\frac{d\nu_n}{d\mu_n}\in B_R^c \right\}}\left\|\frac{d\nu_n}{d\mu_n}\right\|d\mu_n + \int_{\left\{\frac{d\nu_n}{d\mu_n}\in B_R \right\}}\left\|\frac{d\nu_n}{d\mu_n}\right\|d\mu_n 
    \\
    \leq &
    1+ R\sup_{n\in \N} \mu_n(Y) <+\infty. 
    \end{align*}
    The measurability of $v$ follows from the measurability of the maps $y\mapsto \theta_y$ and $\displaystyle \theta \mapsto \int x d\theta(x)$, see Proposition \ref{lemma: bochner measurability} for details.
    Then, for all $\zeta\in C_b(Y)$ and $\boldsymbol{\zeta}\in \mathfrak{Z}_b(Y;X^*)$, exploiting \eqref{eq:disintegration} and \eqref{eq: limit in theta}, we obtain
    \begin{align*}
        \lim_{n\to +\infty} & \int_{Y}  \zeta(y) d\mu_n(y) = \lim_{n\to +\infty} \int_{Y\times X}  \zeta(y) d\theta_n(y,x)
        \\
        = &
        \int_{Y\times X}  \zeta(y) d\theta(y,x) = \int_Y \zeta(y) d\mu(y),
    \end{align*}
    and
    \begin{align*}
        \lim_{n\to +\infty} & \int_{Y}  \langle\boldsymbol{\zeta}(y), d\nu_n(y)\rangle
        = 
        \lim_{n\to +\infty} \int_{Y\times X}\langle  \boldsymbol{\zeta}(y),x\rangle d\theta_n (y,x) 
        \\
        = & \int_{Y\times X}\langle  \boldsymbol{\zeta}(y),x \rangle d\theta (y,x) 
        = 
        \int_{Y} \left( \int_X\langle \boldsymbol{\zeta}(y),x \rangle d\theta_y(x) \right)d\mu(y) 
        \\
        = &
        \int_{Y} \left\langle \boldsymbol{\zeta}(y), \int_X x \ d\theta_y(x) \right\rangle  d\mu(y) 
        = 
        \int_{Y} \langle  \boldsymbol{\zeta}(y),v(y)\rangle  d\mu(y) 
        = \int_Y \langle\boldsymbol{\zeta}(y) , d\nu(y)\rangle,
    \end{align*}
    and so we conclude.
\end{proof}

\subsection{The action functional}\label{subsec:action}
Let us consider a general functional $\Psi\colon Y\times X\to (-\infty,+\infty]$ satisfying the following assumptions: 
\begin{enumerate}[label=\textup{(${\Psi\arabic*}$)}]
    \item \label{hyp:Fy1} $\Psi(y,\cdot)$ is proper, convex, and lower semicontinuous for all $y\in Y$;
  \item \label{hyp:Fy2} $\Psi(y,v)\ge G(v)$ for all $(y,v)\in Y\times X$, for some $G\colon X\to(-\infty,+\infty]$ bounded below with superlinear growth at infinity;
  \item \label{hyp:Fy3} $\sup\limits_{y\in Y}\Psi(y,0)<+\infty$.  
\end{enumerate}

Notice that, in the particular case in which the functional $\Psi$ does not depend on the variable $y\in Y$, i.e. $\Psi(y,v) = \Diss(v)$, the above set of hypotheses coincides with \ref{hyp:F1}-\ref{hyp:F3}.

Here and henceforth, by $\Psi^*$ we mean the Fenchel conjugate of $\Psi$ with respect to the $v$-variable, namely
\begin{equation*}
    \Psi^*(y,z)=\sup\limits_{v\in X}\{\langle z,v\rangle-\Psi(y,v)\}.
\end{equation*}
Observe that both $\Psi$ and $\Psi^*$ are bounded below, respectively by \ref{hyp:Fy2} and \ref{hyp:Fy3}. Moreover, since $\Psi^*(y,z)\le G^*(z)$, we deduce that $\Psi^*(y,\cdot)$ is locally bounded in $X^*$, uniformly with respect to $y\in Y$. In particular, $\Psi^*(y,\cdot)$ is continuous in $X^*$ for all $y\in Y$. Since we will need some regularity of $\Psi^*$ also with respect to $y\in Y$, it is crucial to require the next condition, which is automatically satisfied in the case $\Psi(y,v)=\Diss(v)$:
\begin{enumerate}[label=\textup{(${\Psi4}$)}]
 \item \label{hyp:Fy4} the map $(y,z)\mapsto \Psi^*(y,z)$ is upper semicontinuous in $Y\times X^*$.
 \end{enumerate}
 Our next proposition shows a sufficient property on the primal dissipation potential $\Psi$ implying assumption \ref{hyp:Fy4}.

\begin{prop}
    Assume \ref{hyp:Fy1}-\ref{hyp:Fy3}. If in addition $\operatorname{dom}\Psi=Y\times \mathsf{D}$ for some $\mathsf{D}\subseteq X$, and the map $y\mapsto \Psi(y,v)$ is lower semicontinuous in $Y$ uniformly on sublevels of $G$, namely: for all $y\in Y$, for all $\varepsilon>0$ and for all $C\in \R$ there exists a neighborhood $U$ of $y$ such that
  \begin{equation}\label{eq:lscunif}
      \sup\limits_{v\in \{G\le C\}\cap \mathsf{D}}(\Psi(y,v)-\Psi(\widetilde y,v))\le \varepsilon,\qquad\text{for all }\widetilde y\in U.
  \end{equation} 
  Then $\Psi^*$ is upper semicontinuous in $Y\times X^*$.
\end{prop}
\begin{proof}
    We first claim that for all $R>0$ there exists $C_R>0$ such that for all $(y,z)\in Y\times B_R^{*}$ there holds
    \begin{equation}\label{claimF*}
        \Psi^*(y,z)=\sup\limits_{v\in \{G\le C_R\}}\{\langle z,v\rangle-\Psi(y,v)\}.
    \end{equation}
    Indeed, let $\{v_n\}_{n\in\N}\subseteq \mathsf{D}$ such that $\Psi^*(y,z)=\lim\limits_{n\to +\infty}(\langle z,v_n\rangle-\Psi(y,v_n))$. Since $\Psi^*$ is bounded below, recalling \ref{hyp:Fy2} we deduce
    \begin{equation}\label{eq:bd}
        -C\le\liminf\limits_{n\to +\infty}(\|z\|_*\|v_n\|-G(v_n))\le \liminf\limits_{n\to +\infty}( R\|v_n\|-G(v_n)).
    \end{equation}
    Due to the superlinear growth of $G$, we know that $G(v)\ge (1+R)\|v\|-c_R$, whence
    \begin{equation*}
        -C\le\liminf\limits_{n\to +\infty}(c_R-\|v_n\|),\qquad\text{i.e. }\qquad \limsup\limits_{n\to +\infty}\|v_n\|\le \widetilde c_R.
    \end{equation*}
    Plugging this bound into \eqref{eq:bd} we infer
    \begin{equation*}
        -C\le\liminf\limits_{n\to +\infty}(R\,\widetilde c_R-G(v_n)),\qquad\text{i.e. }\qquad \limsup\limits_{n\to +\infty}G(v_n)\le \widetilde C_R,
    \end{equation*}
    and so \eqref{claimF*} is proved.

    We now consider a sequence $(y_n,z_n)$ converging to $(y,z)$ in $Y\times X^*$ and we fix $\varepsilon>0$. Choosing $R$ large enough such that both $\|z\|_*\le R$ and $\sup\limits_{n\in\N}\|z_n\|_*\le R$, by combining \eqref{eq:lscunif} and \eqref{claimF*} we observe that for all $v\in \{G\le C_R\}$ definitively there holds
    \begin{align*}
        \Psi^*(y,z_n)\ge \langle z_n,v\rangle-\Psi(y,v)\ge \langle z_n,v\rangle-\Psi(y_n,v)-\varepsilon,
    \end{align*}
    whence we obtain
    \begin{equation*}
        \Psi^*(y,z_n)\ge\Psi^*(y_n,z_n)-\varepsilon.
    \end{equation*}
    Finally, recalling that $\Psi^*(y,\cdot)$ is continuous in $X^*$, we deduce
    \begin{align*}
        \limsup\limits_{n\to +\infty}\Psi^*(y_n,z_n)\le  \limsup\limits_{n\to +\infty}\Psi^*(y,z_n)+\varepsilon=\Psi^*(y,z)+\varepsilon,
    \end{align*}
    and we conclude by sending $\varepsilon$ to zero.
\end{proof}

We now introduce the \emph{action} with respect to $\Psi$ as the functional $\mathcal{A}_\Psi\colon \mathcal{M}_+(Y)\times \mathcal{M}(Y,X)\to (-\infty,+\infty]$ defined as

\begin{equation}\label{eq: action}
    \mathcal{A}_\Psi(\mu,\nu) := \begin{cases}\displaystyle
        \int_Y \Psi\left(y,\frac{d\nu}{d\mu}(y)\right) d\mu(y),
 & \text{ if } |\nu| \ll \mu,
 \\
 +\infty, & \text{ otherwise}.
\end{cases}
\end{equation}

The following representation formula shows that the action functional is convex and lower semicontinuous with respect to the product narrow topology.

\begin{prop}\label{repr formula proposition}
    Assume \ref{hyp:Fy1}-\ref{hyp:Fy4}. For any $(\mu,\nu)\in \mathcal{M}_+(Y)\times \mathcal{M}(Y;X)$ it holds
    \begin{equation}\label{repr formula}
        \mathcal{A}_\Psi(\mu,\nu) = \sup\left\{ \int_Y \langle \boldsymbol{\zeta}(y) , d\nu(y)\rangle - \int_{Y} \Psi^*(y,\boldsymbol{\zeta}(y)) d\mu(y) \ : \ \boldsymbol{\zeta} \in \mathfrak{Z}_b(Y;X^*) \right\}.
    \end{equation}
    In particular, $\mathcal{A}_\Psi$ is convex and lower semicontinuous with respect to the product narrow topology in $\mathcal{M}_+(Y)\times \mathcal{M}(Y;X)$.
\end{prop}

\begin{proof}
The proof follows the line of \cite[Lemma 9.4.4]{ambrosio2005gradient}.  First we prove the $\geq$ inequality. Assume without loss of generality that $\mathcal A_\Psi(\mu,\nu)<+\infty$, so that $\nu = \frac{d\nu}{d\mu} \mu$. Thus, for all $\boldsymbol{\zeta} \in \mathfrak{Z}_b(Y;X^*)$ we have
    \begin{align*}
        &\int_Y  \langle\boldsymbol{\zeta}, d\nu\rangle - \int_Y \Psi^*(y,\boldsymbol{\zeta}) \ d\mu = \int_Y \left(\left\langle \boldsymbol{\zeta}(y),\frac{d\nu}{d\mu}(y)\right\rangle - \Psi^*(y,\boldsymbol{\zeta}(y))\right) d\mu(y) 
         \\\leq& \int_Y \Psi\left(y,\frac{d\nu}{d\mu}(y)\right) d\mu(y) = \mathcal A_\Psi(\mu,\nu),
    \end{align*}
    where we exploited Fenchel's inequality \eqref{Fenchelprop}. 
    
    Regarding the other inequality, let us define ${\widetilde{\mathcal{A}}_\Psi}(\mu,\nu)$ as the right hand side in \eqref{repr formula}. Without loss of generality, assume that ${\widetilde{\mathcal{A}}_\Psi}(\mu,\nu)<+\infty$. We first prove that necessarily $|\nu|\ll\mu$. To this aim, let the Borel set $B\in \mathcal{B}(Y)$ such that $\mu(B) = 0$. Thanks to the inner and outer regularity of finite Borel measures, for any  Borel set $B'\subset B$ there exist compact and open sets $K_n$ and $A_n$ such that $K_n \subseteq B' \subseteq A_n$ and 
    \[\lim\limits_{n\to +\infty}\big(\mu + |\nu|\big)(A_n \setminus K_n)= 0.\]
    Then, for any $n\in \N$ consider a cut-off function $\zeta_n \in C_b(Y)$ such that $0\leq \zeta_n \leq 1$, $\zeta_n \equiv 1$ on $K_n$ and $\zeta_n \equiv 0$ on $A_n^c$. Notice that $\zeta_n \to 0$ in $L^1_\mu(Y)$ and $\zeta_n \to \mathds{1}_{B'}$ in $L^1_{|\nu|}(Y)$. Then, for any $z\in X^*$ and $\lambda>0$, using the test functions $\boldsymbol{\zeta}_n(y) := \lambda\zeta_n(y) z \in \mathfrak{Z}_b(Y;X^*)$, there holds
    \begin{align*}
        {\widetilde{\mathcal{A}}_\Psi}(\mu,\nu)&\geq \lim\limits_{n\to +\infty}\left(\lambda\left \langle z,\int_Y \zeta_n(y) d\nu(y) \right\rangle - \int_Y \Psi^*(y,\lambda\zeta_n(y)z) d\mu(y)\right) \\
        &= \lambda\langle z, \nu(B')\rangle - \int_Y \Psi^*(y,0) d\mu(y).
    \end{align*}
    The passage to the limit in the second integral is allowed by the Dominated Convergence Theorem since $\Psi^*(y,\cdot)$ is continuous and dominated by $G^*$, which is locally bounded. Exploiting again that $\Psi^*(y,z)\le G^*(z)$, for any $z\in X^*$ and any $\lambda >0$ we thus deduce
    \begin{equation*}
        \langle z, \nu(B')\rangle\le \frac{\widetilde{\mathcal{A}}_{\Psi}(\mu,\nu) + G^*(0)\mu(Y)}{\lambda},
    \end{equation*}
    which implies that $\langle z,  \nu(B') \rangle = 0$ for any $z\in X^*$ and so $\nu(B')=0$. By arbitrariness of $B'\subset B$, we finally conclude that $|\nu| \ll \mu$.

    We can now write 
    \begin{align*}
        {\widetilde{\mathcal{A}}_\Psi}(\mu,\nu) = \sup\left\{\int_Y \left(\left\langle \boldsymbol{\zeta}(y) ,\frac{d\nu}{d\mu}(y) \right\rangle - \Psi^*(y,\boldsymbol{\zeta}(y))\right) d\mu(y) \ : \ \boldsymbol{\zeta} \in \mathfrak{Z}_b(Y;X^*)\right\}.
    \end{align*}
    On the other hand, considering a countable and dense subset $\{z_n\}_{n\in\N}$ of $X^*$, by continuity of $\Psi^*(y,\cdot)$ we have 
    \begin{align*}
        \mathcal{A}_\Psi(\mu,\nu) = & \int_Y \sup_{j\in \N} \left(\left\langle z_j, \frac{d\nu}{d\mu}(y) \right\rangle - \Psi^*(y,z_j)\right) d\mu(y)
        \\
        = & \lim_{k\to +\infty} \int_Y \sup_{j\leq k} \left(\left\langle  z_j,\frac{d\nu}{d\mu}(y) \right\rangle - \Psi^*(y,z_j)\right) d\mu(y).
    \end{align*}
    Then, in order to conclude it suffices to prove that
    \begin{equation}\label{eq:thesis}
        {\widetilde{\mathcal{A}}_\Psi}(\mu,\nu) \geq \int_Y \sup_{j\leq k} \left(\left\langle z_j, \frac{d\nu}{d\mu}(y)  \right\rangle - \Psi^*(y,z_j)\right) d\mu(y),\quad \text{for all $k\in \N$.} 
    \end{equation}
    To this aim, we fix $k\in \N$ and set 
    \[A_1:= A_1', \quad A_j:= A_j' \setminus \bigg(\bigcup_{i\leq j-1} A_i'\bigg) \ \text{ for }j=2,\dots,k,\]
    where
    \[A_j' := \left\{y \in Y \ : \ \left\langle  z_j, \frac{d\nu}{d\mu}(y) \right\rangle - \Psi^*(y,z_j) \geq \left\langle z_i \frac{d\nu}{d\mu}(y)\right\rangle - \Psi^*(y,z_i) \ \text{ for all } i\leq k\right\}.\]
    Now, for all $\varepsilon>0$ and $j\le k$, consider $K_j\subseteq A_j \subseteq O_j$ such that $K_j$ is compact, $O_j$ is open, $O_j \cap K_i = \emptyset$ for all $i\neq j$ and 
    \begin{equation}\label{111}
        \sum_{j=1}^k \mu(O_j \setminus K_j) + |\nu|(O_j \setminus K_j) \leq \varepsilon.
    \end{equation}
    This can be done by exploiting inner and outer regularity of $\mu$ and $|\nu|$ in order to first have $K_j \subset A_j \subset O_j'$ such that $K_j$ is compact, $O_j$ is open and \eqref{111} is satisfied with $O_j'$ replacing $O_j$. Then one simply defines 
    \[O_j := O_j' \cap \bigcup_{i\neq j} K_i^c.\]
    Now, for any $j \leq k$ consider a cut-off function $\zeta_j \in C_b(Y)$ such that $0 \leq \zeta_j \leq 1$, $\zeta_j \equiv 1$ on $K_j$ and $\zeta_j \equiv 0 $ on $O_j^c$, and define 
    \[\boldsymbol{\zeta}(y):= \sum_{j=1}^k \zeta_j(y) z_j \in \mathfrak{Z}_b(Y;X^*).\]
    Setting $M:= \sup_{j\leq k} \|z_j\|_*$ and $\alpha\ge 0$ so that $\sup\limits_{y\in Y}|\Psi^*(y,z)| \leq \alpha$ for $\|z\|_*\leq M$, we obtain 
    \begin{align*}
       & \int_Y  \sup_{n\leq k} \left(\left\langle z_n, \frac{d\nu}{d\mu}(y) \right\rangle - \Psi^*(y,z_n)\right) d\mu(y) = \sum_{j=1}^k \int_{A_j} \left(\left\langle z_j, \frac{d\nu}{d\mu}(y) \right\rangle - \Psi^*(y,z_j) \right)d\mu(y)
        \\
        \leq  &
        \sum_{j=1}^k \left[\int_{K_j} \left(\left\langle z_j, \frac{d\nu}{d\mu}(y) \right\rangle - \Psi^*(y,z_j)\right) d\mu(y) + \int_{O_j \setminus K_j} \left(\|z_j\|_* \left\|\frac{d\nu}{d\mu}(y) \right\| + |\Psi^*(y,z_j)| \right)d\mu(y)\right]
        \\
        \leq &
        \sum_{j=1}^k \left[\int_{K_j} \left(\left\langle \zeta_j(y)z_j, \frac{d\nu}{d\mu}(y)\right\rangle - \Psi^*(y,\zeta_j(y)z_j)\right) d\mu(y)\right] + \varepsilon(M+\alpha)
        \\
        = &
        \int_Y\!\left(\!\left\langle \boldsymbol{\zeta}(y),  \frac{d\nu}{d\mu}(y) \right\rangle {-} \Psi^*(y,\boldsymbol{\zeta}(y))\right)\! d\mu(y) - \int_{Y\setminus\bigcup_{j=1}^k K_j} \!\!\!\left(\!\left\langle \boldsymbol{\zeta}(y), \frac{d\nu}{d\mu}(y) \right\rangle {-} \Psi^*(y,\boldsymbol{\zeta}(y))\right)\! d\mu(y) \\&+ \varepsilon(M+\alpha).
    \end{align*}
    Also, observe that
    \begin{align*}
        &\bigg|\int_{Y\setminus\bigcup_{j=1}^k K_j}  \left(\left\langle  \boldsymbol{\zeta}(y), \frac{d\nu}{d\mu}(y) \right\rangle - \Psi^*(y,\boldsymbol{\zeta}(y))\right) d\mu(y)\bigg| \\
        \leq & \sum_{j=1}^k \int_{O_j\setminus K_j} \left|  
        \left\langle \boldsymbol{\zeta}(y), \frac{d\nu}{d\mu}(y) \right\rangle - \Psi^*(y,\boldsymbol{\zeta}(y)) \right| d\mu(y) \\
        \leq &  \sum_{j=1}^k M|\nu|(O_j\setminus K_j) + \alpha \mu(O_j\setminus K_j) \leq \varepsilon(M+\alpha).
    \end{align*}
    Putting all together we infer
    \begin{align*}
        &\int_Y  \sup_{n\leq k} \left(\left\langle z_n, \frac{d\nu}{d\mu}(y) \right\rangle - \Psi^*(y,z_n)\right) d\mu(y)\\
        \le & \int_Y\left(\left\langle \boldsymbol{\zeta}(y),  \frac{d\nu}{d\mu}(y) \right\rangle - \Psi^*(y,\boldsymbol{\zeta}(y))\right) d\mu(y)+2\varepsilon (M+\alpha)\\
        \le & \widetilde{\mathcal A}_\Psi(\mu,\nu)+2\varepsilon (M+\alpha),
    \end{align*}
    and letting $\varepsilon \to 0$ we obtain \eqref{eq:thesis} and we conclude.
\end{proof}

We now prove a sort of dual formula to \eqref{repr formula}, which we will exploit in Proposition \ref{prop:dualproblem}

\begin{lemma}\label{lemma: dual repr for F^*}
    Assume \ref{hyp:Fy1}-\ref{hyp:Fy4}. For all $\boldsymbol{\zeta} \in C_b(Y;X^*)$ and $\mu \in \mathcal{M}_+(Y)$, it holds
    \begin{equation*}
        \int_Y \Psi^*(y,\boldsymbol{\zeta}(y)) d\mu(y) = \sup \left\{ \int_Y \left(\langle \boldsymbol{\zeta}(y), v(y) \rangle - \Psi(y,v(y))\right) d\mu(y) \ : \ v\in L^1_\mu(Y;X) \right\}.
    \end{equation*}
\end{lemma}

\begin{proof}
    The $\geq$ inequality follows from the pointwise inequality $\Psi^*(y,\boldsymbol{\zeta}(y)) \geq \langle \boldsymbol{\zeta}(y),v(y) \rangle - \Psi(y,v(y))$. For the other inequality we instead argue similarly to the proof of Proposition \ref{repr formula proposition}. For all $y\in Y$, consider the epigraph of $\Psi(y,\cdot)$, i.e.
    \[\operatorname{epi}\Psi(y,\cdot):= \left\{ (v,r) \in X\times \R \ : \ \Psi(y,v) \leq r \right\},\]
    which is a closed and convex subset of $X \times \R$, since $\Psi(y,\cdot)$ is lower semicontinuous and convex. It is also straightforward to check that the map $y\mapsto \mathrm{I}_{\operatorname{epi}\Psi(y,\cdot)}(v,r)$ is lower semicontinuous in $Y$ for any fixed $(v,r)$. By definition of $\Psi^*$, it immediately follows that 
    \[\Psi^*(y,z) = \sup_{(v,r)\in\operatorname{epi}\Psi(y,\cdot)} \{\langle z , v \rangle - r\}.\]
    Let now $\{ v_j, r_j \}_{j\in \N} $ be a dense and countable subset of $X\times \mathbb R$. Then
    \begin{equation}\label{eq:final}
    \begin{aligned}
        \int_Y \Psi^*(y,\boldsymbol{\zeta}(y)) d\mu(y) &= \int_Y \sup_{j\in \N} (\langle \boldsymbol{\zeta}(y),v_j \rangle - r_j-\mathrm{I}_{\operatorname{epi}\Psi(y,\cdot)}(v_j,r_j)) \ d\mu(y)
    \\
    &=
    \lim_{k\to +\infty} \int_Y \sup_{j\leq k} (\langle \boldsymbol{\zeta}(y),v_j \rangle - r_j -\mathrm{I}_{\operatorname{epi}\Psi(y,\cdot)}(v_j,r_j))d\mu(y).
    \end{aligned}       
    \end{equation}
    Then, let 
    \[B_1:= B_1', \quad B_j:= B_j' \setminus \bigg(\bigcup_{i\leq j-1} B_i'\bigg) \ \text{ for }j=2,\dots,k,\]
    where
    \begin{align*}
        B_j' := \{y \in Y \ : \ \langle  \boldsymbol{\zeta}(y),v_j \rangle - r_j-&\mathrm{I}_{\operatorname{epi}\Psi(y,\cdot)}(v_j,r_j) \geq\\& \langle  \boldsymbol{\zeta}(y),v_i \rangle - r_i-\mathrm{I}_{\operatorname{epi}\Psi(y,\cdot)}(v_i,r_i) \ \text{ for all } i\leq k\},
    \end{align*}
    and define $v_k(y) := \sum_{j=1}^k \mathds{1}_{B_j}(y) v_j \in L^1_\mu(Y;X)$. Notice that in $B_j'$ it necessarily holds $(v_j,r_j)\in \operatorname{epi}\Psi(y,\cdot)$, namely $\Psi(y,v_j)\le r_j$. Now, we compute
    \begin{align*}
        \int_Y & \left(\langle \boldsymbol{\zeta}(y),v_k(y) \rangle - \Psi(y,v_k(y)\right) d\mu(y) = \sum_{j=1}^k \int_{B_j} \left(\langle \boldsymbol{\zeta}(y),v_j \rangle - \Psi(y,v_j)\right) d\mu(y)
        \\
        =&\sum_{j=1}^k \int_{B_j} \left(\langle \boldsymbol{\zeta}(y),v_j \rangle - \Psi(y,v_j) -\mathrm{I}_{\operatorname{epi}\Psi(y,\cdot)}(v_j,r_j)\right) d\mu(y)\\
        \geq & 
        \sum_{j=1}^k \int_{B_j} \left(\langle \boldsymbol{\zeta}(y),v_j \rangle - r_j -\mathrm{I}_{\operatorname{epi}\Psi(y,\cdot)}(v_j,r_j)\right) d\mu(y) \\=& 
        \sum_{j=1}^k \int_{B_j} \sup_{i\leq k} \left(\langle  \boldsymbol{\zeta}(y),v_i \rangle - r_i-\mathrm{I}_{\operatorname{epi}\Psi(y,\cdot)}(v_i,r_i)\right) d\mu(y) 
        \\
        = & \int_Y \sup_{i\leq k}\left(\langle  \boldsymbol{\zeta}(y), v_i\rangle - r_i-\mathrm{I}_{\operatorname{epi}\Psi(y,\cdot)}(v_i,r_i)\right) d\mu(y),
    \end{align*}
    and we conclude by \eqref{eq:final}.
\end{proof}

\section{Convexification in spaces of measures}\label{sec: convexification}

This section contains the first step in the proof of Theorem \ref{thm:main}, namely the relaxation procedure of the functional $\ucalJ^a(T,\cdot)$ to spaces of measures, in order to make the problem convex.

We recall that, as we showed in Section \ref{sec:reduction}, without loss of generality we can assume that 
the energy $\varphi$ has \emph{compact} sublevels
and it is nonnegative (see \eqref{eq:nonneg1}), beyond of course \ref{hyp:phi1}. Moreover, besides \ref{hyp:F1}-\ref{hyp:F3}, we can also assume that the dissipation potential $\Diss$ is nonnegative (see \eqref{eq:nonneg2}). 
 Finally, 
when $a=0$, 
we also assume 
that $\rS$ has compact sublevels. 
Here and henceforth, we tacitly assume these stronger assumptions.

In order to state the main result of this section, for the sake of clarity, it is useful to introduce the following notation. We define
\begin{equation}\label{eq:Btilde}
    \widetilde B:=\mathcal M_+(X)\times \mathcal M_+(X_T)\times\mathcal M(X_T;X),
\end{equation}
 and we consider the functional $\mathcal E^a\colon \widetilde B\to (-\infty,+\infty]$ defined as 
\begin{equation}\label{eq:Ea}
    \mathcal E^a(m,\mu,\nu) := e^{-aT}\int_X \varphi(x) dm(x) +\mathcal A^a_\Diss(\mu,\nu)+ \int_{X_T}e^{-at} \big(\rS (x)+a\varphi(x)\big) d\mu(t,x)-\varphi(x_0),
\end{equation}
where, with a slight abuse of notation, we denote by $\mathcal{A}^a_\Diss$ the action functional introduced in \eqref{eq: action} with $Y=X_T$ and $\Psi(t,x,v) = e^{-at}\Diss(v)$, which fulfils all the required properties \ref{hyp:Fy1}-\ref{hyp:Fy4}. Notice that we have kept the notation of Definition \ref{def:CEsol}, indeed the continuity equation \eqref{eq:conteq} will appear as a constraint in the relaxed problem.

\begin{teorema}\label{thm:convexification}
    For all $a\geq 0$, the optimization problem
    \begin{equation}\label{eq:probcurv}
        \inf\left\{\ucalJ^a(T,\xx):\, \xx\in AC([0,T];X),\, \xx(0)=x_0\right\}
    \end{equation}
    is equivalent to the relaxed one
    \begin{equation}\tag{$\mathcal M$}\label{relaxation CE}
    \inf \left\{ \mathcal{E}^a(m,\mu,\nu) \ : \ (m,\mu,\nu) \in \widetilde B,\,(\mu,\nu)\in\CE(\delta_{x_0},m) \right\},
\end{equation}
    where we recall that we write $(\mu,\nu)\in\CE(\delta_{x_0},m)$ if the continuity equation is satisfied in the sense of \eqref{eq:CE}.
\end{teorema}

In order to prove the above result it is first convenient to rewrite the problem on curves \eqref{eq:probcurv} as
\begin{equation}\tag{$\mathcal C$}\label{main problem}
    \inf\left\{ E^a(\xx) \ : \ \xx  \in C([0,T];X), \ \xx(0) = x_0 \right\},
\end{equation}
where we set
\begin{equation*}
    E^a(\xx) := \begin{cases}\displaystyle
    \ucalJ^a(T,\xx), \quad & \text{ if }\xx \in AC([0,T];X),
    \\
    +\infty, & \text{ otherwise}.
    \end{cases}
\end{equation*}

We then introduce a third auxiliary problem in which, now, one minimizes over probability measures on curves, namely: 
\begin{equation}\label{eq: third problem}\tag{$\mathcal{P}\!\mathcal{C}$}
    \inf\left\{ \ \mathfrak{E}^a(\lambda) \ : \ \lambda \in \PP(C([0,T];X)), \ (e_0)_\# \lambda = \delta_{x_0}\right\}, 
\end{equation}
where 
\begin{equation}\label{eq: mathfrak E}
   \mathfrak{E}^a(\lambda) := \int_{C([0,T],X)} E^a(\xx ) d\lambda(\xx).
\end{equation}
We recall that for any $t\in [0,T]$ the evaluation map $e_t:C([0,T];X) \to X$ is defined as 
$e_t(\xx) = \xx(t)$, and the symbol $f_\#\mu$ denotes the pushforward of the measure $\mu$ through the function $f$.

Theorem \ref{thm:convexification} is a byproduct of the following result.

\begin{prop}
    The optimization problems \eqref{main problem}, \eqref{eq: third problem} and \eqref{relaxation CE} are equivalent. Namely, the optimal value of the three problems coincide and if the minimum is attained in one of them, then it is in the others as well.
\end{prop}

In turn, the above proposition follows by the following two lemmas, which assess, respectively, the equivalence between \eqref{main problem} and \eqref{eq: third problem}, and between \eqref{eq: third problem} and \eqref{relaxation CE}. In particular, in Lemma \ref{lemma:equiv} we will employ a suitable version of the so-called \emph{superposition principle}, valid in general Banach spaces, in order to lift solutions to the continuity equation to probability measures on curves. We refer the Reader to Appendix \ref{app: CE}, specifically to Theorem \ref{thm: superposition Banach}, for more details.

\begin{lemma}
    Let $x_0 \in X$. If $\xx\in C([0,T];X)$ is a competitor for \eqref{main problem}, then $\delta_{\xx}$ is a competitor for \eqref{eq: third problem} and $\mathfrak{E}^a(\delta_\xx) = E^a(\xx)$. Conversely, if $\lambda\in \PP(C([0,T];X))$ is a competitor for \eqref{eq: third problem}, then it is supported over curves that are competitors for \eqref{main problem} and there exists a curve $\xx \in \operatorname{supp}\lambda$ such that $E^a(\xx)\leq \mathfrak{E}^a(\lambda)$. In particular, the optimal values of the two problems coincide.
    \\
    Moreover, if $\xx_{\operatorname{\min}}$ is a minimum for \eqref{main problem}, then $\delta_{\xx_{\operatorname{min}}}$ is a minimum for $\eqref{eq: third problem}$, and if $\lambda_{\operatorname{min}}$ is a minimum for \eqref{eq: third problem}, then it is concentrated over curves realizing the minimum of \eqref{main problem}.
\end{lemma}

\begin{proof}
    If $\xx \in C([0,T];X)$ is a competitor for \eqref{main problem}, then $(e_0)_\#\delta_\xx = \delta_{\xx(0)} = \delta_{x_0}$, and it is trivial to verify that $\mathfrak{E}^a(\delta_\xx) = E^a(\xx)$. If $\lambda \in \PP(C([0,T];X))$ satisfies $(e_0)_\#\lambda = \delta_{x_0}$, then its is supported over curves $\xx \in C([0,T];X)$ with $\xx(0) = x_0$, and in particular on competitors for \eqref{main problem}. By the very definition \eqref{eq: mathfrak E}, it must hold $\lambda\big( \left\{\xx \ : \ E^a(\xx)> \mathfrak{E}^a(\lambda)\right\} \big)<1$, which implies that there exists a curve in the support of $\lambda$ such that $E^{a}(\xx)\leq \mathfrak{E}^a(\lambda)$. 

    To conclude, if \eqref{main problem} admits a minimum $\xx_{\operatorname{min}}$, then $\lambda = \delta_{\xx_{\operatorname{min}}}$ is a minimum for \eqref{eq: third problem} as well, arguing exactly as before. Finally, if $\lambda_{\operatorname{min}}$ is a minimum for \eqref{eq: third problem}, then $\lambda\big( \left\{\xx \ : \ E^a(\xx)> \mathfrak{E}^a(\lambda)\right\} \big) = 0$, otherwise we could find a curve in its support for which the value is strictly lower then the minimum. This concludes the proof.
\end{proof}

\begin{lemma}\label{lemma:equiv}
    Let $x_0 \in X$. Then:
    \begin{enumerate}
        \item for any $(m,\mu,\nu) \in \operatorname{dom} \mathcal{E}^a$ such that $(\mu,\nu)\in\CE(\delta_{x_0},m)$ there exists $\lambda \in \PP(C([0,T];X))$ satisfying $(e_0)_\#\lambda = \delta_{x_0}$ and $\mathfrak{E}^a(\lambda) = \mathcal{E}^a(m,\mu,\nu)$;
        \item for any $\lambda \in \PP(C([0,T];X))$ such that $(e_0)_\#\lambda = \delta_{x_0}$ and $\int E^a(\xx) d\lambda(\xx) <+\infty$, there exists a triple $(m,\mu,\nu) \in \widetilde B$ such that $(\mu,\nu)\in\CE(\delta_{x_0},m)$ and $\mathcal{E}^a(m,\mu,\nu) \leq \mathfrak{E}^a(\lambda)$.
    \end{enumerate}
    In particular, the optimal values of the two problems \eqref{eq: third problem} and \eqref{relaxation CE} coincide, and \eqref{eq: third problem} attains its minimum if and only if \eqref{relaxation CE} does. 
\end{lemma}

\begin{proof}
    $(1)$ Since $\mathcal{A}^a_\Diss(\mu,\nu)<+\infty$, in particular it holds $|\nu|\ll\mu$. Then, we can apply the superposition principle Theorem \ref{thm: superposition Banach} to obtain a measure $\lambda \in \PP(C([0,T];X))$ satisfying 
    \begin{equation*}
        (e_0)_\#\lambda = \delta_{x_0},\quad (e_T)_\#\lambda = m,\quad\text{and}\quad\mathrm{e}_\#(\mathcal{L}^1|_{[0,T]}\otimes\lambda )= \mu,
    \end{equation*}
    where $\mathrm{e}:[0,T]\times C([0,T];X)  \to X_T$ is defined as
    \begin{equation*}
        \mathrm{e}(t,\xx):= (t,\xx(t)).
    \end{equation*}
    Moreover, $\lambda$ is concentrated over absolutely continuous curves solving the equation 
    \begin{equation*}
        \dot{\xx}(t) = \frac{d\nu}{d\mu}(t,\xx(t)),\quad \text{ for a.e. } t\in[0,T].
    \end{equation*}
    Then, simple computations show that 
    \begin{align*}
        &\mathfrak{E}^a(\lambda)\\ 
        =& \int \!\left(e^{-aT}\varphi(\xx(T))+\int_0^T\!\!\!\! e^{-a t}\left(\Diss\left(\frac{d\nu}{d\mu}(t,\xx(t))\right) + \rS (\xx(t))+a\varphi(\xx(t))\right) dt \right)\! d\lambda(\xx)-\varphi(x_0) 
        \\
        = &
        \int_X \!\!e^{-aT}\varphi(x) dm(x) + \mathcal{A}_\Diss^a(\mu,\nu) + \int_{X_T} \!\!e^{-at}\big(\rS (x)+a\varphi(x)\big) d\mu(t,x)-\varphi(x_0)
        =  \mathcal{E}(\mu,\nu,m).
    \end{align*}
    (2) Let us first set 
    \begin{equation}\label{eq:mmu}
        \mu := \mathrm{e}_\# (\mathcal{L}^1\otimes\lambda) \quad \text{and} \quad m=\mu(T).
    \end{equation}
    In order to define a proper vector-valued measure $\nu$, we also define a Borel measurable vector field $v \in L^1_\mu(X_T,X)$ such that the continuity equation $\partial_t\mu + \operatorname{div}(v \mu) = 0$ is satisfied (in the sense of Definition \ref{continuity equation over X def}). Let $\{\lambda_{t,x}\}\subset \PP(C([0,T];X))$ be the disintegration of $\mathcal{L}^1_{[0,T]}\otimes\lambda$ with respect to the map $\mathrm{e}$, i.e. each $\lambda_{t,x}$ is concentrated over curves that at time $t$ pass through $x$ and 
    \[\mathcal{L}^1_{[0,T]}\otimes\lambda = \int \lambda_{t,x}\, d\mu(t,x).\]
    Using Bochner integration, define the vector field 
    \begin{equation}\label{eq:v}
        v(t,x) := \int \dot{\xx}(t) d\lambda_{t,x}(\xx),
    \end{equation}
    which is Borel measurable (see Proposition \ref{lemma: bochner measurability}) and well-defined for $\mu$-a.e. $(t,x)$ since 
    \[\int \int|\dot{\xx}(t)| d\lambda_{t,x}(\xx)d\mu(t,x) = \int \int_0^T \|\dot{\xx}(t)\| dtd\lambda(\xx)  \leq C\left( \int \int_0^T \Diss(\dot{\xx}(t)) dt d\lambda(\xx) + 1\right)<+\infty,\]
    where we exploited the superlinearity of $\Diss$.
    
    Then, the continuity equation $\partial_t\mu + \operatorname{div}(v\mu) = 0$ is automatically satisfied by \eqref{eq:mmu} and \eqref{eq:v}, indeed for any $\xi \in \operatorname{Cyl}_b(X_T)$ it holds
    \begin{align*}
        \int_{X_T} & \partial_t \xi d\mu + \int_{X_T} \langle \mathrm D_x\xi, v\rangle d\mu - \int_X \xi(T,x) dm(x) + \int_X \xi(0,x) d\delta_{x_0}(x)
        \\
        = &
       \int  \left[\int_0^T \partial_t \xi(t,\xx(t)) + \langle \mathrm D_x \xi (t,\xx(t)), \dot{\xx}(t) \rangle dt - \xi(T,\xx(T)) + \xi (0,\xx(0)) \right] d\lambda(\xx)
       \\
       = &
       \int \left[ \int_0^T \frac{d}{dt} \xi(t,\xx(t)) dt - \xi(T,\xx(T)) + \xi (0,\xx(0)) \right] d\lambda(\xx) = 0.
    \end{align*}
   Setting $\nu:=v\mu$, we are left to show that $\mathcal{E}(\mu,\nu,m) \leq \mathfrak{E}^a(\lambda)$. This follows by simple computations: 
    \begin{align*}
        &\mathcal{E}^a(\mu,\nu,m) \\
        =&  \int_X e^{-aT}\varphi(x) dm(x) + \int_{X_T} e^{-at}\left(\Diss\left(v(t,x)\right)+\rS (x)+a\varphi(x)\right) d\mu(t,x)-\varphi(x_0)
        \\
        = & 
        \int \left[ e^{-aT}\varphi(\xx(T)) + \int_0^T e^{-at}\left(\rS (\xx(t)) + a\varphi(\xx(t))\right)
        \right]d\lambda(\xx) 
        \\
        & + \int e^{-at}\Diss \left( \int \dot{\xx}(t)d\lambda_{t,x}(\xx) \right) d\mu(t,x)-\varphi(x_0)
        \\
        \le & 
        \int \left[ e^{-aT}\varphi(\xx(T)) + \int_0^T e^{-at}\left(\rS (\xx(t)) + a\varphi(\xx(t))\right)
        \right]d\lambda(\xx) 
        \\
        & + \int \int e^{-at}\Diss (\dot{\xx}(t))d\lambda_{t,x}(\xx)  d\mu(t,x)-\varphi(x_0)
        \\
        = & 
        \int \left[ e^{-aT}\varphi(\xx(T)) + \int_0^T e^{-at}\left(\rS (\xx(t)) + a\varphi(\xx(t)) + \Diss(\dot{\xx}(t))\right)
        \right]d\lambda(\xx)-\varphi(x_0)
        =  \mathfrak{E}^a(\lambda),
    \end{align*}
    where we used Jensen's inequality, the fact that $\lambda_{t,x}$ is the disintegration of $\mathcal{L}^1_{[0,T]}\otimes \lambda$ with respect to $\mathrm{e}$ and since $\mathrm{e}_\#(\mathcal{L}^1_{[0,T]}\otimes\lambda) = \mu$.
\end{proof}

\section{Minimax theorem and dual problem}\label{sec: proof}

In this section we conclude the proof of Theorem \ref{thm:main}, namely we show, accordingly to Theorem \ref{thm:convexification}, that problem \eqref{relaxation CE} admits a minimizer and that the optimal value is nonpositive. Actually, we slightly extend the result by replacing the initial measure $\delta_{x_0}$ with a general $\mu_0\in \mathcal M_+(X)$ such that $\varphi\in L^1_{\mu_0}(X)$. Consistently, we replace the last term in the definition \eqref{eq:Ea} of $\mathcal E^a$ by $\int_X\varphi d\mu_0$. Again, we recall that we are assuming that $\varphi$ has compact sublevels and that both $\varphi$ and $\Diss$ are nonnegative.
Finally, when $a=0$,
we will also assume that $\rS$
has compact sublevels.

This simple proposition first states that the minimum in problem \eqref{relaxation CE} exists.
\begin{prop}\label{prop:minE}
    Given $\mu_0\in \mathcal M_+(X)$, if there exists 
    $(\bar\mu,\bar \nu,\bar m)\in \mathrm{dom}\mathcal E^a$
    with $(\bar\mu,\bar\nu)\in \CE(\mu_0,m)$,
    then the minimum 
    \begin{equation*}
        \min\left\{\mathcal E^a(m,\mu,\nu):\, (m,\mu,\nu)\in \widetilde B,\, (\mu,\nu)\in \CE(\mu_0,m)\right\}.
    \end{equation*}
    is attained.
\end{prop}
\begin{proof}
    It follows by the direct method of Calculus of Variations. The lower semicontinuity of $\mathcal E^a$ with respect to the product narrow topology in $\widetilde B$ is granted by the lower semicontinuity and boundedness from below of $\varphi$ and $\rS  $ and by Proposition~\ref{repr formula proposition}. Moreover, the constraint $(\mu,\nu)\in \CE(\mu_0,m)$ is closed in the product narrow topology. Indeed, in \eqref{eq:CE} one tests against cylinder functions $\xi\in \operatorname{Cyl}_b(X_T)$, which satisfy $\partial_t \xi \in C_b(X_T)$ and $\mathrm D_x\xi \in \mathfrak{Z}_b(X_T;X^*)$, using the notation of Definition \ref{def: narrow topology}.

    Let now consider the sublevel $\Sigma_C:=\left\{(m,\mu,\nu)\in \widetilde B:\,\mathcal E^a(m,\mu,\nu)\le C,\, (\mu,\nu)\in \CE(\mu_0,m)\right\}$. Then, for any triple $(m,\mu,\nu)\in \Sigma_C$ we infer
    \begin{itemize}
        \item $\mu(X_T)+m(X)=(T+1)\mu_0(X)$, by using $\xi(t,x)=t-T-1$ in \eqref{eq:CE};
        \item $\int_X \varphi d m\le e^{aT}\left(C+\int_X\varphi\, d\mu_0+\Diss(0)\mu(X_T)\right)$, whence we deduce tightness in the $m$ component since $\varphi$ has compact sublevels;
        \item $\int_{X_T} e^{-at}(a\varphi+\rS  ) d \mu\le C+\int_X\varphi\, d\mu_0$, whence we deduce tightness in the $\mu$ component since the integrand has compact sublevels;
        \item $\mathcal A_\Diss(\mu,\nu)\le e^{aT}\mathcal A^a_\Diss(\mu,\nu)\le e^{aT}\left(C+\int_X\varphi\, d\mu_0+\Diss(0)\mu(X_T)\right)$.
    \end{itemize}
    By means of Theorem~\ref{thm: compactness for nu} we finally obtain that $\Sigma_C$ is compact with respect to the product narrow topology, and we conclude.
\end{proof}

We now want to rewrite problem \eqref{relaxation CE} as a saddle-point problem including the constraint into the functional, with the aim of applying a minimax theorem and then passing to its dual problem. For simplicity, we thus set
\begin{equation*}
    A:=\operatorname{Cyl}_b(X_T),
\end{equation*}
and we introduce the Lagrangian $\mathcal L^a\colon A\times \widetilde B \to(-\infty,+\infty]$ defined as
\begin{equation}\label{eq:lagrangian}
    \mathcal L^a(\xi, m,\mu,\nu):=\mathcal E^a( m,\mu,\nu)+\mathcal C(\xi, m,\mu,\nu),
\end{equation}
where, as in \eqref{eq:constraintintro}, we consider 
\begin{equation*}
    \mathcal C(\xi, m,\mu,\nu):= \int_{X_T} \partial_t \xi d\mu + \int_{X_T} \langle \mathrm D_x\xi, d\nu\rangle  - \int_X \xi(T) dm + \int_X \xi(0) d\mu_0.
\end{equation*}

Notice that $(\mu,\nu)\in \CE(\mu_0,m)$ if and only if $\mathcal C(\xi, m,\mu,\nu)=0$ for all $\xi\in A$. Since the map $\xi\mapsto \mathcal C(\xi, m,\mu,\nu)$ is linear, this yields that for all $(m,\mu,\nu)\in \widetilde B$ there holds
\begin{equation*}
    \sup\limits_{\xi\in A}\mathcal L^a(\xi, m,\mu,\nu)=\begin{cases}
        \mathcal E^a(m,\mu,\nu),&\text{if }(\mu,\nu)\in \CE(\mu_0,m),\\
        +\infty,&\text{otherwise,}
    \end{cases} 
\end{equation*}
whence
\begin{equation}\label{eq:minimax}
    \min\left\{\mathcal E^a(m,\mu,\nu):\, (m,\mu,\nu)\in \widetilde B,\, (\mu,\nu)\in \CE(\mu_0,m)\right\}=\min\limits_{(m,\mu,\nu)\in B}\sup\limits_{\xi\in A}\mathcal L^a(\xi, m,\mu,\nu),
\end{equation}
where we set
\begin{equation}\label{eq:B}
    B:=\operatorname{dom}\mathcal E^a=\left\{(m,\mu,\nu)\in \widetilde B:\, \varphi\in L^1_m(X),\, \rS  +a\varphi\in L^1_\mu(X_T),\, (\mu,\nu)\in \operatorname{dom} \mathcal A_\Diss\right\}.
\end{equation}

To the right hand-side of \eqref{eq:minimax} we will now apply Von Neumann minimax theorem \eqref{thm:VonNeumann}.

\begin{prop}\label{prop: min sup = sup inf}
   If $\sup\limits_{\xi\in A}\inf\limits_{(m,\mu,\nu)\in B}\mathcal L^a(\xi, m,\mu,\nu)< +\infty$, then one has
    \begin{equation*}
        \min\limits_{(m,\mu,\nu)\in B}\sup\limits_{\xi\in A}\mathcal L^a(\xi, m,\mu,\nu)=\sup\limits_{\xi\in A}\inf\limits_{(m,\mu,\nu)\in B}\mathcal L^a(\xi, m,\mu,\nu).
    \end{equation*}
\end{prop}
\begin{proof}
    We need to check the hypotheses of Theorem~\ref{thm:VonNeumann}. 
    
    The set $A=\operatorname{Cyl}_b(X_T)$ is linear, hence convex. The set $B$ in \eqref{eq:B} is convex since the action $\mathcal A_\Diss$ is convex by Proposition~\ref{repr formula proposition} and the constraints are linear in $m$ and $\mu$. On $B$ we consider the product narrow topology.

    The map $\xi\mapsto \mathcal L^a(\xi,m,\mu,\nu)$ is linear, hence concave, while $(m,\mu,\nu)\mapsto \mathcal L^a(\xi,m,\mu,\nu)$ is convex since $\mathcal A^a_\Diss$ is convex and the other terms are affine. Moreover, it is lower semicontinuous with respect to the chosen topology since $\mathcal E^a$ is (see Proposition~\ref{prop:minE}) and $\mathcal C(\xi,\cdot,\cdot,\cdot)$ is even continuous, again because $\partial_t \xi \in C_b(X_T)$, $\xi(T)\in C_b(X)$ and $\mathrm D_x\xi \in \mathfrak{Z}_b(X_T;X^*)$.

    Let now $\overline x\in X$ be a minimum point of $\varphi$, so that $0\in \partial\varphi(\bar x)$ whence $\rS  (\bar x)\le \Diss^*(0)=-\inf \Diss\le0$. Without loss of generality we may also assume that $\varphi(\bar x)=0$. Pick any constant $\overline C> \sup\limits_{\xi\in A}\inf\limits_{(m,\mu,\nu)\in B}\mathcal L(\xi, m,\mu,\nu)$ such that $\overline C\ge T\mu_0(X)\Diss(0)$, and set $\overline \xi(t,x):=(\Diss(0)+1)(t-T-1)$. For an arbitrary $\bar t\in [0,T]$ it is then immediate to compute
    \begin{align*}
        &\mathcal L^a(\overline \xi, \mu_0(X)\delta_{\overline x},T\mu_0(X)\delta_{(\bar t,\overline x)},0)= \mathcal E^a( \mu_0(X)\delta_{\overline x},T\mu_0(X)\delta_{(\bar t,\overline x)},0)\\
        =&e^{-aT}\mu_0(X)\varphi(\bar x)+T\mu_0(X)e^{-a\bar t}(\Diss(0)+\rS  (\bar x)+a\varphi(\bar x))-\int_X\varphi \, d \mu_0\le T\mu_0(X)\Diss(0)\le\overline C,
    \end{align*}
    so that we are just let to prove that the set
    \begin{equation*}
        \{\mathcal L^a(\overline\xi,\cdot,\cdot,\cdot)\le \overline C\}\qquad\text{is compact.}
    \end{equation*}
    This follows by arguing exactly as in the proof of Proposition~\ref{prop:minE}. Indeed, we observe that if $\overline C\ge \mathcal L^a(\overline\xi,m,\mu,\nu)= \mathcal E^a(m,\mu,\nu)+(\Diss(0)+1)(\mu(X_T)+m(X)-(T+1)\mu_0(X)) $, recalling that $\mathcal E^a(m,\mu,\nu)\ge -\Diss(0)\mu(X_T)-\int_X\varphi \, d \mu_0$ then there holds
    \begin{itemize}
        \item $\mu(X_T)+m(X)\le (\Diss(0)+1)(T+1)\mu_0(X) +\overline C+\int_X\varphi \, d \mu_0$;
        \item $\mathcal E^a(m,\mu,\nu)\le (\Diss(0)+1)(T+1)\mu_0(X) +\overline C+\int_X\varphi \, d \mu_0$.
    \end{itemize}
\end{proof}

In the language of optimization, the use of the minimax theorem allowed us to ensure that \emph{strong duality} holds for problem \eqref{relaxation CE}, namely the the optimal values of the primal and the dual problem coincide. We stress that for the moment we do not know whether $\sup \inf \mathcal L^a$ is finite yet; however, this will actually be a byproduct of the following results. Next proposition shows that the dual problem turns out to be an extremely simple affine problem constrained to solutions of an Hamilton-Jacobi inequality.

\begin{prop}\label{prop:dualproblem}
    The following equality holds true:
    \begin{equation}\label{eq:dual}
        \sup\limits_{\xi\in A}\inf\limits_{(m,\mu,\nu)\in B}\mathcal L^a(\xi, m,\mu,\nu)=\sup\left\{\int_X(\xi(0)-\varphi)d\mu_0:\, \xi\in \HJ^a(X_T)\right\},
    \end{equation}
    where $\xi\in \HJ^a(X_T)$ if and only if it belongs to $A=\operatorname{Cyl}_b(X_T)$ and satisfies the following Hamilton-Jacobi type inequality for all $(t,x)\in X_T$:
    \begin{equation}\label{eq:HJ}
        \begin{cases}
            -\partial_t\xi(t,x)+e^{-at}\Diss^*(-e^{at}\mathrm D_x\xi(t,x))\le e^{-at}(\rS  ( x)+a\varphi(x)),\\
            \xi(T,x)\le e^{-aT}\varphi(x).
        \end{cases}
    \end{equation}
\end{prop}
\begin{proof}
    By the very definition of $\mathcal L^a$, for all $\xi\in A$ one deduces
    \begin{align*}
        &\inf\limits_{(m,\mu,\nu)\in B}\mathcal L^a(\xi, m,\mu,\nu)=\inf\limits_{(m,\mu,\nu)\in \widetilde B}\mathcal L^a(\xi, m,\mu,\nu)\\
        =&\int_X(\xi(0)-\varphi)d\mu_0+\inf\limits_{m\in \mathcal M_+(X)}\int_X (e^{-aT}\varphi-\xi(T))dm\\
        &+\inf\limits_{\mu\in \mathcal M_+(X_T)}\Bigg( \int_{X_T}(\partial_t\xi+e^{-at}(\rS  +a\varphi)) d\mu\\
        &\qquad\qquad\qquad\qquad\qquad\qquad+ \inf\limits_{\substack{{\nu\in \mathcal M(X_T;X)}\\{|\nu|\ll\mu}}}\int_{X_T}\left( e^{-at}\Diss\left(\frac{d\nu}{d\mu}\right)+\left\langle \mathrm D_x\xi,\frac{d\nu}{d\mu}\right\rangle\right) d\mu\Bigg).
    \end{align*}
    Notice that by means of Lemma \ref{lemma: dual repr for F^*}, and recalling that $(\lambda \Diss)^*(z)=\lambda \Diss^*(z/\lambda)$ for $\lambda>0$, we know that
    \begin{equation*}
        \inf\limits_{\substack{{\nu\in \mathcal M(X_T;X)}\\{|\nu|\ll \mu}}}\int_{X_T}\left( e^{-at}\Diss\left(\frac{d\nu}{d\mu}\right)+\left\langle \mathrm D_x\xi,\frac{d\nu}{d\mu}\right\rangle\right) d\mu= -\int_{X_T} e^{-at}\Diss^*(- e^{at}\mathrm D_x\xi)d\mu,
    \end{equation*}
    whence we obtain
    \begin{align*}
        \inf\limits_{(m,\mu,\nu)\in B}\mathcal L^a(\xi, m,\mu,\nu)=&\int_X(\xi(0)-\varphi)d\mu_0+\inf\limits_{m\in \mathcal M_+(X)}\int_X (e^{-aT}\varphi-\xi(T))dm\\
        &+ \inf\limits_{\mu\in \mathcal M_+(X_T)}\int_{X_T}(\partial_t\xi+e^{-at}(\rS  +a\varphi-\Diss^*(-e^{at}\mathrm D_x\xi))) d\mu.
    \end{align*}
    Exploiting \eqref{eq:generalintro}, we finally infer
    \begin{align*}
        \inf\limits_{(m,\mu,\nu)\in B}\mathcal L^a(\xi, m,\mu,\nu)=\begin{cases}\displaystyle
            \int_X(\xi(0)-\varphi)d\mu_0,&\text{if }\xi\in \HJ^a(X_T),\\
            -\infty,&\text{otherwise,}
        \end{cases}
    \end{align*}
    and we conclude.
\end{proof}

The final step of our argument consists in proving that solutions to \eqref{eq:HJ} preserve the final bound $e^{aT}\xi(T,x)\le\varphi(x)$ at all previous times (we call this property \lq\lq backward boundedness\rq\rq). Indeed, from this fact one directly infers that the right-hand side of \eqref{eq:dual} is nonpositive, and so Theorem \ref{thm:main} follows.

\begin{lemma}[\textbf{Backward boundedness}]\label{lemma: proof conclusion}
    Any $\xi\in \HJ^a(X_T)$ satisfies $\xi(t,x)\le e^{-at}\varphi (x)$ for all $(t,x)\in X_T$.
\end{lemma}
\begin{proof}
    For $\varepsilon>0$ we define $\xi_\varepsilon(t,x):=\xi(t,x)+\varepsilon(t-T-1)$, which thus satisfies
    \begin{equation}\label{eq:epsineq}
        \begin{cases}
            -\partial_t\xi_\varepsilon(t,x)+e^{-at}\Diss^*(-e^{at}\mathrm D_x\xi_\varepsilon(t,x))\le e^{-at}(\rS  ( x)+a\varphi(x))-\varepsilon,\\
            \xi_\varepsilon(T,x)\le e^{-aT}\varphi(x)-\varepsilon.
        \end{cases}
    \end{equation}
    for all $(t,x)\in X_T$. We now claim that
    \begin{equation}\label{eq:claim}
        \xi_\varepsilon(t,x) < e^{-at}\varphi (x), \quad\text{ for all $(t,x)\in X_T$,}
    \end{equation}
    whence the thesis follows by sending $\varepsilon\to 0$. 
    
    In order to prove the claim, by contradiction let $(\widetilde t, \widetilde x)\in X_T$ be such that $\xi_\varepsilon(\widetilde t, \widetilde x)\ge e^{-a\widetilde t}\varphi(\widetilde x)$. We can then define
    \begin{equation*}
        \overline{t}:=\sup\{t\in [0,T]:\, \text{there exists } x\in X \text{ for which }\xi_\varepsilon(t,x) \ge e^{-at}\varphi (x)\},
    \end{equation*}
    which fulfils the following properties:
    \begin{itemize}
        \item[(0)] $\overline{t}$ is a maximum;
        \item [(1)] $\overline t<T$;
        \item[(2)] for all $t\in(\overline t,T]$ one has $\xi_\varepsilon(t,x) < e^{-at}\varphi (x)$ for all $x\in X$;
        \item[(3)] there exists $\overline x\in X$ for which $\xi_\varepsilon(\overline t, \overline x) = e^{-a\bar t}\varphi (\overline x)$;
        \item[(4)] $\partial_t\xi_\varepsilon(\overline t, \overline x)\le -ae^{-a\bar t}\varphi(\bar x)$;
        \item[(5)]$\Diss^*(-e^{a\bar t}\mathrm D_x\xi_\varepsilon(\overline t, \overline x) )\ge \rS(\bar x) $. 
    \end{itemize}
    To show (0), we pick a sequence of times $t_n\nearrow \overline t$ such that there exists $x_n\in X$ satisfying $\xi_\varepsilon(t_n,x_n) \ge e^{-at_n}\varphi (x_n)$. Since $\xi_\varepsilon$ is bounded and $\varphi$ is coercive, from $\{x_n\}_n$ we can extract a subsequence (not relabelled) converging to $\overline x\in X$. By lower semicontinuity of $\varphi$, sending $n\to +\infty$ it is now immediate to check that $\xi_\varepsilon(\overline t,\overline x) \ge e^{-a\overline t}\varphi (\overline x)$, namely $\overline t$ is a maximum. Property (1) then directly follows from the second inequality in \eqref{eq:epsineq}, while (2) from the definition of $\overline{t}$. Notice in particular that sending $t\searrow \overline t$ in (2) yields
    \begin{equation}\label{eq:ineqbart}
        \xi_\varepsilon(\overline t,x) \le e^{-a\overline t}\varphi (x),\quad\text{for all }x\in X.
    \end{equation}
    Recalling that (0) implies that there exists $\overline x\in X$ satisfying $\xi_\varepsilon(\overline t,\overline x) \ge e^{-a\overline t}\varphi (\overline x)$, by plugging $\overline x$ in \eqref{eq:ineqbart} we deduce (3). Now, by (2) and (3), the $C^1$ map $\phi(t):= \xi_\varepsilon(t,\overline x)-e^{-at}\varphi(\overline x)$ vanishes at $\overline t$ and is negative in $(\overline t,T]$: this implies $\phi'(\overline t)\le 0$, and so (4) holds true. In order to show (5), by \eqref{eq:obvious-relaxation} it is enough to check that $e^{a\bar t}\mathrm D_x\xi_\varepsilon(\overline t, \overline x)\in \partial\varphi(\overline x)$: to this aim, by means of \eqref{eq:ineqbart} and (3), we verify
    \begin{equation*}
        \liminf\limits_{x\to \overline x} \frac{\varphi(x){-}\varphi(\overline x){-}\langle e^{a\bar t}\mathrm D_x\xi_\varepsilon(\overline t, \overline x),x{-}\overline x\rangle}{\|x-\overline x\|}\ge e^{a\bar t}\liminf\limits_{x\to \overline x} \frac{\xi_\varepsilon(\overline t,x){-}\xi_\varepsilon(\overline t,\overline x){-}\langle \mathrm D_x\xi_\varepsilon(\overline t, \overline x),x{-}\overline x\rangle}{\|x-\overline x\|}=0,
    \end{equation*}
    where the last equality holds since $\xi_\varepsilon\in C^1(X_T)$. Thus (5) is proved.

    By plugging $(\overline t,\overline x)$ into the first inequality in \eqref{eq:epsineq}, and exploiting properties (4) and (5) we now obtain
    \begin{equation*}
        \varepsilon\le\partial_t\xi_\varepsilon(\overline t,\overline x)+e^{-a\bar t}(-\Diss^*(-e^{a\bar t}\mathrm D_x\xi_\varepsilon(\overline t,\overline x))+\rS  ( \bar x)+a\varphi(\bar x))        
        \le 0,
    \end{equation*}
    and we reach a contradiction. So \eqref{eq:claim} is proved and we conclude.
\end{proof}

\section{Non autonomous and state-dependent dissipation potential}\label{sec:nonautonomous}

In this last section we extend the analysis to time-dependent energies $\varphi\colon X_T\to (-\infty,+\infty]$ and time- and state-dependent dissipation potentials $\Diss\colon X_T\times X\to (-\infty,+\infty]$, namely we consider the doubly nonlinear equation
\begin{equation}\label{eq:DNEt}
    \begin{cases}
        \partial_v\Diss(t,\xx(t),\dot {\xx}
        (t))+\partial_\ell \varphi(t,\xx(t))\ni 0,\quad \text{in }X^*,\quad\text{for a.e. }t\in (0,T),\\
        \xx(0)=x_0,
    \end{cases}
\end{equation}
where the limiting subdifferential $\partial_{\ell}\varphi$ is now defined as
 \begin{equation*}
       \partial_{\ell}\varphi(t,x):= \left\{ z\in X^* \, : \, \text{there exist } (t_n,x_n) \to (t,x)\text{ and } z_n \in \partial_x\varphi(t_n,x_n) \text{ such that } \ z_n \rightharpoonup z \right\}.
   \end{equation*}
 This framework reflects the fact that the underlying geometry may change at any point $x \in X$, possibly varying in time (see \cite{RosMieSav2008metricapproach} for a detailed discussion).

We require that the domain of $\varphi$ does not depend on time, namely it has the form $[0,T]\times D$ for a certain subset $D\subseteq X$, and we assume:

\begin{enumerate}[label=\textup{($\widetilde{\varphi\arabic*}$)}]
		\item \label{hyp:phit1} $\varphi(t,\cdot)$ is proper and lower semicontinuous in $X$ for all $t\in [0,T]$;
  \item  sublevels of $\varphi(0,\cdot)$ are boundedly compact in $X$;
	\item \label{hyp:phit3} for all $x\in D$ the map $t\mapsto \varphi(t,x)$ is differentiable in $[0,T]$ and there exist $b>0$ and $c\ge 0$ such that
    \begin{equation}\label{eq:powerbound}
        |\partial_t\varphi(t,x)|\le b(\varphi(t,x)+c(1+\|x\|)),\quad\text{for all }(t,x)\in X_T.
    \end{equation}
    Moreover, there exists $\bar a\ge 0$ such that $\bar a\varphi-\partial_t\varphi$ is lower semicontinuous in $X_T$.
\end{enumerate}

Observe that, by means of Gr\"onwall's Lemma, condition \eqref{eq:powerbound} yields that
\begin{equation*}
    \varphi(s,x)+c(1+\|x\|)\le e^{b|t-s|}(\varphi(t,x)+c(1+\|x\|)),\qquad\text{for all }s,t\in [0,T]\text{ and }x\in D.
\end{equation*}
From this inequality, one can easily deduce that actually 
\begin{gather*}
    \text{sublevels of $\varphi$ are boundedly compact in $X_T$,}\\
    \text{so that $\varphi$ is lower semicontinuous in $X_T$.}
\end{gather*}

As regards the dissipation potential $\Diss$, we assume:
\begin{enumerate}[label=\textup{($\widetilde{\Diss\arabic*}$)}]
    \item \label{hyp:Ft1} $\Diss(t,x,\cdot)$ is convex and lower semicontinuous for all $(t,x)\in X_T$;
  \item \label{hyp:Ft2} $ \Diss(t,x,v)\ge G(v)$ for all $(t,x,v)\in X_T\times X$, for some function $G\colon X\to(-\infty,+\infty]$ bounded below with superlinear growth at infinity;
  \item \label{hyp:Ft3} $\sup\limits_{(t,x)\in X_T}\Diss(t,x,0)<+\infty$;
  \item \label{hyp:Ft4} the map $(t,x,z)\mapsto  \Diss^*(t,x,z)$ is upper semicontinuous in $X_T\times X^*$.  
\end{enumerate}

Observe that previous assumptions imply that $\Diss$ satisfies \ref{hyp:Fy1}-\ref{hyp:Fy4} of Section \ref{subsec:action}, with $Y=X_T$. As a consequence, we recall that both $\Diss$ and $\Diss^*$ are bounded below, that $\Diss^*(t,x,\cdot)$ is locally bounded in $X^*$, uniformly with respect to $(t,x)\in X_T$, and so $\Diss^*(t,x,\cdot)$ is continuous in $X^*$ for all $(t,x)\in X_T$. 

\begin{oss}
    Our assumptions cover the simple example presented in Remark \ref{rmk:nonautonomous}. In particular, consider $$\Diss(t,x,v):= a(t,x) \Diss_0(v),$$ for some $a\colon X_T \to (0,+\infty)$ and a nonnegative $\Diss_0 \colon X \to [0,+\infty]$. The requirements \ref{hyp:Ft1}-\ref{hyp:Ft4} are met assuming the following: 
    \begin{enumerate}
        \item[(i)] $\Diss_0$ satisfies \ref{hyp:F1}, \ref{hyp:F2} and \ref{hyp:F3}; 
        \item[(ii)] $a$ is continuous and bounded away from $0$, that is $ a(t,x)\geq \delta>0$ for all $(t,x) \in X_T$;
        \item[(iii)] if $\Diss_0(0)>0$, then $a$ is bounded. 
    \end{enumerate}
    Indeed, clearly \ref{hyp:Ft1} is satisfied, while \ref{hyp:Ft3} follows from (iii). Regarding \ref{hyp:Ft2}, it suffices to set $G(v) := \delta \Diss_0(v)$. As for \ref{hyp:Ft4}, we simply compute
    \begin{equation*}
        \Diss^*(t,x,z)=a(t,x)\Diss_0^*\left(\frac{z}{a(t,x)}\right),
    \end{equation*}
    which is even continuous in $X_T\times X^*$.

    If we further require that 
    \begin{enumerate}
        \item[(iv)] $\Diss_0$ is continuous at $0$ and $a$ is bounded,
    \end{enumerate}
    then also \ref{hyp:Ft5} below (see Proposition \ref{prop:F*nablat}) is fulfilled, while \ref{hyp:Ft6} directly follows from (ii).
\end{oss}

\begin{oss}
    A more interesting case from the PDE perspective is given by the following concrete example in $X=L^p(\Omega)$, where $p>1$ and $\Omega\subseteq \R^n$ is a Lipschitz domain:
    \begin{equation*}
        \Diss(t,u,v):=\int_\Omega \frac{A(t,u(x))}{p}|v(x)|^p\, dx.
    \end{equation*}
    Above, $A\colon [0,T]\times \R\to (0,+\infty)$ is a continuous function such that $0< \alpha\le A(t,u)\le \beta$ for all $(t,u)\in [0,T]\times \R$. This choice, together with the energy
    \begin{equation*}
        \varphi(u):=\begin{cases}\displaystyle
            \int_\Omega \frac 12|\nabla u(x)|^2+W(u(x)) \, dx,&\text{if }u\in H^1_0(\Omega),\\
            +\infty,&\text{otherwise},
        \end{cases}
    \end{equation*}
    with $W\colon \R\to [0,+\infty)$ smooth, corresponds to the equation
    \begin{equation*}
        A(t,u(t,x))|\dot u(t,x)|^{p-2}\dot u(t,x)-\Delta u(t,x)+W'(u(t,x))=0,\quad\text{in }[0,T]\times \Omega,
    \end{equation*}
    with homogeneous Dirichlet boundary conditions.

    Properties \ref{hyp:Ft1}, \ref{hyp:Ft2},\ref{hyp:Ft3} and \ref{hyp:Ft5} are immediate. Computing
    \begin{equation*}
        \Diss^*(t,u,z)= \int_\Omega \frac{|z(x)|^q}{qA(t,u(x))^{q-1}}\, dx,
    \end{equation*}
    where $q>1$ is the conjugate of $p$, we deduce that $\Diss^*$ is even continuous in $[0,T]\times L^p(\Omega)\times L^q(\Omega)$, whence \ref{hyp:Ft4} and \ref{hyp:Ft6} hold.
\end{oss}

In this nonautonomous setting, the energy functional in the De Giorgi's principle must be modified as follows. For $a\ge b$, with $b$ appearing in \eqref{eq:powerbound}, let $\mathcal J^a\colon [0,T]\times AC([0,T];X)\to (-\infty,+\infty]$ be defined as
  \begin{align*}
      \mathcal{J}^a(t,\xx):=&e^{-at}\varphi(t,\xx(t))-\varphi(0,x_0)\\&+\int_0^t e^{-a\tau}\Big(\Diss(\tau,\xx(\tau),\dot{\xx}(\tau)) +\rS(\tau,\xx(\tau))+a\varphi(\tau,\xx(\tau))-\partial_t\varphi(\tau,\xx(\tau))\Big)\, d\tau.
  \end{align*}
  Here, the time-dependent relaxed-slope $\rS\colon X_T\to (-\infty,+\infty]$ takes the obvious form
\begin{equation*}
      \rS(t,x):=\inf\Big \{
        \liminf_{n\to\infty}
        \Diss^*(t_n,x_n,-z_n):\, z_n\in\partial_x\varphi(t_n,x_n),
        \
        (t_n,x_n)\to (t,x)
        \Big\}
        \ \text{if }(t,x)\in 
        \overline{\operatorname{dom}\varphi},
  \end{equation*}
   and, similarly to \eqref{eq:nablaF}, we denote by $\partial^*_x\varphi(t,x)$ the set of elements $z\in\partial_\ell\varphi(t,x)$ satisfying the inequality $\Diss^*(t,x,-z)\le\rS(t,x)$.

   De Giorgi's principle can be thus stated as follows (with the very same proof as in the autonomous case).
\begin{prop}\label{prop:DeGiorgit}
    Assume the validity of the following two properties:
    \begin{enumerate}[label=\textup{($\widetilde S$)}]
    \item\label{Mt} $\partial^*_x\varphi(t,x)$ is nonempty for every $(t,x)\in \operatorname{dom}\rS$;
    \end{enumerate}
    \begin{enumerate}[label=\textup{($\widetilde{CR}$)}]
    \item \label{CRt} the following \emph{chain-rule} holds: for all $x\in AC([0,T];X)$ such that $$\int_0^T \Diss(\tau,\xx(\tau),\dot{\xx}(\tau)) + \rS(\tau,\xx(\tau))\, d\tau<+\infty,$$ the map $t\mapsto \varphi(t,\xx(t))$ is absolutely continuous and 
    \begin{equation*}
        \frac{d}{dt}\varphi(t,\xx(t))=\partial_t\varphi(t,\xx(t))+\langle \mathrm z(t), \dot {\xx}(t)\rangle,\quad\text{for a.e. }t\in(0,T),
    \end{equation*}
    for any measurable selection $\mathrm z\colon [0,T]\to X^*$ satisfying $\mathrm z(t)\in \partial^*_x\varphi(t,\xx(t))$ a.e. in $[0,T]$.
     \end{enumerate}

    Then, for every initial datum $x_0\in D$ and trajectory $\bar {\xx}\in AC([0,T];X)$ with $\bar {\xx}(0)=x_0$, we have
    \begin{equation*}
        \mathcal J^a(t,\bar\xx)\ge 0,\quad\text{ for all $t\in [0,T]$ and $a\ge b$},
    \end{equation*}    
    and the following conditions are equivalent:
    \begin{itemize}
        \item [(a)] $\bar {\xx}$ solves the doubly nonlinear equation \eqref{eq:DNEt}, in the sense that there exists a measurable map $\bar {\mathrm z}\colon [0,T]\to X^*$ such that
        \begin{equation}\label{eq:solrel}
            -\bar {\mathrm{z}}(t)\in \partial_v \Diss(t,\bar {\xx}(t),\dot{\bar {\xx}}(t)),\quad \bar{\mathrm{z}}(t)\in \partial_\ell\varphi(t,\bar{\xx}(t))\quad\text{ for a.e. }t\in (0,T),
        \end{equation}
        and
        \begin{equation}\label{eq:sol2t}
         \displaystyle 
        \psi^*(t,{\bar \xx}(t),-\bar{\mathrm z}(t))=S^-(t,\bar\xx(t))
        \quad\text{
        a.e.~in }(0,T),
        \ \int_0^T \Diss(\tau,\bar\xx(\tau),\dot{\bar {\xx}}(\tau)) + \rS(\tau,\bar {\xx}(\tau))\, d\tau<+\infty;
        \end{equation}
        \item[(b)] $\mathcal{J}^a(t,\bar {\xx})\le 0$ for all $t\in [0,T]$ and for some $a\ge b$;
        \item[(c)] $\mathcal{J}^a(t,\bar {\xx})= 0$ for all $t\in [0,T]$ and for some $a\ge b$;
        \item[(d)] $\mathcal{J}^a(T,\bar {\xx})\le 0$ for some $a\ge b$.
    \end{itemize}
    Moreover, if one of the above is in force, then $(b)$, $(c)$ and $(d)$ hold true for all $a\ge b$, and $\bar{\xx}$ attains the minimum of $\mathcal{J}^a(T,\cdot)$, that is $\mathcal{J}^a(T,\bar {\xx})=\min\left\{\mathcal{J}^a(T, \xx):\, \xx\in AC([0,T];X),\, \xx(0)=x_0 \right\}$.
\end{prop}

It is also useful to consider the analogue of Proposition \ref{prop:F*nabla} in the current nonautonomous setting.

\begin{prop}\label{prop:F*nablat}
      In addition to \ref{hyp:phit1}-\ref{hyp:phit3} and \ref{hyp:Ft1}-\ref{hyp:Ft4}, assume that      
      \begin{enumerate}[label=\textup{($\widetilde{\Diss 5}$)}]
      \item\label{hyp:Ft5} $\Diss(t,x,\cdot)$ is continuous at $0$, uniformly with respect to $(t,x)\in X_T$;
      \end{enumerate}
       \begin{enumerate}[label=\textup{($\widetilde{\Diss 6}$)}]
      \item\label{hyp:Ft6} The map $(t,x)\mapsto \Diss(t,x,v)$ is upper semicontinuous in $X_T$ for all $v\in X$.
      \end{enumerate}
      Then for all $(t,x)\in \operatorname{dom}\rS$, the set $\partial^*_x\varphi(t,x)$ is nonempty, so that in particular \ref{Mt} holds.

      If moreover $\partial_x\varphi$ is strongly-weakly sequentially closed, namely
  \begin{equation*}
      (t_n,x_n)\to (t,x)\text{ in }X_T,\quad z_n\xrightharpoonup[]{} z\text{ weakly in }X^*,\quad z_n\in \partial_x\varphi(t_n,x_n) \implies z\in \partial_x\varphi(t,x),
      \end{equation*}     
      then we have
      \begin{equation*}
          \rS(t,x)=\min\left\{\Diss^*(t,x,-z):\, z\in \partial_x\varphi(t,x)\right\},
      \end{equation*}
      
  \end{prop}
  \begin{proof}
      The thesis follows by arguing exactly as in Proposition \ref{prop:F*nabla} once we observe that $\Diss^*$ is lower semicontinuous on $X_T\times X^*$ with respect to the strong-weak topology due to \ref{hyp:Ft6}, and that \ref{hyp:Ft5} implies that for all $C\in \R$ the set $\{z\in X^*:\, \Diss^*(t,x,z)\le C\text{ for some }(t,x)\in X_T\}$ is bounded (see the proof of Lemma \ref{lemma:propFenchel}).
  \end{proof}

In the current framework, our result can be stated as follows.

\begin{teorema}\label{thm:time}
Let us assume \ref{hyp:phit1}--\ref{hyp:phit3} and \ref{hyp:Ft1}--\ref{hyp:Ft4}, and let $a\ge \max\{\bar a,b\}$, $T>0$, and $x_0\in D$ be given.
      If $a>b$, or if $a=b$ and $\rS$ has boundedly compact sublevels, then the functional ${\mathcal{J}}^a(T,\cdot)$ attains the minimum in the class of absolutely continuous curves starting from $x_0$ and
      \begin{equation}\label{eq:probt}
          \min\left\{{\mathcal{J}}^a(T,\xx):\, \xx\in AC([0,T];X),\, \xx(0)=x_0\right\}\le 0.
      \end{equation}

      If in addition conditions \ref{Mt} and \ref{CRt} of Proposition \ref{prop:DeGiorgit} are in force (again, see Proposition \ref{prop:F*nablat} for sufficient conditions ensuring \ref{Mt}), then the set of minimizers of ${\mathcal{J}}^a(T,\cdot)$ does not depend on $a$, and its elements satisfy \eqref{eq:solrel} and \eqref{eq:sol2t}. In particular, if $\partial_x\varphi$ is strongly-weakly sequentially closed, the minimizers are solutions of the more standard doubly nonlinear equation \eqref{eq:DNEt} with $\partial_x\varphi$ in place of $\partial_\ell\varphi$.
  \end{teorema}

  \subsection{Proof of Theorem \ref{thm:time}}  

The proof of the previous theorem follows the lines of the autonomous case. Here we present the main differences which arise due to time- and state-dependence of $\varphi$ and $\Diss$.

First, we observe that \eqref{eq:powerbound} yields the bound
\begin{equation*}
    a\varphi(t,x)-\partial_t\varphi(t,x)\ge (a-b)\varphi(t,x)-c(1+\|x\|)\ge -C(1+a-b)(1+\|x\|).
\end{equation*}
So, the very same argument of Section \ref{sec:reduction} shows that we can assume without loss of generality that 
\begin{itemize}
    \item $\varphi$ has compact sublevels, and $\rS$ has bounded sublevels;
    \item $|\partial_t\varphi (t,x)|\le b(\varphi(t,x)+c)$ for all $(t,x)\in [0,T]\times D$.
\end{itemize}
In particular, $\varphi$ admits minimum on $X_T$ and so we may require in addition that
\begin{itemize}
    \item $\varphi(t,x)\ge 0$ for all $(t,x)\in X_T$;
    \item $\Diss(t,x,v)\ge 0$ for all $(t,x,v)\in X_T\times X$.
\end{itemize}
Here and henceforth, we thus tacitly assume such stronger assumptions.

\subsubsection{Relaxed problem}
The same convexification procedure of Section \ref{sec: convexification} illustrates that the equivalent relaxed problem to \eqref{eq:probt} is given by (we already replaced $\delta_{x_0}$ with a more general $\mu_0$ as in Section \ref{sec: proof})
\begin{equation}\label{eq:inf}
    \inf \left\{ \mathcal{E}^a(m,\mu,\nu) \ : \ (m,\mu,\nu) \in \widetilde B,\,(\mu,\nu)\in\CE(\mu_0,m) \right\},
\end{equation}
where $\widetilde B$ has been introduced in \eqref{eq:Btilde} and now $\mathcal E^a\colon \widetilde B\to (-\infty,+\infty]$ is given by 
\begin{equation}\label{eq:Eat}
    \begin{aligned}
    \mathcal E^a(m,\mu,\nu) :=& e^{-aT}\int_X \varphi(T,x) dm(x)-\int_X\varphi(0,x)\,d\mu_0(x) +\mathcal A^a_\Diss(\mu,\nu)
    \\
    & + \int_{X_T}e^{-at} \big(\rS(t,x)+a\varphi(x)-\partial_t\varphi(t,x)\big) d\mu(t,x).
\end{aligned}
\end{equation}
We point out that the action functional is now obtained from \eqref{eq: action} choosing $Y=X_T$ and $\Psi(t,x,v)=e^{-at}\Diss(t,x,v)$, which still fulfils properties \ref{hyp:Fy1}-\ref{hyp:Fy4}.

Let us now check that \eqref{eq:inf} attains the minimum via the direct method of Calculus of Variations, as in Proposition \ref{prop:minE}. The first line in \eqref{eq:Eat} is lower semicontinuous with respect to the product narrow topology (recall Proposition \ref{repr formula proposition}), and the same holds for the second line since the integrand is lower semicontinuous by \ref{hyp:phit3}:
\begin{equation*}
    a\varphi-\partial_t\varphi=\underbrace{(a-\bar a)}_{\ge 0}\varphi+\bar a\varphi- \partial_t\varphi,
\end{equation*}
and bounded below:
\begin{equation*}
    a\varphi-\partial_t\varphi\ge (a-b)\varphi-bc\ge -bc. 
\end{equation*}

Moreover, coercivity of $\mathcal E^a$ follows by the same estimates of Proposition \ref{prop:minE}. Indeed, notice that the integrand in the second line of \eqref{eq:Eat} has compact sublevels: if $a>b$, recalling that $\rS$ is bounded below, it holds
\begin{equation*}
    \rS(t,x)+a\varphi(x)-\partial_t\varphi(t,x)\ge (a{-}b)\varphi(x)-M-bc\ge \underbrace{(a{-}b)}_{>0}\varphi(x)-M-bc.
\end{equation*}
If $a=b$ instead, since in this case we assumed that $\rS$ is coercive, the same conclusion follows from the inequality
\begin{equation*}
    \rS(t,x)+a\varphi(x)-\partial_t\varphi(t,x)\ge \rS(t,x)-ac.
\end{equation*}

\subsubsection{Duality and Hamilton-Jacobi inequality}

We now consider the Lagrangian $\mathcal L^a$ defined as in \eqref{eq:lagrangian} with $\mathcal E^a$ given by \eqref{eq:Eat}, so that
\begin{equation*}
    \min \left\{ \mathcal{E}^a(m,\mu,\nu) \ : \ (m,\mu,\nu) \in \widetilde B,\,(\mu,\nu)\in\CE(\mu_0,m) \right\}=\min\limits_{(m,\mu,\nu)\in \widetilde B}\sup\limits_{\xi\in A}\mathcal L^a(\xi, m,\mu,\nu).
\end{equation*}
Exploiting the previous estimates and arguing as in Proposition \ref{prop: min sup = sup inf}, it is immediate to check that also in the nonautonomous setting we can apply Von Neumann Theorem \ref{thm:VonNeumann}, so that we can interchange $\sup$ and $\inf$ in the above saddle-point problem. 

The very same proof of Proposition \ref{prop:dualproblem} then shows that the dual problem takes the following form.
\begin{prop}
    There holds:
    \begin{equation*}
        \sup\limits_{\xi\in A}\inf\limits_{(m,\mu,\nu)\in \widetilde B}\mathcal L^a(\xi, m,\mu,\nu)=\sup\left\{\int_X(\xi(0)-\varphi(0,\cdot))d\mu_0:\, \xi\in \HJ^a(X_T)\right\},
    \end{equation*}
    where the cylinder function $\xi\in A$ belongs to $\HJ^a(X_T)$ if and only if for all $(t,x)\in X_T$ it satisfies:
    \begin{equation}\label{eq:HJt}
        \begin{cases}
            -\partial_t\xi(t,x)+e^{-at}\Diss^*(t,x,-e^{at}\mathrm D_x\xi(t,x))\le e^{-at}(\rS(t,x)+a\varphi(t,x)-\partial_t\varphi(t,x)),\\
            \xi(T,x)\le e^{-aT}\varphi(T,x).
        \end{cases}
    \end{equation}
\end{prop}

Finally, we conclude by proving that also system \eqref{eq:HJt} possesses the backward boundedness property.

\begin{lemma}
    Any $\xi\in \HJ^a(X_T)$ satisfies $\xi(t,x)\le e^{-at}\varphi (t,x)$ for all $(t,x)\in X_T$.
\end{lemma}
\begin{proof}
    As in Lemma \ref{lemma: proof conclusion}, we conclude if we prove that $\xi_\varepsilon (t,x):=\xi(t,x)+\varepsilon(t-T-1)$ fulfils 
    \begin{equation}\label{eq:claimt}
        \xi_\varepsilon(t,x)<e^{-at}\varphi(t,x),\quad\text{for all }(t,x)\in X_T.
    \end{equation}
    Assuming by contradiction that the above inequality fails at some point, it is well-defined the time
    \begin{equation*}
        \overline{t}:=\max\{t\in [0,T]:\, \text{there exists } x\in X \text{ for which }\xi_\varepsilon(t,x) \ge e^{-at}\varphi (t,x)\},
    \end{equation*}
   which fulfils the following properties:
    \begin{itemize}
        \item [(1)] $\overline t<T$;
        \item[(2)] for all $t\in(\overline t,T]$ one has $\xi_\varepsilon(t,x) < e^{-at}\varphi (t,x)$ for all $x\in X$;
        \item[(3)] there exists $\overline x\in X$ for which $\xi_\varepsilon(\overline t, \overline x) = e^{-a\bar t}\varphi (\bar t,\overline x)$;
        \item[(4)] $\partial_t\xi_\varepsilon(\overline t, \overline x)\le -e^{-a\bar t}(a\varphi(\bar t,\bar x)-\partial_t\varphi(\bar t,\bar x))$;
        \item[(5)]$\Diss^*(\bar t,\bar x,-e^{a\bar t}\mathrm D_x\xi_\varepsilon(\overline t, \overline x) )\ge \rS(\bar t,\bar x) $. 
    \end{itemize}
    They can be easily checked arguing as in the proof of Lemma \ref{lemma: proof conclusion}. Observing that $\xi_\varepsilon$ satisfies
    \begin{equation*}
        \begin{cases}
            -\partial_t\xi_\varepsilon(t,x)+e^{-at}\Diss^*(t,x,-e^{at}\mathrm D_x\xi_\varepsilon(t,x))\le e^{-at}(\rS(t,x)+a\varphi(t,x)-\partial_t\varphi(t,x))-\varepsilon,\\
            \xi_\varepsilon(T,x)\le e^{-aT}\varphi(T,x)-\varepsilon,
        \end{cases}
    \end{equation*}
    we reach a contradiction since the first inequality cannot be true at $(\bar t,\bar x)$ by (4) and (5), whence \eqref{eq:claimt} must hold.
\end{proof}

\appendix

\section{Metrizability in the space of Banach-valued measures}\label{sect app: narrow}

Let $(Y,\tau)$ and $(X,\|\cdot\|)$ be, respectively, a Polish space and a reflexive and separable Banach space. Here we show that the narrow topology in the space of measures $\mathcal{M}(Y;X)$, introduced in Definition \ref{def: narrow topology}, which we denote by $\tau_{\mathtt{n}}$, is metrizable if restricted to bounded {and uniformly tight sets} (in total variation), namely we prove Proposition \ref{prop:localmetrnarrow}. We begin with the following preliminary lemma, {in which we adopt the notation $\mathcal{M}_R(Y;X) := \{\nu \in \mathcal{M}(Y;X) \ : \ \|\nu\|_{\operatorname{TV}}\leq R\}$ for $R>0$.}

\begin{lemma}\label{lemma: D-cont functions}
    Let $D$ be a countable and dense subset of the unit ball of $X^*$. Then, {for $R>0$,} the narrow topology on {$\mathcal{M}_R(Y;X)$} is the smallest topology that makes continuous the functionals
    \begin{equation}\label{eq app: D-cont functions}
        \mathcal{M}(Y;X)\ni\nu \mapsto \sum_{i=1}^N \langle z_i, \int_Y \zeta_i d\nu \rangle, \quad  \text{for all } N\in \N,\ z_i \in D, \  \zeta_i \in C_b(Y). 
    \end{equation}
\end{lemma}

\begin{proof}
    Let $\tau_D$ be the smallest topology that makes continuous the functionals in \eqref{eq app: D-cont functions}. Then the inclusion $\tau_D \subset \tau_{\mathtt{n}}$ is trivial. On the other hand, any functional of the form \eqref{eq: narrow cont functions} can be uniformly approximated {on $\mathcal{M}_R(Y;X)$} by a sequence of functionals of the form 
    \[\nu \mapsto \sum_{i=1}^N \langle z_i, \int_Y \zeta_i d\nu \rangle, \quad  \text{for some }N\in \N, \ z_i \in \operatorname{Span}(D), \  \zeta_i \in C_b(Y),\]
    which are $\tau_D$-continuous. In particular, any functional in \eqref{eq: narrow cont functions} is $\tau_D$-continuous. This provides the opposite inclusion  $\tau_{\mathtt{n}}\subset \tau_D$ {when restricted to $\mathcal{M}_R(Y;X)$}.
\end{proof}

We now recall that the narrow topology over signed measures $\mathcal{M}(Y)$, which we here indicate with the symbol $\sigma_\mathtt{n}$, is metrizable on bounded {and uniformly tight} subsets. In particular, there exists a distance $d_{\mathcal{M}(Y)}$ that metrizes $\sigma_\mathtt{n}$ if restricted to bounded {and uniformly tight} subsets (see {\cite[Theorem 3.1.9]{Bogachevweakconv}}). We can also assume that such distance is bounded by $1$, up to replace it with the topologically equivalent $d_{\mathcal{M}(Y)}\wedge 1$. In order to transport such distance to vector valued measures $\mathcal M(Y;X)$ we will employ the following map.

\begin{df}
    Let $z\in X^*$. Define the functional $\pi_z:\mathcal{M}(Y;X) \to \mathcal{M}(Y)$ as
    \begin{equation*}
        (\pi_z \nu)(A) := \langle z, \nu(A) \rangle,
    \end{equation*}
    for any Borel subset $A\subset Y$. In particular, for any $\zeta \in C_b(Y)$, it holds the following integration formula
    \begin{equation*}
        \int_Y \zeta(y) d(\pi_z\nu)(y) = \left\langle  z, \int_Y\zeta(y)  d\nu(y)\right\rangle.
    \end{equation*}
\end{df}

Next lemma lists some properties of the map $\pi_z$.

\begin{lemma}\label{lemma: properties of pi_z}
 Let $D$ be as in Lemma \ref{lemma: D-cont functions}. Then the following properties hold true:
\begin{itemize}
    \item[(1)] For all $z\in X^*$, the map $\pi_z: \mathcal{M}(Y;X) \to \mathcal{M}(Y)$ is $(\tau_\mathtt{n},\sigma_\mathtt{n})$-continuous.
    \item[(2)] Let $\tau$ be any topology on $\mathcal{M}(Y;X)$ that makes continuous the maps $\pi_z$ for any $z\in D$. Then $\tau_\mathtt{n}|_{{\mathcal{M}_R(Y;X)}} \subset \tau$.
    \item[(3)] The narrow topology $\tau_\mathtt{n}|_{{\mathcal{M}_R(Y;X)}}$ is the smallest topology that makes continuous $\pi_z$ for any $z\in D$.
    \item[(4)] For all $z\in X^*$ {and $B\subset Y$ Borel set, it holds $|\pi_z \nu|(B)\leq \|z\|_*|\nu|(B)$.} In particular, $\|\pi_z\nu\|_{\operatorname{TV}} \leq \|z\|_* \|\nu\|_{\operatorname{TV}}$.
\end{itemize}
\end{lemma}

\begin{proof}
    $(1)$ Since $\sigma_\mathtt{n}$ is the smallest topology that makes continuous the functions $\mathcal{M}(Y)\ni \mu \mapsto \int_Y \zeta d\mu$ for all $\zeta \in C_b(Y)$, the functional $\pi_z$ is $(\tau_\mathtt{n},\sigma_\mathtt{n})$-continuous if and only if for all $\zeta\in C_b(Y)$ the map
    $\displaystyle \nu\mapsto \int_Y \zeta d(\pi_z\nu) = \left\langle z, \int_Y \zeta d\nu \right\rangle$ is continuous.
    \\
    $(2)$ Any functional as in \eqref{eq app: D-cont functions} is the sum of $\tau$-continuous functions. Then we conclude thanks to Lemma \ref{lemma: D-cont functions}.
    \\
    $(3)$ It is a consequence of $(1)$ and $(2)$.
    \\
    $(4)$ For any partition $(A_i)_{i=1}^{+\infty}$ of ${B}$, it holds
    \[\sum_{i=1}^{+\infty} |(\pi_z \nu)(A_i)|\ = \sum_{i=1}^{+\infty} |\langle z, \nu(A_i) \rangle| \leq \|z\|_* \sum_{i=1}^{+\infty}\|\nu(A_i)\| \leq \|z\|_* \|\nu\|_{\operatorname{TV}},\]
    and we conclude passing to the supremum on the left-hand side. 
\end{proof}
In this proposition we explicitely build a distance on ${\mathcal{M}(Y;X)}$ which metrizes the narrow topology on bounded {and uniformly tight} sets, exploiting the maps $\pi_z$. {This immediately yields the validity of Proposition \ref{prop:localmetrnarrow}, since the formula 
    \begin{equation*}
        |\nu|(U) = \sup\left\{ \int_Y \langle \boldsymbol{\zeta} , d\nu \rangle : \boldsymbol{\zeta} \in \mathfrak{Z}_b(Y;X^*), \ \|\boldsymbol{\zeta} \|_\infty \leq 1, \, \boldsymbol{\zeta}\equiv 0 \text{ outside }U \right\}, \quad U\subset Y \text{ open},
    \end{equation*}
    which can be proven by standard arguments, implies that the closure of a bounded and uniformly tight set is still bounded and uniformly tight.}
\begin{prop}\label{prop: bounded narrow metric}
    Let $D= \{z_n\}_{n\in \N}$ be as in Lemma \ref{lemma: D-cont functions} and define
    \begin{equation*}
        d_{\mathcal{M}(Y;X)}(\nu_1,\nu_2) := \sum_{n\in \N} 2^{-n} d_{\mathcal{M}(Y)}(\pi_{z_n}\nu_1,\pi_{z_n}\nu_2) , \quad \text{for all } \nu_1,\nu_2 \in \mathcal{M}(Y;X).
    \end{equation*}
    {Let $\mathcal{K}\subset \mathcal{M}(Y;X)$ be bounded in total variation and uniformly tight.
    Then, $d_{\mathcal{M}(Y;X)}$ defines a distance over $\mathcal{K}$ and induces its restricted narrow topology.}
\end{prop}

\begin{proof}
    \textbf{Step 1}: $d_{\mathcal{M}(Y;X)}$ is a distance. The positivity, the symmetry and the triangular inequality are trivially inherited from $d_{\mathcal{M}(Y)}$. We are left to show that $d_{\mathcal{M}(Y;X)}(\nu_1,\nu_2) = 0$ implies $\nu_1 = \nu_2$. In this case, we have that $\pi_{z_n}\nu_1 = \pi_{z_n}\nu_2$ for all $n\in \N$, and by density of $\operatorname{Span}(D)$, we have also $\pi_z \nu_1 = \pi_z\nu_2$ for all $z\in X^*$. Then, for all $A\subset Y$ Borel subset it holds $\nu_1(A) = \nu_2(A)$ by definition of $\pi_z$, namely $\nu_1=\nu_2$. 
    
    \textbf{Step 2}: thanks to point $(2)$ of Lemma \ref{lemma: properties of pi_z}, the topology $\tau_d{|_\mathcal{K}}$ induced by $d_{\mathcal{M}(Y;X)}$ over ${\mathcal{K}}$ is larger than its restricted narrow topology, indeed the maps $\pi_{z_n}$ are clearly $(\tau_d{|_\mathcal{K}},\sigma_\mathtt{n})$-continuous. In particular $\tau_\mathtt{n}{|_\mathcal{K}} \subset \tau_{d}{|_\mathcal{K}}$.

    \textbf{Step 3}: we are left to show that $\tau_d {|_\mathcal{K}} \subset \tau_{\mathtt{n}} {|_\mathcal{K}}$. Since the topology induced by a distance is generated by balls, it suffices to show that for any $\tau_d{|_\mathcal{K}}$-open set of the form 
    \[B_\varepsilon(\bar{\nu}):= \{\nu\in {\mathcal{K}} \ : \ d_{\mathcal{M}(Y;X)} (\nu,\bar{\nu}) < \varepsilon\}, \quad\text{for } \bar{\nu} \in {\mathcal{K}}, \ \varepsilon>0,\]
    we can find a $\tau_\mathtt{n}{|_\mathcal{K}}$-open set $O(\bar{\nu})$ {in $\mathcal{K}$} such that $\bar{\nu}\in O(\bar{\nu})$ and $O(\bar{\nu})\subset B(\bar{\nu})$. To this end, consider $\delta = \varepsilon/2$ and $M\in \N$ such that $\sum_{n\geq M+1} 2^{-n}\leq \varepsilon/2 $. To this aim, consider the set 
    \[O(\bar{\nu}) := \left\{ \nu \in {\mathcal{K}} \ : \ d_{\mathcal{M}(Y)}(\pi_{z_i}\nu, \pi_{z_i}\bar{\nu}) < \delta \ \text{for all }i = 1,\dots, M \right\}.\]
    Thanks to point  $(1)$ of Lemma \ref{lemma: properties of pi_z}, we deduce that $O(\bar{\nu})\in\tau_\mathtt{n}{|_\mathcal{K}}$. Clearly $\bar{\nu} \in O(\bar{\nu})$ and thanks to the choices of $\delta$ and $M$, it is not hard to verify that $O(\bar{\nu})\subset B(\bar{\nu})$. 
\end{proof}

\section{Continuity equation and superposition principle}\label{app: CE}
\noindent Let $X$ be a reflexive Banach space. In this appendix we discuss on the continuity equation \eqref{eq:conteq} and on its relation with the key result known as superposition principle.

We start by comparing our Definitions \ref{continuity equation over X def} and \ref{def:CEsol} with more common (but equivalent) definitions present in literature (see e.g. \cite{ambrosio2005gradient}). Indeed, from our point of view a solution to the continuity equation is given by a measure on the space-time $\mu\in \mathcal{M}_+(X_T)$ and, either, a vector-field $v\in L^1_\mu(X_T;X)$ or a vector-valued measure still defined on the space-time $\nu\in \mathcal{M}(X_T;X)$. Instead, in the literature solutions to \eqref{eq:CE} usually involve a narrowly continuous curve of positive measures $(\mu_t)_{t\in[0,T]}\subset \mathcal{M}_+(X)$, as follows.
\begin{df}
    Let $\mu_0 \in \mathcal{M}_+(X_T)$ and $v\in L^1_\mu(X_T;X)$. We say that a narrowly continuous curve of measures $(\mu_t)_{t\in[0,T]}\subset \mathcal{M}_+(X)$ is a solution to the continuity equation \eqref{eq:conteq} with starting measure $\mu_0$  if for all $\xi\in \operatorname{Cyl}_b(X_T)$ it holds
    \begin{equation*}
     \int_0^T \int_X  \partial_t \xi(t,x)+  \langle D_x\xi(t,x), v(t,x)\rangle d\mu_t(x)dt = \int_X \xi(T,x) d\mu_T(x) - \int_X \xi(0,x) d\mu_0(x).
\end{equation*}
\end{df}
Observe that, given $v$ and $(\mu_t)_{t\in[0,T]}$ as in the above definition, it is not difficult to define measures $\mu,\nu,m_0,m$ such that $(\mu,\nu)\in \operatorname{CE}(m_0,m)$ in the sense of Definition \ref{def:CEsol}. Indeed, let us consider:
\begin{equation}\label{eq: decomposition}
    \mu = \int_0^T \delta_t \otimes \mu_t dt, \quad \nu = v\mu, \quad m_0 = \mu_0, \quad m = \mu_T.
\end{equation}
Next lemma  shows that the opposite can be done as well, giving a one-to-one correspondence between the two approaches.

\begin{lemma}\label{lemma CE}
Let $\mu\in \mathcal{M}_+(X_T)$, $m_0,m\in \mathcal{M}_+(X)$ and $v\in L^1_\mu(X_T;X)$. Assume that the continuity equation is satisfied by $(\mu,v)$ in the sense of Definition \ref{continuity equation over X def}, with starting and ending measure $m_0$ and $m$. 

Then, there exists a unique narrowly continuous curve of positive measures $(\mu_t)_{t\in[0,T]}\subset \mathcal{M}_+(X)$, with constant mass $\mu_t(X) = \mu(X_T)/T=m_0(X)$ for all $t\in [0,T]$, such that \eqref{eq: decomposition} holds.
\end{lemma}

\begin{proof}
Uniqueness follows by uniqueness of disintegration of measures, so we focus on the existence. Up to dividing $\mu$, $m_0$ and $m$ by a constant, we may assume that $\mu(X_T) = T$. 
\\
\textbf{Step 1}: we show that the time-marginal of $\mu$, denoted by $L\in \mathcal{M}_+([0,T])$, coincides with the restricted Lebesgue measure $\mathcal{L}^1|_{[0,T]}$. First of all, notice that $m_0(X) = m(X) = 1$: indeed, we can use $\xi_1(t,x)= t$ and $\xi_2(t,x) \equiv 1$ as test functions in the continuity equation, obtaining:
\[\mu(X_T) = T m(X) \quad \text{ and } \quad 0 = m(X) - m_0(X).\]
Then, we test the equation by an arbitrary $\xi \in C_b^1([0,T])$ independent of $x\in X$, deducing
\[\int_0^T \xi'(t)dL(t) = \int_{X_T} \xi'(t) d\mu(t,x) = \xi(T) - \xi(0) = \int_0^T \xi'(t) dt.\]
In particular, for all $g\in C_b(X)$, choosing $\xi(t):= \int_0^t g(s)ds$, we have $\int_0^T g(t) dL(t) = \int_0^T g(t)dt$, which implies that $L = \mathcal{L}^1|_{[0,T]}$, since bounded and continuous functions approximate pointwise any Borel and bounded function from $[0,T]$ to $\R$.
\\
\textbf{Step 2}: using the disintegration theorem, we find a curve of probability measures $(\tilde{\mu}_t)_{t\in[0,T]} \subset \PP(X)$ satisfying $\mu = \int_0^T \delta_t \otimes \delta_{\tilde{\mu}_t} dt$, and in particular
\begin{equation}\label{eq: CE tilde mu}
    \int_0^T \int_X \partial_t \xi(t,x) + \langle D_x\xi (t,x), v(t,x)\rangle d\tilde{\mu}_t(x) dt = \int_0^T \xi(T,x) dm(x) - \int_0^T \xi(0,x) dm_0(x),
\end{equation}
for all $\xi \in \operatorname{Cyl}_b(X_T)$. The rest of the proof is devoted to showing that there exists a continuous representative $(\mu_t)_{t\in[0,T]}$ satisfying $\mu_t = \tilde{\mu}_t$ for a.e. $t\in [0,T]$. Before proceeding, we need the next intermediate step.
\\
\textbf{Step 3}: there exists a countable set $\mathcal{D}\subset \operatorname{Cyl}_c(X)$, i.e. cylinder functions independent of the variable $t\in [0,T]$, such that $\|D_x\xi\|$ is uniformly bounded by $1$ and
\begin{equation*}
    D(\mu,\sigma) := \sup_{\xi\in \mathcal{D}} \int_X \xi d(\mu-\sigma), \quad \text{for all } \mu,\sigma \in \PP(X),
\end{equation*}
is a complete distance on $\PP(X)$ that induces the narrow topology. To this aim, we build a correspondence between $X$ and $\R^\infty$: let $(z_n)_{n\in \N}\subset X^*$ be a dense subset of the unit ball in $X^*$ and define the map 
\begin{equation*}
    \iota: X \to \R^\infty, \quad \iota(x):= \big(\langle z_1,x\rangle, \langle z_2,x\rangle, \dots\big).
\end{equation*}
Let $D_\infty(\underline{x},\underline{y}) := \sup_{n}| x_n - y_n|\wedge 1$ be a complete distance over $\R^\infty$. It is not hard to show that $\iota$ is an isometry between the Polish spaces $(X,\|\cdot\|\wedge 1)$ and $(\iota(X),D_\infty)$. In particular, thanks to \cite[Lemma C.4 \& Remark C.6]{pinzisavare2025}, for all $n\in \N$ we can find a countable set $\mathcal{D}_n$ of functions $\zeta\in C_c^1(\R^n)$ satisfying, for all $\mu,\sigma \in \PP(X)$,
\begin{align*}
    D(\mu,\sigma):= & W_{1,\|\cdot\|\wedge 1}(\mu,\sigma) = W_{1,D_{\infty}}(\iota_\#\mu,\iota_\#\sigma) = \sup_n\sup_{\zeta \in \mathcal{D}_n} \int_{\R^\infty} \zeta(x_1,\dots,x_n) d(\iota_\#\mu-\iota_\#\sigma)(\underline{x})
    \\
    = & 
    \sup_n\sup_{\zeta \in \mathcal{D}_n} \int_{X} \zeta(\langle z_1,x\rangle,\dots,\langle z_n,x\rangle) d(\mu-\sigma)(x)
    = 
    \sup_{\xi \in \mathcal{D}} \int_{X} \xi(x) d(\mu-\sigma)(x),
\end{align*}
where $W_{1,d}$ denotes the $1$-Wasserstein distance with underlying distance $d$, and $\mathcal{D}:= \{\xi\in \operatorname{Cyl}_c(X)  :  \text{there exist } n\in\N,\zeta\in \mathcal{D}_n \text{ such that } \xi(x) = \zeta(\langle z_1,x\rangle,\dots,\langle z_n,x\rangle)\}$. Notice that $W_{1,\|\cdot\|\wedge 1}$ is a complete distance that metrizes the narrow convergence over $\PP(X)$. Moreover, all the functions $\zeta \in \mathcal{D}_n$ satisfy $\sum_{i=1}^n |\partial_i \zeta| \leq 1$, which, together with the fact that $\|z_i\|\leq 1$ for all $i\in \N$, yields that $\|D_x \xi\| \leq 1$ for all $\xi \in \mathcal{D}$.
\\
\textbf{Step 4}: we show that $t\mapsto \tilde{\mu}_t$ is absolutely continuous with respect to the distance $D$. Indeed, thanks to \eqref{eq: CE tilde mu}, for all cylinder functions $\xi \in \mathcal{D}$, the map 
$t\mapsto \int_X \xi d\tilde{\mu}_t$ is in $W^{1,1}(0,T)$, so that there exists a full Lebesgue measure set $A_{\xi} \subset(0,T)$ such that 
\[\int_X \xi d(\tilde{\mu}_t - \tilde{\mu}_s) = \int_s^t \int_X \langle D \xi(x),v(\tau,x)\rangle d\tilde{\mu}_\tau(x) d\tau,\]
for all $s,t \in A_{\xi}$, with $s<t$. Consider the full Lebesgue measure set $A := \bigcup_{\xi \in \mathcal{D}} A_\xi$, to deduce that for all $s,t\in A$ with $s<t$ it holds
\begin{align*}
    D(\tilde{\mu}_s,\tilde{\mu}_t) = \sup_{\xi \in \mathcal{D}} \int \xi d(\tilde{\mu}_t - \tilde{\mu}_s) = \sup_{x\in\mathcal{D}} \int_s^t \int_X \langle D \xi(x),v(\tau,x)\rangle d\tilde{\mu}_\tau(x) d\tau \leq \int_s^t \|v(\tau,x)\| d\tilde{\mu}_\tau(x) d\tau. 
\end{align*}
Since $\int_0^T \int_X \|v(\tau,x)\| d\tilde{\mu}_\tau(x) d\tau = \int_{X_T} \|v\|d\mu <+\infty$, this shows that the map $t\mapsto \tilde{\mu}_t$ has a $D$-absolutely continuous representative, which concludes the proof.
\end{proof}

We finally state the version of the superposition principle valid in general Banach space, which we used in the paper. It can be found in \cite[Theorem 5.2]{ambrosio2021spatially}.

\begin{teorema}[\textbf{Superposition principle on Banach spaces}]\label{thm: superposition Banach}
    Let $\mu\in \mathcal{M}_+(X_T)$, $m_0,m\in \mathcal{M}_+(X)$ and $v\in L^1_\mu(X_T;X)$. Assume that the continuity equation is satisfied by the pair $(\mu,v)$ in the sense of Definition \ref{continuity equation over X def}, with starting and ending measure $m_0$ and $m$. 
    
    Then, there exists a measure $\lambda \in \mathcal{M}_+(C([0,T];X))$ concentrated over absolutely continuous curves that solve the ordinary differential equation
    \begin{equation}\label{eq:odeapp}
        \dot\xx(t) = v(t,\xx(t)), \quad \text{for a.e. } t\in(0,T),
    \end{equation}
    satisfying $(e_0)_\#\lambda = m_0$, $(e_T)_\#\lambda = m$ and $\mathrm{e}_\#(\mathcal{L}^1|_{[0,T]}\otimes\lambda) = \mu$, where 
    \(\mathrm{e}(t,\xx):= (t,\xx(t)).\) In particular, if $m_0\in \PP(X)$, then $\lambda \in \PP(C([0,T],X))$.
\end{teorema}

\begin{proof}
    Thanks to Lemma \ref{lemma CE}, there exists a curve of narrowly continuous measures $t\mapsto \mu_t$ with constant total mass, that is $\mu_t(X) = \mu(X_T)/T=:\alpha$ for all $t\in[0,T]$, and such that $\mu = \int_0^T \delta_t\otimes \mu_t dt$, $\mu_0 = m_0$ and $\mu_T = m$. Then, applying \cite[Theorem 5.2]{ambrosio2021spatially} to the curve $t\mapsto \mu_t/\alpha$, we find a probability measure $\Tilde{\lambda} \in \PP(C([0,T],X))$ satisfying $(e_t)_\#\tilde\lambda = \mu_t/\alpha$ and that is concentrated over curves $\xx$ that are absolutely continuous and solve \eqref{eq:odeapp}. Now, the measure $\lambda:= \alpha \tilde\lambda$ satisfies all the requirements.
\end{proof}

\section{On the measurability of Bochner integration and limiting subdifferentials}

This final appendix contains some technical results regarding mesurability properties of the Bochner integral and of the limiting subdifferential. From now on, let $X$ be a Banach space with separable dual $X^*$. We start with two general lemmas.

\begin{lemma}\label{lemma: SW Borel}
    The Borel $\sigma$-algebras $\mathcal{B}(X^*,\tau_s^*)$ and $\mathcal{B}(X^*,\tau_w^*)$ coincide, where $\tau_s^*$ and $\tau_w^*$ denote, respectively, the strong and the $\text{weak}^*$ topology of $X^*$.
\end{lemma}
\begin{proof}
    It follows from the fact that $\tau_w^* \subset \tau_s$ and that both the topological spaces $(X^*,\tau_s^*)$ and $(X^*,\tau_w^*)$ are Lusin spaces. Indeed, the first is Polish by assumption, while the second can be written as 
    \[X^* = \bigcup_{n\geq 1} B_n^*,\]
    and thanks to \cite[Theorems 3.26 and 3.28]{brezis2011functional}, the subsets $B_n^*$ endowed with the $\text{weak}^*$-subspace topology, are Polish and a union of Polish subsets is Lusin. Then we conclude thanks to \cite[Corollary 2, pp. 101]{schwariz1973radon}.
\end{proof}

\begin{lemma}\label{lemma: meas}
    Let $(S,\mathcal{S})$ be a measurable space and $G:S \to X^*$. The map $G$ is Borel measurable if and only if for all $x\in X$ the map $G_x:S \to \R$, defined as $G_x(s):= \langle G(s), x\rangle$, is measurable.
\end{lemma}
\begin{proof}
    The necessity is trivial. On the other hand, a basis for the weak topology is given by sets of the form 
    \begin{equation}\label{eq: gen top}
    \{z\in X^* \ : \ \langle z, x_i \rangle \in (a_i,b_i) \quad \text{for all } i = 1,\dots,n\},
    \end{equation}
    for some $n\in \N$, $x_i \in X$, and $-\infty \leq a_i < b_i \leq +\infty$. By assumption, the preimage of these sets through $G$ is measurable. Since $(X^*,\tau_w^*)$ is separable, it is enough to check that $G^{-1}(A) \in \mathcal{S}$ for all $A$ in a countable generator for the topology, given by taking sets of the form \eqref{eq: gen top} with $x_i$ in a dense set for $X$, $a_i,b_i \in \mathbb{Q}$.
\end{proof}

Given a Borel measurable function $f:Y\to X^*$, we set $$\mathcal{M}_{+,f}(Y) := \{\mu\in \mathcal{M}_+(Y) : \int_Y \|f(y)\|_{X^*}d\mu(y) <+\infty\}.$$
Thanks to \cite[Lemma D.1]{pinzisavare2025}, the set $\mathcal{M}_{+,f}(Y)$ is Borel as a subset of $\mathcal{M}_+(Y)$, endowed with the narrow topology.

\begin{prop}\label{lemma: bochner measurability}
    The Bochner integral map
    \begin{equation*}
         \mathcal{M}_{+,f}(Y) \ni \mu \mapsto \int_Y f(y) d\mu(y)\in X^*,  
    \end{equation*}
    is Borel measurable, where the integral is intended in the Bochner sense.
\end{prop}
\begin{proof}
    Due to Lemma \ref{lemma: meas}, it is enough to check that $\mu \mapsto \int \langle f,x \rangle d\mu$ is measurable for any $x\in X$. We decompose it as $\int\langle f,x \rangle ^+d\mu - \int \langle f,x \rangle^- d\mu$, and both terms are measurable thanks to \cite[Lemma D.1]{pinzisavare2025}.
\end{proof}

\begin{prop}\label{prop: meas lim}
    Let $\Phi:X \to (-\infty,+\infty]$. The graph of its limiting subdifferential $\partial_\ell \Phi$ (see Definition \ref{def:limiting-subdifferential}), defined as
    \[\partial_\ell \Phi :=\left\{(x,z) \in X\times X^* \, : \, z\in \partial_\ell \Phi(x)\right\},\]
    is a Borel subset of $X\times X^*$, endowed with the product topology. 
\end{prop}
\begin{proof}
    By the very definition of limiting subdifferential, its graph is $\text{strongly-weakly}^*$ sequentially closed. Then, recalling that $B_R^*\subset X^*$ is $\text{weakly}^*$ closed and that the $\text{weak}^*$ topology is metrizable therein, the set 
    \(\partial_\ell \Phi \cap (X\times B_R^*)\)
    is $\text{strongly-weakly}^*$ closed. In particular, thanks to Lemma \ref{lemma: SW Borel}, the Borel sets of $X \times X^*$ generated by the strong-strong topology coincide with the ones generated by the $\text{strong-weak}^*$ topology. Since $\partial_\ell \Phi$ can be written as countable union of $\text{strongly-weakly}^*$ closed sets, it is Borel. 
\end{proof}

\bigskip

	\noindent\textbf{Data availability statement.} The manuscript has no associated data.

\bigskip

	\noindent\textbf{Conflict of interest statement.} The authors have no competing interests to declare that are relevant to the content of this article.

    \bigskip
	
	\noindent\textbf{Acknowledgements.}
    The authors are members of GNAMPA (INdAM). F. R.\ has been partially supported by the INdAM-GNAMPA project 2025 \lq\lq DISCOVERIES\rq\rq (CUP E5324001950001). G.S.~and F.R.\ have been partially supported by the MIUR-PRIN 202244A7YL project {\em Gradient Flows and Non-Smooth Geometric Structures with Applications to Optimization and Machine Learning}. F.R. also acknowledges the support of BIDSA of Bocconi University.
    G.S. has been partially supported by the INDAM project E53C23001740001, and by funding from the European Research Council (ERC) under the European Union’s Horizon Europe research and innovation programme (grant agreement No. 101200514, project acronym OPTiMiSE).
 Views and opinions expressed are however those of the author(s) only and do not necessarily reflect those of the European Union or the European Research Council Executive Agency. Neither the European Union nor the granting authority can be held responsible for them. 
	\bigskip

\bibliographystyle{plaine}
\bibliography{bibliography}

	{\small
		
		\vspace{15pt} (Alessandro Pinzi) Bocconi University, Department of Decision Sciences, \par
		\textsc{via Roentgen 1, 20136 Milano, Italy}
		\par
		\textit{e-mail address}: \textsf{alessandro.pinzi@phd.unibocconi.it}
		\par
		\textit{Orcid}: \textsf{https://orcid.org/0009-0007-9146-5434}	
        
		\vspace{5pt} (Filippo Riva) Charles University, Faculty of Mathematics and Physics, Department of Mathematical Analysis, 
		\textsc{Sokolovsk\'a 49/83, 186 75 Prague 8, Czechia}
		\par
		\textit{e-mail address}: \textsf{filippo.riva@matfyz.cuni.cz}
		\par
		\textit{Orcid}: \textsf{https://orcid.org/0000-0002-7855-1262}
	
        \vspace{5pt} (Giuseppe Savar\'e) Bocconi University, Department of Decision Sciences and BIDSA \par
		\textsc{via Roentgen 1, 20136 Milano, Italy}
		\par
		\textit{e-mail address}: \textsf{giuseppe.savare@unibocconi.it}
		\par
		\textit{Orcid}: \textsf{https://orcid.org/0000-0002-0104-4158}
		\par
		
	}

\end{document}